\documentclass[pdflatex]{amsart}
\usepackage{amsmath,amscd,amssymb}
\usepackage{bbm}
\usepackage{mathrsfs}
\usepackage{mathtools}
\usepackage{ifthen}
\usepackage{upgreek}
\usepackage{stmaryrd}
\usepackage{afterpage}
\usepackage[all]{xy}

\title[The FRT construction over non-commutative algebras]
{Tannaka theory and the FRT construction over non-commutative algebras}
\author[K.~Shimizu]{Kenichi Shimizu}
\email{kshimizu@shibaura-it.ac.jp}
\address{Department of Mathematical Sciences \\
  Shibaura Institute of Technology \\
  307 Fukasaku, Minuma-ku, Saitama-shi, Saitama 337-8570, Japan.}
\date{}
\usepackage{amsthm}
\numberwithin{equation}{section}
\newtheorem{counter}{}[section]
\theoremstyle{definition}
\newtheorem{definition}         [counter]{Definition}

\theoremstyle{plain}
\newtheorem{lemma}              [counter]{Lemma}

\newtheorem{proposition}        [counter]{Proposition}
\newtheorem{theorem}            [counter]{Theorem}

\newtheorem*{theorem*}          {Theorem}
\theoremstyle{remark}
\newtheorem{remark}             [counter]{Remark}
\newtheorem{example}            [counter]{Example}
\usepackage[bookmarksopen,bookmarksdepth=3]{hyperref}
\newcommand{\PICDIR}{./}
\newcommand{\PICSCALE}{1}
\newcommand{\PIC}[2][\PICSCALE]{%
  \vcenter{\hbox{\includegraphics[scale=#1]{\PICDIR/#2.pdf}}}}
\newcommand{\id}{\mathrm{id}}
\newcommand{\eval}{\mathrm{eval}}
\newcommand{\coev}{\mathrm{coev}}
\newcommand{\op}{\mathrm{op}}
\newcommand{\env}{\mathrm{e}}
\newcommand{\rev}{\mathrm{rev}}
\newcommand{\unitobj}{\mathbbm{1}}
\newcommand{\Hom}{\mathrm{Hom}}

\newcommand{\Nat}{\mathrm{Nat}}
\mathchardef\mhyphen="2D
\newcommand{\lmod}[1]{#1\mhyphen\mathrm{mod}}
\newcommand{\lcomod}[1]{#1\mhyphen\mathrm{comod}}
\newcommand{\rcomod}[1]{\mathrm{comod}\mhyphen#1}
\newcommand{\lZlax}{\mathscr{W}_{\ell}}
\newcommand{\rZlax}{\mathscr{W}_r}
\newcommand{\src}{\mathfrak{s}}
\newcommand{\tgt}{\mathfrak{t}}
\newcommand{\Mod}{\mathfrak{M}}
\newcommand{\BiMon}{\mathrm{Bimon}}
\newcommand{\BiAlg}{\mathrm{Bialg}}
\newcommand{\Obj}{\mathrm{Obj}}
\newcommand{\Mor}{\mathrm{Mor}}
\newcommand{\Img}{\mathrm{Im}}
\newcommand{\Ker}{\mathrm{Ker}}
\newcommand{\Rel}{\mathrm{Rel}}
\newcommand{\rig}{\mathrm{rig}}
\newcommand{\xE}[3]{{#1\big[#2; #3\big]}}
\newcommand{\wE}[2]{\xE{\widetilde{\mathbf{e}}}{#1}{#2}}
\newcommand{\eE}[2]{\xE{\mathbf{e}}{#1}{#2}}
\newcommand{\bE}[2]{\xE{\overline{\mathbf{e}}}{#1}{#2}}
\begin{document}

\begin{abstract}
  Let $A$ be an algebra over a commutative ring $k$. We introduce the notion of a coquasitriangular left bialgebroid over $A$ and show that the category of left comodules over such a bialgebroid has a braiding. 
  We also investigate a Tannaka type construction of bimonads and bialgebroids.
  As an application, the Faddeev-Reshetikhin-Takhtajan (FRT) construction over the algebra $A$ is established.
  Our construction associates a coquasitriangular bialgebroid to a braided object $(M, c)$ in the category of $A$-bimodules such that $M$ is finitely generated and projective as a left $A$-module. A Hopf algebroid version of this construction is also provided.
\end{abstract}
 
\maketitle
\tableofcontents

\section{Introduction}

An invertible linear map $c: M \otimes M \to M \otimes M$ on the two-fold tensor product of a finite-dimensional vector space $M$ is called a {\em Yang-Baxter operator} if it satisfies
\begin{equation}
  \label{eq:Intro-YBE}
  (c \otimes \id_M) \circ (\id_M \otimes c) \circ (c \otimes \id_M) = (\id_M \otimes c) \circ (c \otimes \id_M) \circ (\id_M \otimes c)
\end{equation}
on $M \otimes M \otimes M$. This equation, called the Yang-Baxter equation (YBE), originally appeared in the context of mathematical physics. The study of the YBE gave birth to quantized enveloping algebras $U_q(\mathfrak{g})$. Roughly speaking, we can obtain a solution of the YBE from each finite-dimensional representation of $U_q(\mathfrak{g})$.

The Faddeev-Reshetikhin-Takhtajan (FRT) construction, introduced in \cite{MR1015339}, may be considered as an inverse procedure of the above method of obtaining a solution of the YBE. The FRT construction takes a solution $c: M \otimes M \to M \otimes M$ of the YBE as an input and yields a bialgebra or a Hopf algebra as an output. The resulting bialgebra, say $B$, has the remarkable property that the category of $B$-comodules is braided. Furthermore, the vector space $M$ has a canonical $B$-comodule structure and the original solution $c$ can be recovered as the braiding at $M$.

Several generalizations of the YBE have been introduced and studied; see \cite{MR1722951,MR1769723,MR2142557,MR2183847,MR2742743} and references therein. In the language of category theory, most of such generalizations can be defined as a braided object in a suitable monoidal category. Here, a braided object in a monoidal category $\mathcal{V}$ is a pair $(M, c)$ consisting of an object $M$ of $\mathcal{V}$ and an isomorphism $c: M \otimes M \to M \otimes M$ in $\mathcal{V}$ satisfying \eqref{eq:Intro-YBE}, but now the symbols $\otimes$ and $\circ$ mean the tensor product and the composition in $\mathcal{V}$, respectively.

With the progress of the study of such generalizations of the YBE, the FRT construction has also been generalized in various settings; see, {\it e.g.}, \cite{MR1313748,MR1645196,MR1623965,MR1722951,MR2096667,MR2588133,MR3448180}.
Let $A$ be an algebra over a commutative ring $k$, and let ${}_A \Mod_A$ be the monoidal category of $A$-bimodules. One of contributions of this paper is a generalization of the FRT construction for a braided object $(M, c)$ in ${}_A \Mod_A$ such that $M$ is finitely generated and projective as a left $A$-module.

Given such a pair $(M, c)$, one can construct a bialgebroid by existing frameworks as follows:
Let $\mathfrak{B}$ be the monoidal category of braids. There is a strong monoidal functor $\omega: \mathfrak{B} \to {}_A \Mod_A$ assigning $c$ to a crossing of two strands. Although $\mathfrak{B}$ is not even an additive category, a construction of bialgebroids studied in \cite{MR2448024,MR1984397,2009arXiv0907.1578S} can be applied to the functor $\omega$.
The bialgebroid obtained in this way may be called the `FRT bialgebroid' associated to $(M, c)$ since, in the case where $A = k$, the story so far reduces to the Tannaka theoretic interpretation of the FRT construction explained in \cite{MR1623637,MR1623636}.

However, the above method does not extend an important aspect of the original FRT construction.
The bialgebra constructed by the original FRT construction has a universal R-form and therefore the category of its comodules is braided, while it is not yet known whether the same holds for the FRT bialgebroid constructed in the above.

In this paper, mentioning this question, we introduce a general method for constructing bimonads ($=$ comonoidal monads) and study their properties.
Our main result is explained as follows: Let $\mathcal{V}$ be a monoidal category with cocontinuous tensor product. A {\em construction data} over $\mathcal{V}$ (Definition~\ref{def:construction-data}) is a pair $(\mathcal{D}, \omega)$ consisting of an essentially small monoidal category $\mathcal{D}$ and a strong monoidal functor $\omega: \mathcal{D} \to \mathcal{V}$ such that $\omega(x)$ admits a right dual object for all $x \in \mathcal{D}$ and the coend
\begin{equation}
  \label{eq:Intro-bimonad-T}
  T(M) = \int^{x \in \mathcal{D}} \omega(x) \otimes M \otimes \omega(x)^{\vee}
\end{equation}
exists for all $M \in \mathcal{V}$, where $(-)^{\vee}$ denotes the right dual (see Subsection~\ref{subsec:convention-duals} for our convention on duals). We extend the assignment $M \mapsto T(M)$ to an endofunctor on $\mathcal{V}$ and equip $T$ with a structure of a bimonad on $\mathcal{V}$ in a similar way as we do in the Tannaka theory (see Subsection~\ref{subsec:bimonad-intro}). Our main result (Theorem~\ref{thm:FRT-bimonad}) states that the bimonad $T$ enjoys the following properties:
\begin{enumerate}
\item The category $\lmod{T}$ of $T$-modules is isomorphic to the lax right centralizer of $\omega$ (see Subsection~\ref{subsec:lax-centralizers} for the definition of left/right lax centralizer).
\item The category $\lcomod{T}$ is isomorphic to the left lax centralizer of the forgetful functor from $\lmod{T}$ to $\mathcal{V}$ (here $\lcomod{T}$ is the category of $T$-comodules, which corresponds to the category of comodules over a bialgebroid; see Definition~\ref{def:comodule-over-bimonad}).
\item If $\mathcal{V}$ is complete, then $T$ admits a right adjoint.
\item If $\mathcal{V}$ is has a lax braiding ({\it i.e.}, a not necessarily invertible braiding), then the bimonad $T$ is implemented by a bialgebra object in $\mathcal{V}$.
\item If $\mathcal{D}$ is right rigid, then $T$ is a left Hopf monad in the sense of \cite{MR2793022}.
\item If $\mathcal{D}$ is left rigid, then $T$ is a right Hopf monad in the sense of \cite{MR2793022}.
\item If $\mathcal{D}$ has a lax braiding, then so does $\lcomod{T}$.
\end{enumerate}

For bialgebroids, this theorem works as follows:
We fix an algebra $A$ and consider the case where $\mathcal{V} = {}_A \Mod_A$.
Let $\mathcal{D}$ be a small monoidal category, and let $\omega : \mathcal{D} \to \mathcal{V}$ be a strong monoidal functor such that $\omega(x)$ is right rigid for all objects $x \in \mathcal{D}$ (note that an object of ${}_A \Mod_A$ is right rigid if and only if it is finitely generated and projective as a left $A$-module).
Since $\mathcal{V}$ is cocomplete, the coend \eqref{eq:Intro-bimonad-T} exists and thus we have the associated bimonad $T$ on $\mathcal{V}$.
According to Szlach\'anyi \cite{MR1984397}, a bialgebroid over $A$ can be identified with a bimonad on ${}_A \Mod_A$ admitting a right adjoint.
Thus, by Part (3), there is a left bialgebroid ${B}_{\omega}$ over $A$ associated to the bimonad $T$.
Part (5) implies that ${B}_{\omega}$ is a left Hopf algebroid (in the sense of Schauenburg \cite{MR1800718}) if $\mathcal{D}$ is right rigid.

In this paper, we also introduce the notion of a (lax) universal R-form on a left bialgebroid (Definition~\ref{def:univ-R}) and show that the category of comodules over a bialgebroid $B$ has a (lax) braiding if $B$ has a (lax) universal R-form (Theorem~\ref{thm:univ-R-lax-braiding}).
Now we suppose that the monoidal category $\mathcal{D}$ has a lax braiding $\sigma$.
Then, by Part (6), the category of left $B_{\omega}$-comodules has a lax braiding, say $\widetilde{\sigma}$.
It turns out that $\widetilde{\sigma}$ arises from a lax universal R-form on ${B}_{\omega}$.
Furthermore, if $\sigma$ is invertible, then $\widetilde{\sigma}$ is also invertible (Theorem~\ref{thm:B-omega-lax-universal-R}).

The FRT bialgebroid is a special case where $\mathcal{D} = \mathfrak{B}$ is the category of braids and, in particular, has a universal R-form. If we replace $\mathfrak{B}$ with the `rigid extension' of $\mathfrak{B}$, we also obtain a Hopf algebroid version of the FRT construction.

The bialgebra constructed by the original FRT construction has the generators $T_i^j$ subject to certain relations expressed in terms of a matrix presentation of the solution of the YBE and its coalgebra structure is given by $\Delta(T_{i}^{j}) = \sum_r T_{i}^r \otimes T_{r}^j$ on the generators.
In this paper, we also provide a general method to give a presentation of the left bialgebroid ${B}_{\omega}$ from a presentation of the monoidal category $\mathcal{D}$ (Theorem~\ref{thm:B-omega-gen-rel}). This yields a presentation of the FRT bialgebroid and its Hopf version (Theorems~\ref{thm:FRT-coqt-bialgebroid} and \ref{thm:FRT-coqt-Hopf}).

We remark that a similar construction has been given by Rosengren \cite{MR2096667} in the setting of $\mathfrak{h}$-bialgebroids. According to \cite{2020arXiv200809081K}, $\mathfrak{h}$-bialgebroids might be understood as monoidal comonads on the category of Harish-Chandra bimodules and, in particular, they are different objects than ordinary bialgebroids. It could be interesting to examine how our construction (or its dual version) relates to some constructions of dynamical quantum groups.

\subsection*{Organization of this paper}

This paper is organized as follows: In Section \ref{sec:preliminaries}, we recall basic results and terminology on monoidal categories, bimonads and Hopf monads from \cite{MR1712872,MR3242743,MR2355605,MR2793022}.

In Section~\ref{sec:monadic-FRT}, we introduce our main construction for a bimonad and investigate its properties. The main result of this section was introduced in the above: Starting from a construction data $(\mathcal{D}, \omega)$ over $\mathcal{V}$, we construct a bimonad $T$ on $\mathcal{V}$. Properties of the bimonad $T$ are summarized in Subsection~\ref{subsec:bimonad-intro}.

In Section~\ref{sec:bialgebroids}, we recall basic results on bialgebroids. Let $A$ be an algebra over a commutative ring $k$. Given a bialgebroid $B$ over $A$, we denote by ${}_B \Mod$ and ${}^B \Mod$ the category of left $B$-modules and the category of left $B$-comodules, respectively. It is known that both ${}_B \Mod$ and ${}^B\Mod$ are $k$-linear monoidal categories admitting the `fiber' functor to ${}_A \Mod_A$. The reason why ${}_B\Mod$ is monoidal is that $B$ defines a bimonad $T = B \otimes_{A^{\env}}(-)$ on ${}_A \Mod_A$. For this bimonad $T$, we explain that the category of $T$-comodules in the sense of Definition~\ref{def:comodule-over-bimonad} is isomorphic to ${}^B\Mod$ as a monoidal category (Theorem~\ref{thm:T-comod-is-B-comod}). In this section, we also introduce the notion of a lax universal R-form of a bialgebroid and show that ${}^B\Mod$ is lax braided if $B$ possesses a lax universal R-form (Theorem~\ref{thm:univ-R-lax-braiding}).

Now let $(\mathcal{D}, \omega)$ be a construction data over ${}_A \Mod_A$. In Section~\ref{sec:construction-bialgebroids}, we introduce a left bialgebroid ${B}_{\omega}$ over $A$ such that the bimonad on ${}_A \Mod_A$ constructed from the data $(\mathcal{D}, \omega)$ is isomorphic to the bimonad ${B}_{\omega} \otimes_{A^{\env}}(-)$. The bialgebroid $B_{\omega}$ is constructed as follows: We first consider the $A^{\env}$-bimodule
\begin{equation*}
  \widetilde{{B}}_{\omega} = \bigoplus_{x \in \Obj(\mathcal{D})} \omega(x) \otimes_k \omega(x)^{\vee}
\end{equation*}
and define a $k$-linear map $\Rel_f: \omega(x) \otimes \omega(y)^{\vee} \to \widetilde{{B}}_{\omega}$ for each morphism $f : x \to y$ in $\mathcal{D}$ in a certain way. By realizing the coend \eqref{eq:Intro-bimonad-T} as a certain form of a coequalizer, we see that the $A^{\env}$-bimodule ${B}_{\omega}$ is the quotient of $\widetilde{B}_{\omega}$ by $J_{\omega} = \sum_{f \in \Mor(\mathcal{D})} \Img(\Rel_f)$ (see Lemma~\ref{lem:construction-coend}). By using this realization, we describe the left bialgebroid structure of ${B}_{\omega}$, its lax universal R-form (provided that $\mathcal{D}$ is braided) and the inverse of the Galois map (provided that $\mathcal{D}$ is right rigid).

In Section~\ref{sec:FRT-over-non-comm-rings}, we first recall the precise meaning of the monoidal category presented by `generators and relations.' We then give a presentation of the bialgebroid ${B}_{\omega}$ provided that a presentation of $\mathcal{D}$ is given (Theorem~\ref{thm:B-omega-gen-rel}). As an application of the result, we establish the FRT construction for bialgebroids over $A$ (Theorem~\ref{thm:FRT-coqt-bialgebroid}) and its Hopf algebroid version (Theorem~\ref{thm:FRT-coqt-Hopf}).
Finally, we explain that our construction generalizes Hayashi's construction of face algebras \cite{MR1623965}.

\subsection*{Acknowledgment}

This work is supported by JSPS KAKENHI Grant Numbers JP16K17568 and JP20K03520.

\section{Preliminaries: Monoidal categories and bimonads}
\label{sec:preliminaries}

\subsection{Monoidal categories}
\label{subsec:mon-cat}
Given a category $\mathcal{C}$, we denote by $\Obj(\mathcal{C})$ and $\Mor(\mathcal{C})$ the class of objects of $\mathcal{C}$ and the class of morphisms of $\mathcal{C}$, respectively. We usually write $X \in \mathcal{C}$ to mean that $X$ belongs to $\Obj(\mathcal{C})$.

A {\em monoidal category} \cite[VII.1]{MR1712872} is a category $\mathcal{C}$ equipped with a functor $\otimes: \mathcal{C} \times \mathcal{C} \to \mathcal{C}$ (called the tensor product), an object $\unitobj \in \mathcal{C}$ (called the unit object) and natural isomorphisms $a_{X,Y,Z}: (X \otimes Y) \otimes Z \to X \otimes (Y \otimes Z)$, $l_X: \unitobj \otimes X \to X$ and $r_X: X \otimes \unitobj \to X$ for $X, Y, Z \in \mathcal{C}$ satisfying the pentagon and the triangle axioms. A monoidal category $\mathcal{C}$ is said to be {\em strict} if the natural isomorphisms $a$, $l$ and $r$ are the identities. In view of the Mac Lane coherence theorem, we may assume that every monoidal category is strict.

Now let $\mathcal{C}$ and $\mathcal{D}$ be monoidal categories.
A {\em monoidal functor} \cite[XI.2]{MR1712872} from $\mathcal{C}$ to $\mathcal{D}$ is a functor $F : \mathcal{C} \to \mathcal{D}$ equipped with a natural transformation
\begin{equation*}
  m_{X,Y} : F(X) \otimes F(Y) \to F(X \otimes Y)
  \quad (X, Y \in \mathcal{C})
\end{equation*}
and a morphism $u : \unitobj_{\mathcal{D}} \to F(\unitobj_{\mathcal{C}})$ in $\mathcal{D}$ such that the equations
\begin{gather*}
  m_{X, Y \otimes Z} \circ (\id_{F(X)} \otimes m_{Y, Z})
  = m_{X \otimes Y, Z} \circ (m_{X, Y} \otimes \id_{F(Z)}), \\
  m_{X, \unitobj} \circ (\id_{F(X)} \otimes u)
  = \id_{F(X)}
  = m_{\unitobj, X} \circ (u \otimes \id_{F(X)})
\end{gather*}
hold for all objects $X, Y, Z \in \mathcal{C}$. We say that a monoidal functor is {\em strong} ({\it resp.} {\em strict}) if its structure morphisms are invertible ({\it resp}. identities).

The dual notion is also used in this paper: A {\em comonoidal functor} from $\mathcal{C}$ to $\mathcal{D}$ is a functor $F : \mathcal{C} \to \mathcal{D}$ equipped with a natural transformation
\begin{equation*}
  \delta_{X,Y} : F(X \otimes Y) \to F(X) \otimes F(Y)
  \quad (X, Y \in \mathcal{C})
\end{equation*}
and a morphism $\varepsilon : F(\unitobj_{\mathcal{C}}) \to \unitobj_{\mathcal{D}}$ in $\mathcal{D}$ making  $F^{\op}: \mathcal{C}^{\op} \to \mathcal{D}^{\op}$ a monoidal functor.

\subsection{Convention for duals}
\label{subsec:convention-duals}

For basic terminology on duality in a monoidal category, we refer the reader to \cite[Section 2.10]{MR3242743}.
Let $\mathcal{C}$ be a monoidal category, and let $\mathcal{C}_{\rig}$ be the full subcategory of $\mathcal{C}$ consisting of right rigid objects. We assume that a right dual object has been chosen for each object of $\mathcal{C}_{\rig}$. Unlike \cite{MR3242743}, we denote by
$(X^{\vee}, \eval_X: X \otimes X^{\vee} \to \unitobj, \coev_{X}: \unitobj \to X^{\vee} \otimes X)$
the fixed right dual object of $X \in \mathcal{C}_{\rig}$. The assignment $X \mapsto X^{\vee}$ extends to a contravariant functor from $\mathcal{C}_{\rig}$ to $\mathcal{C}$. Furthermore, there are canonical isomorphisms
\begin{equation}
  \label{eq:right-dual-monoidal}
  \vartheta_{0}: \unitobj \to \unitobj^{\vee}
  \quad \text{and} \quad
  \vartheta^{(2)}_{X,Y}: Y^{\vee} \otimes X^{\vee} \to (X \otimes Y)^{\vee}
  \quad (X, Y \in \mathcal{C}_{\rig})
\end{equation}
making the functor $X \mapsto X^{\vee}$ a strong monoidal functor from $(\mathcal{C}_{\rig})^{\op}$ to $\mathcal{C}^{\rev}$, which we call the {\em right duality functor}. Here, $\mathcal{C}^{\rev}$ is the monoidal category obtained from $\mathcal{C}$ by reversing the order of the tensor product.

For reader's convenience, we note how the morphisms $\vartheta_{0}$ and $\vartheta^{(2)}_{X,Y}$ are defined:
Under our strictness assumption on $\mathcal{C}$, the isomorphism $\vartheta_{0}$ and its inverse are given by $\vartheta_{0} = \coev_{\unitobj}$ and $\vartheta_{0}^{-1} = \eval_{\unitobj}$, respectively. For $X, Y \in \mathcal{C}_{\rig}$, we define
\begin{gather}
  \label{eq:rigid-mon-cat-eval-2}
  \eval_{X,Y}^{(2)} = \eval_{X} \circ (\id_{X} \otimes \eval_{Y} \otimes \id_{X^{\vee}})
  : X \otimes Y \otimes Y^{\vee} \otimes X^{\vee} \to \unitobj, \\
  \label{eq:rigid-mon-cat-coev-2}
  \coev_{X,Y}^{(2)} = (\id_{Y^{\vee}} \otimes \coev_X \otimes \id_{Y}) \circ \coev_{Y}
  : \unitobj \to Y^{\vee} \otimes X^{\vee} \otimes X \otimes Y.
\end{gather}
Then the isomorphism $\vartheta^{(2)}_{X,Y}$ and its inverse are given as follows:
\begin{gather*}
  \vartheta^{(2)}_{X,Y}
  = (\id_{(X \otimes Y)^{\vee}} \otimes \eval^{(2)}_{X,Y})
  \circ (\coev_{X \otimes Y} \otimes \id_{Y^{\vee}} \otimes \id_{X^{\vee}}), \\
  (\vartheta^{(2)}_{X,Y})^{-1}
  = (\id_{Y^{\vee}} \otimes \id_{X^{\vee}} \otimes \eval_{X \otimes Y})
  \circ (\coev^{(2)}_{X,Y} \otimes \id_{(X \otimes Y)^{\vee}}).
\end{gather*}

\subsection{Lax centralizers and their variants}
\label{subsec:lax-centralizers}

Let $\mathcal{B}$ and $\mathcal{C}$ be monoidal categories. Given two comonoidal functors $F = (F, \delta, \varepsilon)$ and $G = (G, \delta', \varepsilon')$ from $\mathcal{B}$ to $\mathcal{C}$, we define the category $\lZlax(F, G)$ as follows:

\begin{definition}
  \label{def:left-Z-lax-F-G}
  An object of $\lZlax(F, G)$ is a pair $(M, \rho_M)$ consisting of an object $M \in \mathcal{C}$ and a natural transformation
  \begin{equation*}
    \rho_M(X): M \otimes F(X) \to G(X) \otimes M \quad (X \in \mathcal{B})
  \end{equation*}
  satisfying the following two conditions (Z1) and (Z2).
  \begin{enumerate}
  \item [(Z1)] For all objects $X, Y \in \mathcal{B}$, the following diagram commutes:
    \begin{equation*}
      \xymatrix@C=96pt@R=16pt{
        M \otimes F(X \otimes Y)
        \ar[dd]_{\rho_M(X \otimes Y)}
        \ar[r]^(.45){\id_M \otimes \delta_{X, Y}}
        & M \otimes F(X) \otimes F(Y)
        \ar[d]^{\ \rho_M(X) \otimes \id_{F(Y)}} \\
        & G(X) \otimes M \otimes F(Y)
        \ar[d]^{\ \id_{G(X)} \otimes \rho_M(Y)} \\
        G(X \otimes Y) \otimes M
        \ar[r]_(.45){\delta'_{X, Y} \otimes \id_M}
        & G(X) \otimes G(Y) \otimes M \\
      }
    \end{equation*}
  \item [(Z2)] The equation $\id_M \otimes \varepsilon = (\varepsilon' \otimes \id_M) \circ \rho_M(\unitobj)$ holds.
  \end{enumerate}
  Given two objects $\mathbf{M} = (M, \rho_M)$ and $\mathbf{N} = (N, \rho_N)$ of $\lZlax(F, G)$, a {\em morphism} from $\mathbf{M}$ to $\mathbf{N}$ is a morphism $f: M \to N$ in $\mathcal{C}$ satisfying the equation
  \begin{equation*}
    (\id_{G(X)} \otimes f) \circ \rho_M(X) = \rho_N(X) \circ (f \otimes \id_{F(X)})
  \end{equation*}
  for all objects $X \in \mathcal{C}$. The composition of morphisms in $\lZlax(F, G)$ is defined by the composition of morphisms in $\mathcal{C}$.
\end{definition}

For three comonoidal functors $F$, $G$ and $H$ from $\mathcal{B}$ to $\mathcal{C}$, the tensor product
\begin{equation*}
  \otimes: \lZlax(G, H) \times \lZlax(F, G) \to \lZlax(F, H)
\end{equation*}
is defined by $(M, \rho_M) \otimes (N, \rho_N) = (M \otimes N, \rho_{M \otimes N})$ for $(M, \rho_M) \in \lZlax(G, H)$ and $(N, \rho_N) \in \lZlax(F, G)$, where $\rho_{M \otimes N}$ is the natural transformation defined by
\begin{equation*}
  \rho_{M \otimes N}(X) = (\rho_M(X) \otimes \id_N) \circ (\id_M \otimes \rho_N(X))
\end{equation*}
for $X \in \mathcal{B}$. The class of comonoidal functors from $\mathcal{B}$ to $\mathcal{C}$ form a bicategory with this operation. In particular, $\lZlax(F, F)$ is a monoidal category.

There is a right-handed version of the above construction: 
For two comonoidal functors $F$ and $G$ from $\mathcal{B}$ to $\mathcal{C}$, there is a category $\rZlax(F, G)$ whose objects are pairs $(M, \rho_M)$ consisting of an object $M \in \mathcal{C}$ and a natural transformation $\rho_M : F \otimes M \to M \otimes G$ satisfying similar axioms as above. The tensor product for these categories is also defined in a similar way as above and, in particular, $\rZlax(F, F)$ is a monoidal category.

\begin{definition}
  We call $\lZlax(F) := \lZlax(F, F)$ and $\rZlax(F) := \rZlax(F, F)$ the {\em left lax centralizer} and the {\em right lax centralizer} of $F$, respectively.
\end{definition}

\begin{remark}
  We will introduce some `lax' notions in this paper. Some of them are often called `weak' notions in literature. For example, the lax center (Definition~\ref{def:lax-center}) is called the weak center in \cite{MR3572267}
\end{remark}

\begin{remark}
  As a strong monoidal functor is a comonoidal functor in an obvious way, the left/right lax centralizer of such a functor makes sense.
\end{remark}

A {\em comonoidal adjunction} is an adjunction $(L, U, \eta, \varepsilon): \mathcal{C} \to \mathcal{D}$ between monoidal categories $\mathcal{C}$ and $\mathcal{D}$ such that $L: \mathcal{C} \to \mathcal{D}$ and $U: \mathcal{D} \to \mathcal{C}$ are comonoidal functors and the unit $\eta: \id_{\mathcal{C}} \to U L$ and the counit $\varepsilon: L U \to \id_{\mathcal{D}}$ are comonoidal natural transformations.

\begin{lemma}
  \label{lem:Z-lax-and-comonoidal-adj}
  Let $(L, U, \eta, \varepsilon): \mathcal{C} \to \mathcal{D}$ be a comonoidal adjunction as above, and let $\mathcal{V}$ be another monoidal category. Then, for any comonoidal functors $F: \mathcal{C} \to \mathcal{V}$ and $G: \mathcal{D} \to \mathcal{V}$, there are category isomorphisms
  \begin{equation*}
    \lZlax(F U, G) \cong \lZlax(F, G L)
    \quad \text{and} \quad
    \rZlax(F U, G) \cong \rZlax(F, G L).
  \end{equation*}
\end{lemma}
\begin{proof}
  We only give an isomorphism $\lZlax(F U, G) \cong \lZlax(F, G L)$ as the other one can be given in a similar way. We fix an object $M \in \mathcal{V}$ and consider the two sets:
  \begin{equation*}
    \mathcal{N}_1 = \Nat(M \otimes F U(-), G(-) \otimes M),
    \quad
    \mathcal{N}_2 = \Nat(M \otimes F(-), G L(-) \otimes M).
  \end{equation*}
  Given $\alpha \in \mathcal{N}_1$, we have an element of $\mathcal{N}_2$ defined by
  \begin{equation*}
    M \otimes F(X)
    \xrightarrow{\quad \id_M \otimes F(\eta_{X}) \quad}
    M \otimes F U L(X)
    \xrightarrow{\quad \alpha_{L(X)} \quad}
    G L(X) \otimes M
    \quad (X \in \mathcal{C}).
  \end{equation*}
  Given $\beta \in \mathcal{N}_2$, we have an element of $\mathcal{N}_1$ defined by
  \begin{equation*}
    M \otimes F U(X)
    \xrightarrow{\quad \beta_{U(X)} \quad}
    G L U(X) \otimes M
    \xrightarrow{\quad G(\varepsilon_X) \otimes \id_M \quad}
    G(X) \otimes M
    \quad (X \in \mathcal{D}).
  \end{equation*}
  By the zig-zag identities for $\eta$ and $\varepsilon$, we see that the above two constructions are mutually inverse to each other. Since $\eta$ and  $\varepsilon$ are comonoidal natural transformations, the bijection $\mathcal{N}_1 \cong \mathcal{N}_2$ so obtained restricts to the bijection
  \begin{equation*}
    \{ \alpha \in \mathcal{N}_1 \mid (M, \alpha) \in \lZlax(F U, G) \}
    \cong \{ \beta \in \mathcal{N}_2 \mid (M, \beta) \in \lZlax(F, G L) \}.
  \end{equation*}
  Hence we obtain a bijective correspondence between the objects of $\lZlax(F U, G)$ and the objects of $\lZlax(F, G L)$. One easily checks that this correspondence extends to a category isomorphism that is identity on morphisms.
\end{proof}

\subsection{Lax braided categories}

A {\em lax braiding} (also called a weak braiding) on a monoidal category $\mathcal{B}$ is a natural transformation $c_{X,Y}: X \otimes Y \to Y \otimes X$ ($X, Y \in \mathcal{B}$), which is not necessarily invertible, such that the equations
\begin{gather}
  \label{eq:def-lax-br-1}
  c_{X \otimes Y, Z} = (c_{X,Z} \otimes \id_Y) \circ (\id_X \otimes c_{Y,Z}), \\
  \label{eq:def-lax-br-2}
  c_{X, Y \otimes Z} = (\id_{X} \otimes c_{X, Z}) \circ (c_{X, Y} \otimes \id_Z), \\
  \label{eq:def-lax-br-3}
  c_{\unitobj, X} = \id_X = c_{X, \unitobj}
\end{gather}
hold for objects $X, Y, Z \in \mathcal{B}$. An invertible lax braiding is called a {\em braiding}. A (lax) {\em braided category} is a monoidal category equipped with a (lax) braiding.

The monoidal category $\lZlax(\mathcal{C}) := \lZlax(\id_{\mathcal{C}})$ is a lax braided category by the lax braiding given by $c_{\mathbf{M}, \mathbf{N}} = \rho_M(N)$ for two objects $\mathbf{M} = (M, \rho_M)$ and  $\mathbf{N} = (N, \rho_N)$ of $\lZlax(\mathcal{C})$. 
The monoidal category $\rZlax(\mathcal{C}) := \rZlax(\id_{\mathcal{C}})$ is also a lax braided category by the lax braiding given in an analogous way.

\begin{definition}
  \label{def:lax-center}
  We call the lax braided categories $\lZlax(\mathcal{C})$ and $\rZlax(\mathcal{C})$ the {\em left lax center} and the {\em right lax center} of $\mathcal{C}$, respectively.
\end{definition}

\subsection{Bimonads and Hopf monads}

We recall that a {\em monad} \cite[VI.1]{MR1712872} on a category $\mathcal{C}$ is a triple $(T, \mu, \eta)$ consisting of a functor $T: \mathcal{C} \to \mathcal{C}$ and natural transformation $\mu: T T \to T$ and $\eta: \id_{\mathcal{C}} \to T$ such that the equations
\begin{gather*}
  \mu_{M} \circ T(\mu_M) = \mu_{M} \circ \mu_{T(M)}
  \quad \text{and} \quad
  \mu_{M} \circ \eta_{T(M)} = \id_{T(M)} = \mu_{M} \circ T(\eta_{M})
\end{gather*}
hold for all objects $M \in \mathcal{C}$.

\begin{definition}[\cite{MR2355605,MR2793022}]
  Let $\mathcal{C}$ be a monoidal category. A {\em bimonad} on $\mathcal{C}$ is a data $(T, \mu, \eta, \delta, \varepsilon)$ such that
  \begin{enumerate}
  \item [(B1)] the triple $(T, \mu, \eta)$ is a monad on $\mathcal{C}$,
  \item [(B2)] the triple $(T, \delta, \varepsilon)$ is a comonoidal endofunctor on $\mathcal{C}$, and
  \item [(B3)] $\mu$ and $\eta$ are comonoidal natural transformations.
  \end{enumerate}
  Given two bimonads $T$ and $T'$ on $\mathcal{C}$, a {\em morphism} from $T$ to $T'$ is a comonoidal natural transformation $T \to T'$ that is also a morphism of monads. We denote by $\BiMon(\mathcal{C})$ the category of bimonads on $\mathcal{C}$.
\end{definition}


\subsubsection{Modules over a bimonad}

Let $T$ be a monad on a category $\mathcal{C}$. We recall that a {\em module} over $T$ ($=$ $T$-algebra \cite[VI.2]{MR1712872}) is a pair $(M, a)$ consisting of an object $M$ of $\mathcal{C}$ and a morphism $a: T(M) \to M$ in $\mathcal{C}$ satisfying the equations
\begin{gather}
  \label{eq:def-monad-module-1}  
  a \circ T(a) = a \circ \mu_M, \qquad
  a \circ \eta_M = \id_M.
\end{gather}
If $M = (M, a)$ and $N = (N, b)$ are $T$-modules, then a {\em morphism} of $T$-modules from $M$ to $N$ is a morphism $f: M \to N$ in $\mathcal{C}$ such that $b \circ T(f) = f \circ a$.

\begin{definition}
  We denote by $\lmod{T}$ the category of $T$-modules.
\end{definition}

Now we consider the case where $\mathcal{C}$ is a monoidal category. If $T = (T, \mu, \eta, \delta, \varepsilon)$ is a bimonad on $\mathcal{C}$, then the tensor product of two $T$-modules $(M, a)$ and $(N, b)$ is defined by $(M, a) \otimes (N, b) = (M \otimes N, (a \otimes b) \circ \delta_{M,N})$. The category $\lmod{T}$ is in fact a monoidal category with respect to this tensor product and the unit object given by $\unitobj_{\lmod{T}} = (\unitobj, \varepsilon)$.

\subsubsection{Comodules over a bimonad}

It is not obvious what a comodule over a bimonad is since a bimonad is not a comonad. Here we propose the following definition of a comodule over a bimonad:

\begin{definition}
  \label{def:comodule-over-bimonad}
  Let $\mathcal{C}$ be a monoidal category, and let $T = (T, \mu, \eta, \delta, \varepsilon)$ be a bimonad on $\mathcal{C}$. We define the category $\lcomod{T}$ of $T$-comodules by
  \begin{equation*}
    \lcomod{T} = \lZlax(\id_{\mathcal{C}}, T).
  \end{equation*}
\end{definition}

Spelling out the definition, a $T$-comodule is a pair $(M, \rho_M)$ consisting of an object $M \in \mathcal{C}$ and a family $\rho_M = \{ \rho_M(X) : M \otimes X \to T(X) \otimes M \}_{X \in \mathcal{C}}$ of morphisms in $\mathcal{C}$ such that the equations
\begin{gather}
  \label{eq:T-comod-def-0}
  \rho_M(W) \circ (f \otimes \id_M)
  = (T(f) \otimes \id_M) \circ \rho_M(V), \\
  \label{eq:T-comod-def-1}
  (\delta_{X,Y} \otimes \id_M) \circ \rho_M(X \otimes Y)
  = (\id_{T(X)} \otimes \rho_M(Y)) \circ (\rho_M(X) \otimes \id_Y), \\
  \label{eq:T-comod-def-2}
  (\varepsilon \otimes \id_{M}) \circ \rho_M(\unitobj) = \id_{M}
\end{gather}
hold for all all morphisms $f : V \to W$ in $\mathcal{C}$ and objects $X, Y \in \mathcal{C}$.

\begin{remark}
  As we will see in Subsection~\ref{subsec:bialgebroids}, if $T$ is the bimonad associated to a left bialgebroid $B$, then the category $\lcomod{T}$ can be identified with the category of left $B$-comodules.
  This is why we adopt the above definition that is different to \cite[Subsection 4.1]{MR2355605}.
  On the other hand, our definition may not be appropriate for bialgebras in braided monoidal categories (see Remark~\ref{rem:comodules-over-bialgebras}).
\end{remark}

Now let $\mathcal{C}$ and $T$ be as in Definition~\ref{def:comodule-over-bimonad}. Given two $T$-comodules $\mathbf{M} = (M, \rho_M)$ and $\mathbf{N} = (N, \rho_N)$, we define their tensor product by $\mathbf{M} \otimes \mathbf{N} = (M \otimes N, \rho_{M \otimes N})$, where $\rho_{M \otimes N}$ is the natural transformation defined by
\begin{equation}
  \label{eq:T-comod-tensor-def}
  \rho_{M \otimes N}(X) = (\mu_X \otimes \id_M \otimes \id_N)
  \circ (\rho_M(T(X)) \otimes \id_{N})
  \circ (\id_{M} \otimes \rho_N(X))
\end{equation}
for $X \in \mathcal{C}$. It is straightforward to check that $\lcomod{T}$ is a monoidal category with respect to this tensor product and the unit object
\begin{equation}
  \label{eq:T-comod-unit-def}
  \unitobj_{\lcomod{T}} = (\unitobj, \eta).
\end{equation}

\begin{theorem}
  \label{thm:T-comod-and-lax-centralizer}
  There is an isomorphism $\lcomod{T} \cong \lZlax(U)$ of monoidal categories, where $U: \lmod{T} \to \mathcal{C}$ is the forgetful functor.
\end{theorem}
\begin{proof}
  Let $L: \mathcal{C} \to \lmod{T}$ be the free $T$-module functor, that is, the functor defined by $L(X) = (T(X), \mu_X)$ for $X \in \mathcal{C}$. This is a comonoidal functor with structure morphisms inherited from $T$ and the pair $(L, U)$ is a comonoidal adjunction \cite[Example 2.4]{MR2793022}. Hence we have a category isomorphism
  \begin{equation*}
    \lcomod{T}
    = \lZlax(\id_{\mathcal{C}}, T)
    = \lZlax(\id_{\mathcal{C}}, U L)
    \cong \lZlax(U, U)
  \end{equation*}
  by Theorem \ref{lem:Z-lax-and-comonoidal-adj}. It is easy to see that this category isomorphism is in fact a strict monoidal functor.
\end{proof}

\subsubsection{Hopf monads}

A fusion operator \cite[Subsection 2.6]{MR2793022} for a bimonad is a category-theoretical counterpart of the Galois map of a bialgebra. A Hopf monad is defined to be a bimonad with invertible fusion operator. For reader's convenience, we include the full definition:

\begin{definition}[{\cite{MR2793022}}]
  Let $\mathcal{C}$ be a monoidal category, and let $T$ be a bimonad on $\mathcal{C}$. The {\em left fusion operator} $H^{\ell}$ and the {\em right fusion operator} $H^r$ for $T$ are the natural transformations
  \begin{align*}
    H^{\ell}_{M,N} & : T(M \otimes T(N)) \to T(M) \otimes T(N), \\
    H^{r}_{M,N} & : T(T(M) \otimes N) \to T(M) \otimes T(N)
  \end{align*}
  defined, respectively, by
  \begin{equation*}
    H^{\ell}_{M,N} = (\id_{T(M)} \otimes \mu_N) \circ \delta_{M, T(N)}
    \quad \text{and} \quad
    H^{r}_{M,N} = (\mu_M \otimes \id_{T(N)}) \circ \delta_{T(M), N}
  \end{equation*}
  for all objects $M, N \in \mathcal{C}$.  A {\em left} ({\it respectively}, {\em right}) {\em Hopf monad} is a bimonad whose left ({\it respectively}, right) fusion operator is invertible.
\end{definition}

\subsection{Bialgebras and Hopf algebras}
\label{subsec:bialgebras}

Let $\mathcal{B}$ be a lax braided category with lax braiding $\sigma$. A {\em bialgebra} in $\mathcal{B}$ is a data $(H, m, u, \Delta, \varepsilon)$ such that the triple $(H, m, u)$ is an algebra ($=$ a monoid \cite[VII.3]{MR1712872}) in $\mathcal{B}$, the triple $(H, \Delta, \varepsilon)$ is a coalgebra in $\mathcal{B}$ and the following equations hold:
\begin{equation*}
  \Delta \circ m = (m \otimes m) \circ (\id_H \otimes \sigma_{H,H} \otimes \id_H) \circ (\Delta \otimes \Delta), \quad
  \varepsilon \circ u = \id_{\unitobj}.
\end{equation*}
Given two bialgebras $H$ and $H'$ in $\mathcal{B}$, a {\em bialgebra morphism} from $H$ to $H'$ is a morphism $H \to H'$ in $\mathcal{B}$ that is an algebra morphism as well as a coalgebra morphism. We denote by $\BiAlg(\mathcal{B})$ the category of bialgebras in $\mathcal{B}$.

The endofunctor $T_H := H \otimes (-)$ on $\mathcal{C}$ has a natural structure of a monad on $\mathcal{C}$ if $H$ is an algebra in $\mathcal{C}$. If, furthermore, $(H, \rho_H)$ is a bialgebra in $\lZlax(\mathcal{C})$, then the monad $T_H$ is a bimonad on $\mathcal{C}$ by the comonoidal structure given by
\begin{equation*}
  \delta_{X,Y} = (\id_H \otimes \rho_H(X) \otimes \id_Y) \circ (\Delta_H \otimes \id_X \otimes \id_Y)
  \quad (X, Y \in \mathcal{C})
\end{equation*}
and $\varepsilon_H: T_H(\unitobj) \to \unitobj$, where $\Delta_H$ and $\varepsilon_H$ are the comultiplication and the counit of $H$, respectively.

\begin{remark}
  \label{rem:comodules-over-bialgebras}
  Unfortunately, $\lcomod{T_H}$ is {\em not} equivalent to the category of left $H$-comodules in general. For example, we consider the trivial bialgebra $H = \unitobj$. Then the former is equivalent to $\lZlax(\mathcal{C})$, while the latter is equivalent to $\mathcal{C}$. 
\end{remark}

The identity functor $\id_{\mathcal{C}}$ has a trivial structure of a bimonad on $\mathcal{C}$. An {\em augmentation} of a bimonad $T$ on $\mathcal{C}$ is a morphism $T \to \id_{\mathcal{C}}$ in $\BiMon(\mathcal{C})$. We define the category $\BiMon^a(\mathcal{C})$ of {\em augmented bimonads} on $\mathcal{C}$ to be the category of objects of $\BiMon(\mathcal{C})$ over the object $\id_{\mathcal{C}}$. There is a functor
\begin{equation}
  \label{eq:bialg-to-bimon}
  \BiAlg(\lZlax(\mathcal{C})) \to \BiMon^{a}(\mathcal{C}),
  \quad H \mapsto (T_H, \varepsilon_H \otimes \id).
\end{equation}
This functor is not an equivalence in general. An augmented bimonad $(T, e)$ on $\mathcal{C}$ is said to be {\em left regular} \cite[Subsection 5.6]{MR2793022} if the morphism
\begin{equation*}
  u^e_X = (\id_{T(\unitobj)} \otimes e_X) \circ \delta_{\unitobj,X}: T(X) \to T(\unitobj) \otimes X
\end{equation*}
is invertible for all objects $X \in \mathcal{C}$. The functor \eqref{eq:bialg-to-bimon} actually induces an equivalence between the category $\BiAlg(\lZlax(\mathcal{C}))$ and the full subcategory of left regular augmented bimonads on $\mathcal{C}$ \cite[Theorem 5.17]{MR2793022}.

\section{Tannaka construction for bimonads}
\label{sec:monadic-FRT}

\subsection{Tannaka construction for bimonads}
\label{subsec:bimonad-intro}

Let $\mathcal{V}$ be a monoidal category with cocontinuous tensor product (by this we mean that for all objects $X \in \mathcal{V}$ the endofunctors $X \otimes (-)$ and $(-) \otimes X$ on $\mathcal{V}$ preserve arbitrary small colimits existing in $\mathcal{V}$).
As in Subsection \ref{subsec:convention-duals}, we denote by $\mathcal{V}_{\rig}$ the full subcategory of right rigid objects of $\mathcal{V}$.

\begin{definition}
  \label{def:construction-data}
  A {\em construction data} over $\mathcal{V}$ is a pair $(\mathcal{D}, \omega)$ consisting of an essentially small monoidal category $\mathcal{D}$ and a strong monoidal functor $\omega: \mathcal{D} \to \mathcal{V}_{\rig}$ such that the following coend $T(M)$ exists for all $M \in \mathcal{V}$.
  \begin{equation}
    \label{eq:FRT-bimonad-def}
    T(M) = \int^{x \in \mathcal{D}} \omega(x) \otimes M \otimes \omega(x)^{\vee}.
  \end{equation}
\end{definition}

For basics on (co)ends, the reader is referred to \cite{MR1712872}.
Now let $(\mathcal{D}, \omega)$ be a construction data over $\mathcal{V}$.
For objects $x \in \mathcal{D}$ and $M \in \mathcal{V}$, we denote by
\begin{equation*}
  i_x(M) :  \omega(x) \otimes M \otimes \omega(x)^{\vee} \to T(M)
\end{equation*}
the universal dinatural transformation for the coend $T(M)$.
By the parameter theorem for coends ({\it cf}. \cite[IX.7]{MR1712872}), the assignment $M \mapsto T(M)$ extends to an endofunctor on $\mathcal{V}$ in such a way that $i_x(M)$ is natural in $M$. By the Fubini theorem for coends and the assumption that the tensor product of $\mathcal{V}$ preserves small colimits, we see that the object $T T(M)$ is the coend
\begin{equation*}
  T T(M) = \int^{x, y \in \mathcal{D}} \omega(x) \otimes \omega(y) \otimes M \otimes \omega(y)^{\vee} \otimes \omega(x)^{\vee}
\end{equation*}
with the universal dinatural transformation
\begin{equation*}
  i^{(2)}_{x, y}(M) := i_{x}(T(M)) \circ (\omega(x) \otimes i_{y}(M) \otimes \omega(x)^{\vee})
  \quad (x, y \in \mathcal{D}).
\end{equation*}
Now, for $M, N \in \mathcal{V}$, we define morphisms
\begin{gather*}
  \mu_{M} : T T(M) \to T(M),
  \quad \eta_M : M \to T(M), \\
  \delta_{M,N} : T(M \otimes N) \to T(M) \otimes T(N),
  \quad \varepsilon : T(\unitobj) \to \unitobj
\end{gather*}
in $\mathcal{V}$ to be the unique morphisms such that the equations
\begin{gather}
  \label{eq:FRT-bimonad-mu}
  \mu_{M} \circ i^{(2)}_{x, y}(M) = i_{x \otimes y}(M) \circ (\omega^{(2)}_{x, y} \otimes \id_M \otimes \check{\omega}^{(2)}_{x, y}), \\
  \label{eq:FRT-bimonad-eta}
  \eta_M = i_{\unitobj}(M) \circ (\omega_0 \otimes \check{\omega}_{0}), \\
  \label{eq:FRT-bimonad-delta}
  \begin{aligned}
    \delta_{M, N} \circ i_x(M \otimes N)
    & = (i_x(M) \otimes i_x(N)) \\
    & \qquad \circ (\id_{\omega(x)} \otimes \id_M \otimes \coev_{\omega(x)} \otimes \id_N \otimes \id_{\omega(x)^{\vee}}),
  \end{aligned} \\
  \label{eq:FRT-bimonad-eps}
  \varepsilon \circ i_{\unitobj}(x) = \eval_{\omega(x)}
\end{gather}
hold for all objects $M, N \in \mathcal{V}$ and $x, y \in \mathcal{D}$. Here,
\begin{equation*}
  \omega_0 : \unitobj_{\mathcal{V}} \to \omega(\unitobj_{\mathcal{D}})
  \quad \text{and} \quad
  \omega^{(2)}_{x,y} : \omega(x) \otimes \omega(y) \to \omega(x \otimes y)
\end{equation*}
are the monoidal structure of $\omega$ and
\begin{gather}
  \label{eq:omega-check-def}
  \check{\omega}_{0} = (\omega_0^{-1})^{\vee} \circ \vartheta_{0}
  \quad \text{and} \quad
  \check{\omega}^{(2)}_{x, y} = (\omega^{(2)}_{x, y}{}^{-1})^{\vee} \circ \vartheta^{(2)}_{\omega(x), \omega(y)},
\end{gather}
where $\vartheta_{0}$ and $\vartheta^{(2)}$ are the monoidal structure \eqref{eq:right-dual-monoidal} of the right duality functor.

\begin{theorem}
  \label{thm:construction-bimonad}
  $T = (T, \mu, \eta, \delta, \varepsilon)$ is a bimonad on $\mathcal{V}$.
\end{theorem}

The quantum double of the identity Hopf monad \cite{MR2869176} can be thought of as the case where $\omega$ is the identity functor.
This theorem is proved in a similar way as \cite{MR2869176}, however, we give a detailed proof in Subsection~\ref{subsec:proof-bimonad-construction} since the techniques for the proof are also important for later discussions.

The main result of this section is the following properties of this bimonad:

\begin{theorem}
  \label{thm:FRT-bimonad}
  The bimonad $T$ enjoys the following properties:
  \begin{enumerate}
  \item \label{item:main-thm-T-mod}
    There is an isomorphism $\lmod{T} \cong \rZlax(\omega)$ of monoidal categories.
  \item \label{item:main-thm-T-comod}
    There is an isomorphism $\lcomod{T} \cong \lZlax(U_T)$ of monoidal categories, where $U_T$ is the forgetful functor from $\lmod{T}$ to $\mathcal{V}$.
  \item \label{item:main-thm-V-complete}
    If $\mathcal{V}$ is complete, then $T$ admits a right adjoint.
  \item \label{item:main-thm-V-braided}
    If $\mathcal{V}$ is lax braided, then $T \cong B \otimes (-)$ for some bialgebra $B$ in $\mathcal{V}$.
  \item \label{item:main-thm-D-right-rigid}
    If $\mathcal{D}$ is right rigid, then $T$ is a left Hopf monad.
  \item \label{item:main-thm-D-left-rigid}
    If $\mathcal{D}$ is left rigid, then $T$ is a right Hopf monad.
  \item \label{item:main-thm-D-braided}
    If $\mathcal{D}$ has a lax braiding, then so does $\lcomod{T}$.
  \end{enumerate}
\end{theorem}

Let $U^T: \lcomod{T} \to \mathcal{V}$ be the forgetful functor. It turns out that an object of $\mathcal{V}$ of the form $\omega(x)$, $x \in \mathcal{D}$, has a natural structure of a $T$-comodule and, by this way, the functor $\omega: \mathcal{D} \to \mathcal{V}$ lifts to a strict monoidal functor $\widetilde{\omega}: \mathcal{D} \to \lcomod{T}$ such that $U^T \circ \widetilde{\omega} = \omega$. The bimonad $T$ also has the following universal property:

\begin{theorem}
  \label{thm:FRT-bimonad-universality}
  Let $T'$ be a bimonad on $\mathcal{V}$, and let $\omega': \mathcal{D} \to \lcomod{T'}$ be a strict monoidal functor such that $U^T \circ \omega' = \omega$. Then there is a unique morphism $\phi: T \to T'$ of bimonads on $\mathcal{V}$ such that the equation
  \begin{equation*}
    \omega' = \phi_{\sharp} \circ \widetilde{\omega}
  \end{equation*}
  holds, where $\phi_{\sharp}: \lcomod{T} \to \lcomod{T'}$ is the functor induced by $\phi$.
\end{theorem}

There are some interesting consequences of the above theorems. We first consider the case where $\mathcal{V}$ is a braided monoidal category (such as the category of modules over a fixed commutative ring). By Part (2) of the above theorem, we obtain a bialgebra $B$ in $\mathcal{V}$ such that $T \cong B \otimes (-)$ as bimonads (the structure of $B$ are given explicitly in Remark~\ref{rem:bialgebra-structure}). By Part (5), the bialgebra $B$ is a Hopf algebra in $\mathcal{V}$ if $\mathcal{D}$ is right rigid (a description of the antipode is found in Remark~\ref{rem:bialgebra-structure-antipode}). By Part (6), the category of $B$-comodules has a lax braiding if $\mathcal{D}$ is lax braided.

Next, we fix an algebra $A$ over a commutative ring $k$ and consider the case where $\mathcal{V}$ is the category of bimodules over $A$. Then, since $\mathcal{V}$ is complete, the functor $T$ has a right adjoint by Part (1). Thus, by a result of Szlach\'anyi \cite{MR1984397} which will be recalled in Subsection~\ref{subsec:bialgebroid-modules}, the bimonad $T$ actually arises from a left bialgebroid, say $B$, over the algebra $A$. By Part (3), the bialgebroid $B$ is a left Hopf algebroid if $\mathcal{D}$ is right rigid. Now we suppose that $\mathcal{D}$ has a lax braiding. Then, by Part (6), the category of left $B$-comodules has a lax braiding. It turns out that this lax braiding comes from a lax universal R-form $\mathbf{r} : B \times B \to A$ (Theorem~\ref{thm:B-omega-lax-universal-R}).

\begin{remark}
  Let $\mathcal{D}$ be a monoidal category, and let $\omega : \mathcal{D} \to \mathcal{V}$ be a strong monoidal functor. We assume that $\omega(x)$ has a left dual object in $\mathcal{V}$ for all objects $x \in \mathcal{D}$ and a coend
  $T'(M) = \int^{x \in \mathcal{D}} \omega(x)^{\ell} \otimes M \otimes \omega(x)$
  exists for all objects $M \in \mathcal{V}$, where $(-)^{\ell}$ denotes the left dual. This assumption is equivalent to that the pair $(\mathcal{D}^{\rev}, \omega^{\rev})$ is a construction data over $\mathcal{V}^{\rev}$ and thus the endofunctor $T'$ on $\mathcal{V}$ is a bimonad on $\mathcal{V}^{\rev}$, which is the same thing as a bimonad on $\mathcal{V}$. By applying Theorem~\ref{thm:FRT-bimonad} to $(\mathcal{D}^{\rev}, \omega^{\rev})$ and rephrasing the result in terms of $\mathcal{V}$, we obtain:
  \begin{enumerate}
  \item There is an isomorphism $\lmod{T'} \cong \lZlax(\omega)$ of monoidal categories.
  \item There is an isomorphism $\rcomod{T'} \cong \rZlax(U_{T'})$ of monoidal categories.
  \item If $\mathcal{V}$ is complete, then $T'$ admits a right adjoint.
  \item If $\mathcal{V}$ is lax braided, then $T' \cong (-) \otimes B$ for some bialgebra $B$ in $\mathcal{V}$.
  \item If $\mathcal{D}$ is left rigid, then $T'$ is a left Hopf monad.
  \item If $\mathcal{D}$ is right rigid, then $T'$ is a right Hopf monad.
  \item If $\mathcal{D}$ has a lax braiding, then so does $\rcomod{T'}$.
  \end{enumerate}
  Here, $\rcomod{T'} := \rZlax(\id_{\mathcal{C}}, T')$ is the monoidal category of `right' comodules over $T'$, which is defined in a similar way as $\lcomod{T'}$.
\end{remark}

\subsection{Graphical convention}

Till the end of this section, we fix a monoidal category $\mathcal{V}$ with cocontinuous tensor product and a construction data $(\mathcal{D}, \omega)$ over $\mathcal{V}$. To prove Theorems~\ref{thm:construction-bimonad}--~\ref{thm:FRT-bimonad-universality}, we use the graphical technique to represent morphisms in a monoidal category by string diagrams. Our convention is that a morphism goes from the top to the bottom of the diagram. If $X$ is a right rigid object in $\mathcal{V}$, then the evaluation $\eval_X: X \otimes X^{\vee} \to \unitobj$ and the coevaluation $\coev_X: \unitobj \to X^{\vee} \otimes X$ are expressed by a cup and a cap, respectively, like the following diagrams:
\begin{equation*}
  \PIC{gra-cup-01}
  = \PIC{gra-cup-02} \qquad \qquad
  \PIC{gra-cup-03}
  = \PIC{gra-cup-04}
\end{equation*}
By the definition of a right dual object, we have
\begin{equation}
  \label{eq:gra-cal-duality}
  \PIC{gra-cup-zigzag-01}
  \qquad \text{and} \qquad
  \PIC{gra-cup-zigzag-02}.
\end{equation}
Cockett and Seely \cite{MR1731041} introduced an idea representing a functor as a region on a diagram. We adopt McCurdy's variant \cite{MR2948489} of such a graphical expression and represent $\omega$ by a gray-colored region. For example, given a morphism $f: x \otimes y \to z$ in $\mathcal{D}$, we represent the morphism $\omega(f): \omega(x \otimes y) \to \omega(z)$ by
\begin{equation*}
  \PIC{gra-omega-01}
  \quad \text{or} \qquad
  \PIC{gra-omega-02}
\end{equation*}
As in \cite{MR2948489}, the monoidal structure of $\omega$ is expressed as follows:
\begin{equation*}
  \renewcommand{\PICSCALE}{.8}
  \omega^{(2)}_{x,y}
  = \PIC{gra-omega-04}\,, \quad
  \omega_0 =
  \begin{array}{c}
    \rotatebox{180}{$\PIC{gra-omega-06}$}
  \end{array},
  \quad
  (\omega^{(2)}_{x,y})^{-1}
  = \PIC{gra-omega-08}\,, \quad
  \omega_0^{-1} =
  \begin{array}{c}
    \PIC{gra-omega-06}
  \end{array}.
\end{equation*}
For each integer $n \ge 2$, there is the natural isomorphism
\begin{equation*}
  \omega^{(n)}_{x_1, \dotsc, x_n}:
  \omega(x_1) \otimes \dotsb \otimes \omega(x_n)
  \to\omega(x_1 \otimes \dotsb \otimes x_n)
  \quad (x_1, \dotsc, x_n \in \mathcal{D})
\end{equation*}
obtained by iterative use of $\omega^{(2)}$. We represent $\omega^{(3)}_{x, y, z}$ and its inverse as
\begin{equation*}
  \renewcommand{\PICSCALE}{.8}
  \omega^{(3)}_{x,y,z}
  = \PIC{gra-omega-12}
  \qquad \text{and} \qquad
  (\omega^{(3)}_{x,y,z})^{-1}
  = \PIC{gra-omega-14} \ ,
\end{equation*}
respectively, and analogous pictures are used to represent $\omega^{(n)}_{-,\cdots,-}$ for $n \ge 3$ and their inverses. An important observation is that, in a sense, a gray-colored region can be deformed continuously (see \cite{MR2948489}). For example, we have
\begin{equation}
  \label{eq:gra-cal-omega-assoc}
  \PIC{gra-omega-21}
  \ = \ \PIC{gra-omega-22}
  \ = \ \PIC{gra-omega-23}
\end{equation}
since the following equations hold:
\begin{align*}
  \omega^{(2)}_{x \otimes y, z} \circ (\omega^{(2)}_{x, y} \otimes \id_{\omega(z)}) 
  = \omega^{(3)}_{x, y, z}
  = \omega^{(2)}_{x, y \otimes z} \circ (\id_{\omega(x)} \otimes \omega^{(2)}_{y, z}).
\end{align*}
By the definition of a monoidal functor, we also have:
\begin{equation}
  \label{eq:gra-cal-omega-unit}
  \PIC[1.0]{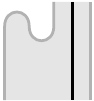}
  \ = \ \PIC[1.0]{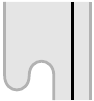}
  \ = \ \PIC[1.0]{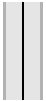}
  \ = \ \reflectbox{$\PIC[1.0]{gra-omega-32}$}
  \ = \ \reflectbox{$\PIC[1.0]{gra-omega-31}$}.
\end{equation}
Since $\omega$ is assumed to be strong monoidal, the following `non-continuous' deformations of regions are also allowed:
\begin{equation}
  \label{eq:gra-cal-omega-slimy}
  \PIC[1.0]{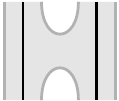}
  = \PIC[1.0]{gra-omega-33}
  \hspace{.5em} \PIC[1.0]{gra-omega-33} \,, \quad
  \PIC[1.0]{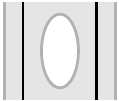}
  = \PIC[1.0]{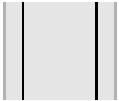} \,, \quad
  \PIC[1.0]{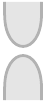}
  = \PIC[1.0]{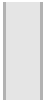} \,, \quad
  \PIC[1.0]{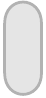}
  = \id_{\unitobj}.
\end{equation}

\subsection{Proof of Theorem~\ref{thm:construction-bimonad}}
\label{subsec:proof-bimonad-construction}

In this subsection, we show that the functor $T$ and the structure morphisms defined by~\eqref{eq:FRT-bimonad-mu}--\eqref{eq:FRT-bimonad-eps} form a bimonad on $\mathcal{V}$.
For our purpose, it is convenient to introduce the natural transformation
\begin{equation}
  \label{eq:partial-def}
  \begin{gathered}
    \partial_x(M): \omega(x) \otimes M \to T(M) \otimes \omega(x)
    \quad (x \in \mathcal{D}, M \in \mathcal{V}), \\
    \partial_x(M) := (i_{x}(M) \otimes \id_{\omega(x)}) \circ (\id_{\omega(x)} \otimes \id_{M} \otimes \coev_{\omega(x)}).
  \end{gathered}
\end{equation}
In string diagrams, we express the morphism $\partial_x(M)$ as follows:
\begin{equation*}
  \PIC{gra-omega-round-01}
  = \PIC{gra-omega-round-02}
\end{equation*}
Let $f: x \to y$ and $g: M \to N$ be morphisms in $\mathcal{D}$ and $\mathcal{V}$, respectively. 
The naturality of $\partial_x(M)$ allows the following deformation of diagrams:
\begin{equation}
  \label{eq:gra-nat-partial}
  \renewcommand{\PICSCALE}{.8}
  \PIC{gra-omega-round-11}
  = \PIC{gra-omega-round-12}
  \qquad
  \PIC{gra-omega-round-13}
  = \PIC{gra-omega-round-14}
\end{equation}
For a positive integer $n$, we define the natural transformation
\begin{equation*}
  \partial^{(n)}_{x_1, \cdots, x_n}(M) :
  \omega(x_1) \otimes \dotsb \otimes \omega(x_n) \otimes M
  \to T^n(M) \otimes \omega(x_1) \otimes \dotsb \otimes \omega(x_n)
\end{equation*}
for $x_1, \cdots, x_n \in \mathcal{D}$ and $M \in \mathcal{V}$ inductively by $\partial^{(1)}_{x_1}(M) = \partial_{x_1}(M)$ and
\begin{align*}
  \partial_{x_1, \cdots, x_n}^{(n)}(M)
  & = (\partial_{x_1, \cdots, x_{n-1}}^{(n-1)}(T(M)) \otimes \id_{\omega(x_n)}) \\
  & \qquad \circ (\id_{\omega(x_1)} \otimes \dotsb \otimes \id_{\omega(x_{n-1})} \otimes \partial_{x_n}(M))  
\end{align*}
for $n \ge 2$. By the same way as \cite[Lemma 5.4]{MR2869176}, we prove:

\begin{lemma}
  \label{lem:partial-universality}
  There is a natural isomorphism
  \begin{gather*}
    \Phi_{M,N}^{(n)}: \Hom_{\mathcal{V}}(T^n(M), N)
    \to \Nat(\omega^{\otimes n} \otimes M, N \otimes \omega^{\otimes n}), \\
    f \mapsto (f \otimes \id_{\omega(-)} \otimes \dotsb \otimes \id_{\omega(-)}) \circ \partial^{(n)}_{-, \dotsc, -}(M)
  \end{gather*}
  for $M, N \in \mathcal{V}$.
\end{lemma}

By using this lemma, we often eliminate dual objects from expressions involving the coend $T^n(M)$. For example, the structure morphisms of the bimonad $T$ can also be defined to be unique morphisms such that the equations
\begin{gather}
  \label{eq:FRT-bimonad-alter-mu}
  (\mu_M \otimes \id \otimes \id) \circ \partial^{(2)}_{x, y}(M)
  = (\id_M \otimes (\omega^{(2)}_{x, y})^{-1}) \circ \partial_{x \otimes y}(M) \circ (\omega^{(2)}_{x, y} \otimes \id_M), \\
  \label{eq:FRT-bimonad-alter-eta}
  \eta_M = (\id_{T(M)} \otimes \omega_0^{-1}) \circ \partial_{\unitobj}(M)
  \circ (\omega_0 \otimes \id_M), \\
  \label{eq:FRT-bimonad-alter-delta}
  (\delta_{M,N} \otimes \id_{\omega(x)}) \circ \partial_x(M \otimes N)
  = (\id_{T(M)} \otimes \partial_x(N)) \circ (\partial_x(M) \otimes \id_{N}), \\
  \label{eq:FRT-bimonad-alter-eps}
  (\varepsilon \otimes \id_{\omega(x)}) \circ \partial_x(\unitobj)
  = \id_{\omega(x)}
\end{gather}
hold for all objects $M, N \in \mathcal{V}$ and $x, y \in \mathcal{D}$.
                           
\begin{proof}[Proof of Theorem~\ref{thm:construction-bimonad}]
  We first show that $(T, \mu, \eta)$ is a monad on $\mathcal{V}$.
  We fix an object $M \in \mathcal{V}$.
  By using \eqref{eq:FRT-bimonad-alter-mu} expressed graphically, we have
  \begin{equation*}
    \renewcommand{\PICSCALE}{.8}
    \PIC{T-verify-mu-01}
    \mathop{=}^{\text{\eqref{eq:FRT-bimonad-alter-mu}}}
    \PIC{T-verify-mu-02}
    \mathop{=}^{\text{\eqref{eq:FRT-bimonad-alter-mu}}}
    \PIC{T-verify-mu-03}
  \end{equation*}
  for all objects $x, y, z \in \mathcal{D}$. In a similar way, we compute:
  \begin{equation*}
    \renewcommand{\PICSCALE}{.8}
    \PIC{T-verify-mu-04}
    \mathop{=}^{\text{\eqref{eq:gra-nat-partial}}}
    \PIC{T-verify-mu-05}
    \mathop{=}^{\text{\eqref{eq:FRT-bimonad-alter-mu}}}
    \PIC{T-verify-mu-03}
  \end{equation*}
  Thus, by Lemma~\ref{lem:partial-universality}, we obtain $\mu_M \mu_{T(M)} = T(\mu_M) \mu_M$. Other axioms of monads are also verified with the help of this lemma:
  For all $x \in \mathcal{D}$, we have
  \begin{align*}
    & \ (\mu_M \eta_{T(M)} \otimes \id_{\omega(x)}) \circ \partial_x(M) \\
    {}^{\eqref{eq:FRT-bimonad-alter-eta}}=
    & \ (\mu_M \otimes \omega_0^{-1} \otimes \id_{\omega(x)}) \circ \partial^{(2)}_{\unitobj,x}(M)
      \circ (\omega_0 \otimes \id_{\omega(x)} \otimes \id_M) \\
    {}^{\eqref{eq:FRT-bimonad-alter-mu}}=
    & \ (\id \otimes (\omega_{\unitobj,x}^{(2)} (\omega_0 \otimes \id))^{-1})
      \circ \partial_{\unitobj \otimes x}(M)
      \circ (\omega_{\unitobj,x}^{(2)} (\omega_0 \otimes \id) \otimes \id)
    = \partial_x(M),
  \end{align*}
  which implies $\mu_M \eta_M = \id_{M}$. For all $x \in \mathcal{D}$, we also have
  \begin{align*}
    & \ (\mu_M T(\eta_{M}) \otimes \id_{\omega(x)}) \circ \partial_x(M) \\
    {}^{\eqref{eq:gra-nat-partial}}=
    & \ (\mu_M \otimes \id_{\omega(x)}) \circ \partial_x(T(M)) \circ (\id_{\omega(x)} \otimes \eta_M) \\
    {}^{\eqref{eq:FRT-bimonad-alter-eta}}=
    & \ (\mu_M \otimes \id_{\omega(x)} \otimes \omega_0^{-1})
      \circ \partial_{x,\unitobj}^{(2)}(M) \circ (\id_{\omega(x)} \otimes \omega_0 \otimes \id_M) \\
    {}^{\eqref{eq:FRT-bimonad-alter-mu}}=
    & \ (\id \otimes (\omega^{(2)}_{x,\unitobj} (\id \otimes \omega_0))^{-1})
      \circ \partial_{\unitobj \otimes x}(M)
      \circ (\omega^{(2)}_{x,\unitobj} (\id \otimes \omega_0) \otimes \id)
      = \partial_x(M),
  \end{align*}
  which implies $\mu_M T(\eta_M) = \id_M$. Now we have completed the proof that $(T, \mu, \eta)$ is a monad on $\mathcal{V}$.

  Secondly, we show that $(T, \delta, \varepsilon)$ is a monoidal endofunctor.
  We fix three objects $L$, $M$ and $N$ of $\mathcal{V}$.
  By using \eqref{eq:FRT-bimonad-alter-delta} expressed graphically, we compute
  \begin{gather*}
    \renewcommand{\PICSCALE}{.8}
    \PIC{T-verify-comonoidal-01}
    \hspace{-.75em}
    \mathop{=}^{\eqref{eq:FRT-bimonad-alter-delta}}
    \hspace{-.5em}
    \PIC{T-verify-comonoidal-02}
    \mathop{=}^{\eqref{eq:FRT-bimonad-alter-delta}}
    \hspace{-.5em}
    \PIC{T-verify-comonoidal-03}
    \hspace{-.25em}
    \mathop{=}^{\eqref{eq:FRT-bimonad-alter-delta}}
    \hspace{.25em}
    \PIC{T-verify-comonoidal-04}
  \end{gather*}
  for all $x \in \mathcal{D}$. We also compute
  \begin{equation*}
    ((\varepsilon \otimes \id)\delta_{\unitobj,M} \otimes \id_{\omega(x)}) \circ \partial_x(M)
    \mathop{=}^{\eqref{eq:FRT-bimonad-alter-delta}}
    (\varepsilon \otimes \partial_x(M)) \circ (\partial_x(\unitobj) \otimes \id)
    \mathop{=}^{\eqref{eq:FRT-bimonad-alter-eps}} \partial_x(M)
  \end{equation*}
  for all $x \in \mathcal{D}$. Thus, by Lemma~\ref{lem:partial-universality}, we obtain
  \begin{equation*}
    (\delta_{L,M} \otimes \id) \circ \delta_{L \otimes M,N} = (\id \otimes \delta_{M,N}) \circ \delta_{L, M \otimes N}
    \quad \text{and} \quad
    (\varepsilon \otimes \id) \circ \delta_{\unitobj,M} = \id_M
  \end{equation*}
  for $L, M, N \in \mathcal{V}$. One can also prove $(\varepsilon \otimes \id) \circ \delta_{\unitobj,M} = \id_M$ in a similar way.
  Hence we have completed the proof that $T$ is a comonoidal functor.

  Next, we check that the multiplication $\mu : T T \to T$ is a comonoidal natural transformations.
  The compatibility between $\mu$ and the comultiplication $\delta$,
  \begin{equation}
    \label{eq:T-verify-mu-comonoidal-1}
    (\mu_M \otimes \mu_N) \circ \delta_{T(M),T(N)} \circ T(\delta_{M,N})
    = \delta_{M,N} \circ \mu_{M \otimes N}
    \quad (M, N \in \mathcal{V}),
  \end{equation}
  is verified by Figure~\ref{fig:T-verify-mu-comonoidal-1} and Lemma~\ref{lem:partial-universality}. To check the compatibility between $\mu$ and the counit $\varepsilon$, the universal property of $T T(\unitobj)$ as a coend seems to be rather useful. For all $x, y \in \mathcal{D}$, we have
  \begin{gather*}
    \varepsilon \circ \mu_{\unitobj} \circ i^{(2)}_{x,y}(\unitobj)
    \mathop{=}^{\eqref{eq:FRT-bimonad-mu}}
    \varepsilon \circ i_{x \otimes y}(\unitobj) \circ (\omega^{(2)}_{x,y} \otimes \check{\omega}^{(2)}_{x,y})
    \mathop{=}^{\eqref{eq:FRT-bimonad-eps}}
    \eval_{\omega(x \otimes y)} \circ (\omega^{(2)}_{x,y} \otimes \check{\omega}^{(2)}_{x,y}) \\[5pt]
    \mathop{=}^{\eqref{eq:rigid-mon-cat-eval-2}}
    \eval^{(2)}_{\omega(y),\omega(x)}
    \mathop{=}^{\eqref{eq:FRT-bimonad-eps}}
    \varepsilon i_x(T(\unitobj)) \circ (\id \otimes \varepsilon i_y(\unitobj) \otimes \id)
    = \varepsilon T(\varepsilon) \circ i_{x,y}^{(2)}(\unitobj)
  \end{gather*}
  and hence $\varepsilon \mu_{\unitobj} = \varepsilon T(\varepsilon)$ by the universal property.

  Finally, we check that the unit $\eta : \id_{\mathcal{V}} \to T$ is a comonoidal natural transformation.
  For all objects $M, N \in \mathcal{V}$, we have
  \begin{align*}
    \delta_{M,N} \circ \eta_{M \otimes N}
    \ {}^{\eqref{eq:FRT-bimonad-alter-eps}}=
    & \ (\delta_{M,N} \otimes \omega_0^{-1})
      \circ \partial_{\unitobj}(M \otimes N) \circ (\omega_0 \otimes \id_M \otimes \id_N), \\
    \ {}^{\eqref{eq:FRT-bimonad-alter-delta}}=
    & \ (\id_{T(M)} \otimes \id_{T(N)} \otimes \omega_0^{-1})
      \circ (\id_{T(M)} \otimes \partial_{\unitobj}(N)) \\
    & \qquad \circ (\partial_{\unitobj}(M) \otimes \id_N)
      \circ (\omega_0 \otimes \id_M \otimes \id_N) \\[-3pt]
    \ {}^{\eqref{eq:FRT-bimonad-alter-eps}}=
    & \ (\id_{T(M)} \otimes \eta_N) \circ (\eta_M \otimes \id_N)
      = \eta_M \otimes \eta_N
  \end{align*}
  and
  $\displaystyle \varepsilon \circ \eta_{\unitobj}
  \mathop{=}^{\text{\eqref{eq:FRT-bimonad-alter-eta}}}
  (\varepsilon \otimes \omega_0^{-1}) \circ \partial_{\unitobj}(\unitobj) \circ \omega_0
  \mathop{=}^{\text{\eqref{eq:FRT-bimonad-alter-eps}}}
  \id_{\unitobj}$. The proof is done.
\end{proof}

\begin{figure}
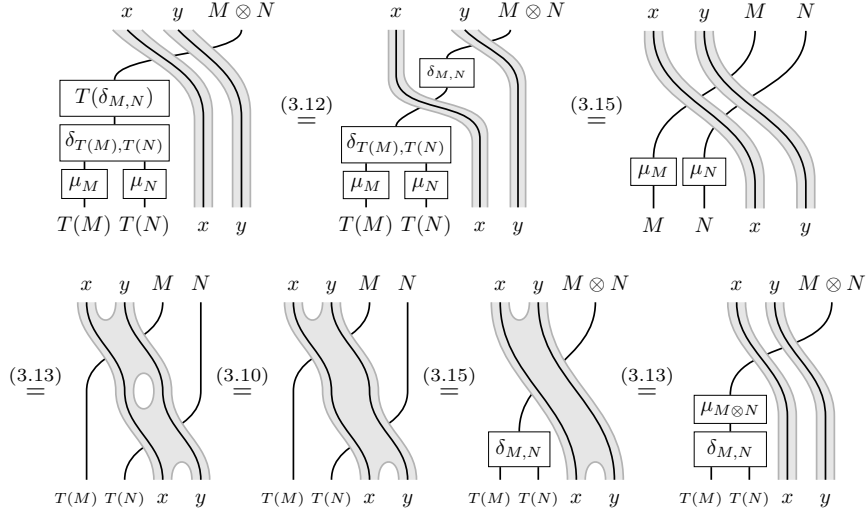

  \begin{equation*}
    \renewcommand{\PICSCALE}{.8}
    \begin{gathered}
      \PIC{T-verify-mu-comonoidal-01}
      \mathop{=}^{\text{\eqref{eq:gra-nat-partial}}}        
      \PIC{T-verify-mu-comonoidal-02}
      \mathop{=}^{\text{\eqref{eq:FRT-bimonad-alter-delta}}}
      \PIC{T-verify-mu-comonoidal-03} \\[10pt]
      \mathop{=}^{\text{\eqref{eq:FRT-bimonad-alter-mu}}}
      \hspace{-.5em}
      \PIC{T-verify-mu-comonoidal-04}
      \mathop{=}^{\text{\eqref{eq:gra-cal-omega-slimy}}}
      \hspace{-.5em}
      \PIC{T-verify-mu-comonoidal-05}
      \mathop{=}^{\text{\eqref{eq:FRT-bimonad-alter-delta}}}
      \hspace{-.5em}
      \PIC{T-verify-mu-comonoidal-06}
      \hspace{-.5em}
      \mathop{=}^{\text{\eqref{eq:FRT-bimonad-alter-mu}}}
      \PIC{T-verify-mu-comonoidal-07}
    \end{gathered}
  \end{equation*}
  \caption{Proof of equation \eqref{eq:T-verify-mu-comonoidal-1}}
  \label{fig:T-verify-mu-comonoidal-1}
\end{figure}

\subsection{Proof of Theorem~\ref{thm:FRT-bimonad} (\ref{item:main-thm-T-mod})}

In this subsection, we prove Part (\ref{item:main-thm-T-mod}) of Theorem~\ref{thm:FRT-bimonad}.
We fix an object $M \in \mathcal{V}$. Let $\rho(x) : \omega(x) \otimes M \to M \otimes \omega(x)$ ($x \in \mathcal{D}$) be a natural transformation, and let $a: T(M) \to M$ be the morphism in $\mathcal{V}$ corresponding to $\rho$ by Lemma~\ref{lem:partial-universality}. Then, by~\eqref{eq:FRT-bimonad-alter-mu}, we have
\begin{gather*}
  (a \mu_M \otimes \id \otimes \id) \circ \partial^{(2)}_{x,y}(M)
  = (\id_M \otimes (\omega^{(2)}_{x, y})^{-1}) \circ \rho(x \otimes y) \circ (\omega^{(2)}_{x, y} \otimes \id_M), \\
  (a T(a) \otimes \id \otimes \id) \circ \partial^{(2)}_{x,y}(M)
  = (\rho(x) \otimes \id_{\omega(y)}) \circ (\id_{\omega(x)} \otimes \rho(y))
\end{gather*}
for all objects $x, y \in \mathcal{D}$. By \eqref{eq:FRT-bimonad-alter-eta}, we also have
\begin{gather*}
  a \circ \eta_M = (\id_M \otimes \omega_0^{-1}) \circ \rho(\unitobj) \circ (\omega_0 \otimes \id_{M}).
\end{gather*}
Thus $(M, a)$ is a $T$-module if and only if $(T, \rho) \in \rZlax(\omega)$. This observation establishes a bijection between $\Obj(\lmod{T})$ and $\Obj(\rZlax(\omega))$. This bijection extends to an isomorphism $\lmod{T} \cong \rZlax(\omega)$ of categories. By \eqref{eq:FRT-bimonad-alter-delta} and \eqref{eq:FRT-bimonad-alter-eps}, we see that this category isomorphism is in fact an isomorphism of monoidal categories. The proof is done.

\subsection{Proof of Theorem~\ref{thm:FRT-bimonad} (\ref{item:main-thm-T-comod})}

This is a special case of Theorem~\ref{thm:T-comod-and-lax-centralizer}.

\subsection{Proof of Theorem~\ref{thm:FRT-bimonad} (\ref{item:main-thm-V-complete})}

In this subsection, we assume that $\mathcal{V}$ is complete and prove Theorem~\ref{thm:FRT-bimonad} (\ref{item:main-thm-V-complete}).
By the completeness of $\mathcal{V}$, the end
\begin{equation*}
  T^{\natural}(M) := \int_{x \in \mathcal{D}} \omega(x)^{\vee} \otimes M \otimes \omega(x)
\end{equation*}
exists for each object $M \in \mathcal{V}$. By the parameter theorem for ends \cite[IX.7]{MR1712872}, we extend the assignment $M \mapsto T^{\natural}(M)$ to an endofunctor $T^{\natural}$ on $\mathcal{V}$. The functor $T^{\natural}$ is right adjoint to $T$. Indeed, we have natural isomorphisms
\begin{align*}
  \Hom_{\mathcal{V}}(M, T^{\natural}(N))
  & \cong \textstyle \int_{x \in \mathcal{D}} \Hom_{\mathcal{V}}(M, \omega(x)^{\vee} \otimes N \otimes \omega(x)) \\
  & \cong \textstyle \int_{x \in \mathcal{D}} \Hom_{\mathcal{V}}(\omega(x) \otimes M \otimes \omega(x)^{\vee}, N)
    \cong \Hom_{\mathcal{V}}(T(M), N)
\end{align*}
for $M, N \in \mathcal{V}$. The proof is done.

\begin{remark}
  Since $T$ is a bimonad ($=$ a comonoidal monad) on $\mathcal{V}$, its right adjoint $T^{\natural}$ has a natural structure of a monoidal comonad on $\mathcal{V}$.
\end{remark}

\subsection{Proof of Theorem~\ref{thm:FRT-bimonad} (\ref{item:main-thm-V-braided})}

In this subsection, we assume that $\mathcal{V}$ has a lax braiding $\sigma$ and aim to give a bialgebra $B$ in $\mathcal{V}$ such that $T \cong B \otimes (-)$. For this purpose, we first equip a left regular augmentation with the bimonad $T$. By using the lax braiding of $\mathcal{V}$, we define $e_M: T(M) \to M$ ($M \in \mathcal{V}$) to be the unique morphism in $\mathcal{V}$ such that the equation
\begin{equation}
  \label{eq:T-aug-def}
  (e_M \otimes \id_{\omega(x)}) \circ \partial_x(M) = \sigma_{\omega(x), M}
\end{equation}
holds for all objects $x \in \mathcal{D}$. The naturality of $\sigma$ ensures that the family $e = \{ e_M \}$ of morphisms in $\mathcal{V}$ is a natural transformation from $T$ to $\id_{\mathcal{V}}$.

\begin{lemma}
  The natural transformation $e$ is an augmentation of $T$.
\end{lemma}
\begin{proof}
  We shall check that the equations
  \begin{gather}
    \label{eq:T-aug-1}
    e_M \circ \mu_M = e_M \circ T(e_M), \\
    \label{eq:T-aug-2}
    e_M \circ \eta_M = \id_M, \\
    \label{eq:T-aug-3}
    (e_M \otimes e_N) \circ \delta_{M, N} = e_{M \otimes N}, \\
    \label{eq:T-aug-4}
    e_{\unitobj} = \varepsilon.
  \end{gather}
  hold for all $M, N \in \mathcal{V}$. Equation~\eqref{eq:T-aug-1} is verified as follows:
  \begin{equation*}
    \renewcommand{\PICSCALE}{.8}
      \PIC{aug-mu-01}
      \mathop{=}^{\text{\eqref{eq:FRT-bimonad-alter-mu}}}
      \PIC{aug-mu-02}
      \mathop{=}^{\text{\eqref{eq:T-aug-def}}}
      \PIC{aug-mu-03}
      =
      \PIC{aug-mu-04}
      \mathop{=}^{\text{\eqref{eq:T-aug-def}}}
      \PIC{aug-mu-05}
      \mathop{=}^{\text{\eqref{eq:gra-nat-partial}}}
      \PIC{aug-mu-06}
  \end{equation*}
  The third equality follows from the naturality of the lax braiding. Equation~\eqref{eq:T-aug-2} is verified as follows:
  \begin{align*}
    e_M \circ \eta_M
    \mathop{=}^{\eqref{eq:FRT-bimonad-alter-eta}}
    & \ e_M \circ (\id_{T(M)} \otimes \omega_0^{-1}) \circ \partial_{\unitobj}(M) \circ (\omega_0 \otimes \id_M) \\
    \mathop{=}^{\eqref{eq:T-aug-def}}
    & \ (\id_{T(M)} \otimes \omega_0^{-1}) \circ \sigma_{\omega(\unitobj), M} \circ (\omega_0 \otimes \id_M) = \id_M.
  \end{align*}
  Equation~\eqref{eq:T-aug-3} is verified as follows:
  \begin{equation*}
    \renewcommand{\PICSCALE}{.8}
    \PIC{aug-delta-01}
    \mathop{=}^{\text{\eqref{eq:FRT-bimonad-alter-delta}}}
    \PIC{aug-delta-02}
    \mathop{=}^{\text{\eqref{eq:T-aug-def}}}
    \PIC{aug-delta-03}
    = \PIC{aug-delta-04}
    \mathop{=}^{\text{\eqref{eq:T-aug-def}}}
    \PIC{aug-delta-05}
  \end{equation*}
  Equation~\eqref{eq:T-aug-4} is verified as follows:
  \begin{equation*}
    \renewcommand{\PICSCALE}{.8}
    (e_{\unitobj} \otimes \id_{\omega(x)}) \circ \partial_x(\unitobj)
    \mathop{=}^{\text{\eqref{eq:T-aug-def}}}
    \sigma_{\omega(x), \unitobj}
    = \id_{\omega(x)}
    \mathop{=}^{\text{\eqref{eq:FRT-bimonad-alter-eps}}}
    (\varepsilon \otimes \id_{\omega(x)}) \circ \partial_x(\unitobj).
    \qedhere
  \end{equation*}
\end{proof}

Now we consider the natural transformation
\begin{equation*}
  u^{e}_M := (\id_{T(\unitobj)} \otimes e_M) \circ \delta_{\unitobj, M}:
  T(M) \to T(\unitobj) \otimes M
  \quad (M \in \mathcal{V}).
\end{equation*}
We construct the inverse of $u^{e}$ as follows: Let $M$ be an object of $\mathcal{V}$. Since the tensor product of $\mathcal{V}$ is assumed to be cocontinuous, the object $T(\unitobj) \otimes M$ is a coend
\begin{equation*}
  T(\unitobj) \otimes M = \int^{x \in \mathcal{D}} \omega(x) \otimes \omega(x)^{\vee} \otimes M
\end{equation*}
with the universal dinatural transformation $i_{\unitobj}(x) \otimes \id_M$. By using the universal property, we define $\overline{u}{}^{e}_M: T(\unitobj) \otimes M \to T(M)$ to be the unique morphism in $\mathcal{V}$ such that the following equation holds for all objects $x \in \mathcal{D}$.
\begin{equation}
  \label{eq:T-aug-left-reg-2}
  \overline{u}{}^{e}_M \circ (i_x(\unitobj) \otimes \id_M)
  = i_x(M) \circ (\id_{\omega(x)} \otimes \sigma_{\omega(x)^{\vee}, M}).
\end{equation}

\begin{lemma}
  $\overline{u}{}^{e}$ is the inverse of $u^e$.
\end{lemma}
\begin{proof}
  By the definitions of $\delta$ and $e$, we have
  \begin{equation}
    \label{eq:T-aug-left-reg-1}
    \renewcommand{\PICSCALE}{.8}
    u^e_M \circ i_{x}(M)
    \mathop{=}^{\text{\eqref{eq:FRT-bimonad-delta}}}
    \PIC{aug-u-def-01}
    \mathop{=}^{\text{\eqref{eq:T-aug-def}}}    
    \PIC{aug-u-def-02}.
  \end{equation}
  for all objects $x \in \mathcal{D}$ and $M \in \mathcal{V}$. We have
  \begin{equation*}
    \renewcommand{\PICSCALE}{.8}
    \PIC{aug-u-inv-01}
    \mathop{=}^{\text{\eqref{eq:T-aug-left-reg-2}}}
    \PIC{aug-u-inv-02}
    \mathop{=}^{\text{\eqref{eq:T-aug-left-reg-1}}}
    \PIC{aug-u-inv-03}
    = i_{\unitobj}(x) \otimes \id_M
  \end{equation*}
  for all $x \in \mathcal{D}$ and $M \in \mathcal{V}$. Thus $u^e \circ \overline{u}{}^e$ is the identity. We also have
  \begin{equation*}
    \renewcommand{\PICSCALE}{.8}
    \PIC{aug-u-inv-11}
    \mathop{=}^{\text{\eqref{eq:T-aug-left-reg-1}}}
    \PIC{aug-u-inv-12}
    \mathop{=}^{\text{\eqref{eq:T-aug-left-reg-2}}}
    \PIC{aug-u-inv-13}
    = i_x(M)
  \end{equation*}
  for all $x \in \mathcal{D}$ and $M \in \mathcal{V}$. Thus $\overline{u}{}^{e} \circ u^e$ is the identity. The proof is done.
\end{proof}

\begin{proof}[Proof of Theorem~\ref{thm:FRT-bimonad} (\ref{item:main-thm-V-braided})]
  We set $B = T(\unitobj)$ and define
  \begin{equation*}
    \tau(M) = (e_M \otimes \id_{B}) \circ \delta_{M, \unitobj} \circ \overline{u}_M^e
    : B \otimes M \to M \otimes B
  \end{equation*}
  for $M \in \mathcal{V}$. By the above lemma and the proof of \cite[Theorem 5.17]{MR2793022} (especially \cite[Lemma 5.15]{MR2793022}), we see that $(B, \tau)$ has a structure of a bialgebra in the lax center $\lZlax(\mathcal{C})$ such that $T \cong B \otimes (-)$ as bimonads. However, we have
  \begin{equation*}
    \renewcommand{\PICSCALE}{.8}
    \PIC{aug-tau-01}
    \mathop{=}^{\substack{
        \eqref{eq:FRT-bimonad-delta} \\
        \eqref{eq:T-aug-left-reg-2}}}
    \ \PIC{aug-tau-02}
    \mathop{=}^{\eqref{eq:T-aug-def}}
    \PIC{aug-tau-03}
    \mathop{=}^{}
    \PIC{aug-tau-04}
  \end{equation*}
  for all $M \in \mathcal{V}$ and $x \in \mathcal{D}$, where the third equality follows from the naturality of the lax braiding. Namely, $\tau$ is a component of the lax braiding of $\mathcal{V}$. This means that $B$ is in fact a bialgebra in $\mathcal{V}$. The proof is done.
\end{proof}

\begin{remark}
  \label{rem:bialgebra-structure}
  According to \cite[Lemma 5.15]{MR2793022}, the multiplication, the unit, the comultiplication and the counit of the bialgebra $B = T(\unitobj)$ in the above proof are $\mu_{\unitobj}\overline{u}{}^{e}_B$, $\eta_{\unitobj}$, $\delta_{\unitobj, \unitobj}$ and $\varepsilon$, respectively. The natural transformation $u^e$ is in fact an isomorphism of bimonads.
\end{remark}

\subsection{Proof of Theorem~\ref{thm:FRT-bimonad} (\ref{item:main-thm-D-right-rigid})}

In this subsection, we aim to show that $T$ is a left Hopf monad if $\mathcal{D}$ is right rigid. For this purpose, we first give a description of the left fusion operator of $T$. We recall that it is the natural transformation
\begin{equation*}
  H^{\ell}_{M,N}: T(M \otimes T(N)) \to T(M) \otimes T(N)
  \quad (M, N \in \mathcal{V})
\end{equation*}
defined by $H^{\ell}_{M,N} = (\id_{T(M)} \otimes \mu_N) \circ \delta_{M, T(N)}$ for $M, N \in \mathcal{V}$. By the Fubini theorem for coends and the assumption that the tensor product of $\mathcal{V}$ preserves colimits, we have natural isomorphisms
\begin{align*}
  & \Hom_{\mathcal{V}}(T(M \otimes T(N)), L) \\
  & \cong \textstyle \int_{y \in \mathcal{D}}
    \Hom_{\mathcal{V}}(\omega(x) \otimes M \otimes T(N) \otimes \omega(x)^{\vee}, L) \\
  & \cong \textstyle \int_{x, y \in \mathcal{D}}
    \Hom_{\mathcal{V}}(\omega(x) \otimes M \otimes \omega(y) \otimes N \otimes \omega(y)^{\vee} \otimes \omega(x)^{\vee}, L) \\
  & \cong \textstyle \int_{x, y \in \mathcal{D}}
    \Hom_{\mathcal{V}}(\omega(x) \otimes M \otimes \omega(y) \otimes N, L \otimes \omega(x) \otimes \omega(y)) \\
  & \cong \Nat(\omega \otimes M \otimes \omega \otimes N, L \otimes \omega \otimes \omega)
\end{align*}
for $L, M, N \in \mathcal{V}$. Given a morphism $f: T(M \otimes T(N)) \to L$, the corresponding natural transformation is given by
\begin{equation*}
  (f \otimes \id_{\omega(x)} \otimes \id_{\omega(y)})
  \circ (\partial_x(M \otimes T(N)) \otimes \id_{\omega(y)})
  \circ (\id_{\omega(x)} \otimes \id_{M} \otimes \partial_y(N))
\end{equation*}
for $x, y \in \mathcal{D}$. The natural transformation corresponding to the morphism $H^{\ell}_{M,N}$ is graphically computed as follows:
\begin{equation}
  \label{eq:T-left-fusion}
  \renewcommand{\PICSCALE}{.9}
  \PIC{T-fusion-01}
  \mathop{=}^{\text{\eqref{eq:FRT-bimonad-alter-delta}}}
  \PIC{T-fusion-02}
  \mathop{=}^{\text{\eqref{eq:FRT-bimonad-alter-mu}}}
  \PIC{T-fusion-03}.
\end{equation}
Here, the unlabeled box means the identity morphism on $M \otimes T(N)$, which is placed to bunch two strands representing $M$ and $T(N)$.

\begin{proof}[Proof of Theorem~\ref{thm:FRT-bimonad} (\ref{item:main-thm-D-right-rigid})]
  We assume that $\mathcal{D}$ is right rigid. For each $x \in \mathcal{D}$, we fix its right dual object $x^{\vee}$ in $\mathcal{D}$. There is an isomorphism $\zeta_x: \omega(x)^{\vee} \to \omega(x^{\vee})$ in $\mathcal{V}$ characterized by either of the equations
  \begin{gather}
    \label{eq:def-duality-trans-1}
    (\omega^{(2)}_{x^{\vee}, x})^{-1} \circ \omega(\coev_x) \circ \omega_0
    = (\zeta_x \otimes \id_{\omega(x)}) \circ \coev_{\omega(x)}, \\
    \label{eq:def-duality-trans-2}
    \omega_0^{-1} \circ \omega(\eval_x) \circ \omega^{(2)}_{x, x^{\vee}}
    \circ (\id_{\omega(x)} \otimes \zeta_x)
    = \eval_{\omega(x)}
  \end{gather}
  for $x \in \mathcal{D}$. The family $\zeta = \{ \zeta_x \}_{x \in \mathcal{D}}$ of morphisms is an isomorphism of monoidal functors \cite[Section 1]{MR2381536}, which we call the {\em duality transformation}. We identify $\omega(x^{\vee})$ with $\omega(x)^{\vee}$ through the duality transformation. In view of \eqref{eq:def-duality-trans-1} and~\eqref{eq:def-duality-trans-2}, we may illustrate $\eval_{\omega(x)}$ and $\coev_{\omega(x)}$ as follows:
  \begin{equation*}
    \eval_{\omega(x)} = \PIC{gra-omega-eval-01}
    \qquad
    \coev_{\omega(x)} = \PIC{gra-omega-eval-02}
  \end{equation*}
  The object $T(M) \otimes T(N)$ is a coend
  \begin{equation*}
    T(M) \otimes T(N) = \int_{x, y \in \mathcal{D}} \omega(x) \otimes M \otimes \omega(x)^{\vee} \otimes \omega(y) \otimes N \otimes \omega(y)^{\vee}
  \end{equation*}
  with the universal dinatural transformation $i_x(M) \otimes i_y(N)$. By the universal property, we define the natural transformation
  \begin{equation*}
    \overline{H}{}^{\ell}_{M,N}: T(M) \otimes T(N) \to T(M \otimes T(N))
    \quad (M, N \in \mathcal{V})
  \end{equation*}
  to be the unique morphism in $\mathcal{V}$ such that the equation
  \begin{align*}
    (\overline{H}{}^{\ell}_{M,N} & \otimes \id_{\omega(y)}) \circ (i_x(M) \otimes \partial_y(N)) \\
    & = (\id_{T(M \otimes T(N)} \otimes \eval_{\omega(x)} \otimes \id_{\omega(y)})
      \circ (\partial_x(M \otimes T(N)) \otimes (\omega^{(2)}_{x^{\vee}, y})^{-1}) \\
    & \qquad \circ (\id_{\omega(x)} \otimes \id_M \otimes \partial_{x^{\vee} \otimes y}(N))
      \circ (\id_{\omega(x)} \otimes \id_M \otimes \omega^{(2)}_{x^{\vee}, y} \otimes \id_N) \\
    & \qquad \circ (\id_{\omega(x)} \otimes \id_M \otimes \zeta_x \otimes \id_{\omega(y)} \otimes \id_N)
  \end{align*}
  holds for all $x, y \in \mathcal{D}$. Graphically, this equation is expressed as follows:
  \begin{equation}
    \label{eq:T-left-fusion-inv}
    \renewcommand{\PICSCALE}{.9}
    \PIC{left-fusion-inv-01}
    \ = \PIC{left-fusion-inv-02}
  \end{equation}
  Figure~\ref{fig:fusion-inverse-verification-1} shows that the equation
  \begin{align*}
    (\overline{H}{}^{\ell}_{M,N} H^{\ell}_{M,N} \otimes \id \otimes \id) \circ \mbox{}
    & (\partial_{x}(M \otimes T(N)) \otimes \id_{\omega(y)}) \circ (\id_{\omega(x)} \otimes \id_M \otimes \partial_y(N)) \\
    = \, & (\partial_{x}(M \otimes T(N)) \otimes \id_{\omega(y)}) \circ (\id_{\omega(x)} \otimes \id_M \otimes \partial_y(N))
  \end{align*}
  holds for all $x, y \in \mathcal{D}$.
  Figure~\ref{fig:fusion-inverse-verification-2} shows that the equation
  \begin{equation*}
    (H^{\ell}_{M,N} \overline{H}{}^{\ell}_{M,N} \otimes \id_{\omega(y)})
    \circ (i_x(M) \otimes \partial_y(N)) = i_x(M) \otimes \partial_y(N)
  \end{equation*}
  holds for all $x, y \in \mathcal{D}$. Thus both $\overline{H}{}^{\ell}_{M,N} H^{\ell}_{M,N}$ and $H^{\ell}_{M,N} \overline{H}{}^{\ell}_{M,N}$ are the identity morphisms. The proof is done.
\end{proof}
\begin{figure}
  \renewcommand{\PICSCALE}{.75}
  \addtolength{\jot}{12pt}
  \makebox[\textwidth]{
    \begin{minipage}{1.15\linewidth}
      \begin{gather*}
        \PIC{left-fusion-inv-11}
        \mathop{=}^{\text{\eqref{eq:T-left-fusion}}}
        \PIC{left-fusion-inv-12}
        \mathop{=}^{\text{\eqref{eq:gra-cal-duality}}}
        \PIC{left-fusion-inv-13} \\
        \mathop{=}^{\text{\eqref{eq:T-left-fusion-inv}}}
        \PIC{left-fusion-inv-14}
        \mathop{=}^{\substack{\text{\eqref{eq:gra-cal-duality}} \\
            \text{\eqref{eq:gra-cal-omega-unit}} \\
            \text{\eqref{eq:gra-cal-omega-slimy}}}}
        \PIC{left-fusion-inv-15}
        \mathop{=}^{\substack{\text{\eqref{eq:gra-cal-duality}} \\
            \text{\eqref{eq:gra-cal-omega-unit}} \\
            \text{\eqref{eq:gra-cal-omega-slimy}}}}
        \PIC{left-fusion-inv-16} \\[-5pt]
        \mathop{=}^{\text{\eqref{eq:gra-cal-omega-slimy}}}
        (\partial_{x}(M \otimes T(N)) \otimes \id_{\omega(y)})
        \circ (\id_{\omega(x)} \otimes \id_M \otimes \partial_y(N))
      \end{gather*}
      \caption{Verification of $\overline{H}{}^{\ell}_{M,N} H^{\ell}_{M,N} = \id_{T(M \otimes T(N))}$}
      \label{fig:fusion-inverse-verification-1}
    \end{minipage}} \par
  \bigskip\bigskip
  \makebox[\textwidth]{
    \begin{minipage}{1.15\linewidth}
      \begin{gather*}
        \PIC{left-fusion-inv-21}
        \mathop{=}^{\text{\eqref{eq:T-left-fusion-inv}}}
        \PIC{left-fusion-inv-22}
        \mathop{=}^{\text{\eqref{eq:T-left-fusion}}}
        \PIC{left-fusion-inv-23} \\
        \mathop{=}^{\substack{
            \text{\eqref{eq:gra-cal-omega-unit}} \\
            \text{\eqref{eq:gra-cal-omega-slimy}}}}
        \PIC{left-fusion-inv-24}
        \mathop{=}^{\substack{
            \text{\eqref{eq:gra-cal-omega-unit}} \\
            \text{\eqref{eq:gra-cal-omega-slimy}}}}
        \PIC{left-fusion-inv-25}
        \mathop{=}^{\substack{
            \text{\eqref{eq:gra-cal-omega-slimy}}}}
        \PIC{left-fusion-inv-26}
      \end{gather*}
      \caption{Verification of $H^{\ell}_{M,N} \overline{H}{}^{\ell}_{M,N} = \id_{T(M) \otimes T(N)}$}
      \label{fig:fusion-inverse-verification-2}
    \end{minipage}}
\end{figure}

\begin{remark}
  \label{rem:bialgebra-structure-antipode}
  Suppose that $\mathcal{V}$ has a lax braiding $\sigma$. Then, as we have already proved, $B := T(\unitobj)$ has a structure of a bialgebra in $\mathcal{V}$ such that $T \cong B \otimes (-)$ as bimonads. Since $T$ is a left Hopf monad, the bialgebra $B$ is in fact a Hopf algebra by \cite[Proposition 5.4]{MR2793022}. In the same way as \cite[Chapter 5]{MR1862634}, one can check that the morphism $S: B \to B$ in $\mathcal{V}$ determined by
  \begin{equation*}
    \renewcommand{\PICSCALE}{.8}
    S \circ i_{\unitobj}(x) = \PIC{antipode}
    \qquad (x \in \mathcal{D})
  \end{equation*}
  is the antipode of $B$.
\end{remark}

\subsection{Proof of Theorem~\ref{thm:FRT-bimonad} (\ref{item:main-thm-D-left-rigid})}

In this subsection, we assume that $\mathcal{D}$ is left rigid and show that the bimonad $T$ is a right Hopf monad.
We begin with the following easy observation for coends: Let $\mathcal{A}$ and $\mathcal{B}$ categories, and let $G: \mathcal{A}^{\op} \times \mathcal{A} \to \mathcal{B}$ be a functor. Since $(\mathcal{A}^{\op})^{\op} = \mathcal{A}$, we have a functor
\begin{equation*}
  \overline{G}: (\mathcal{A}^{\op})^{\op} \times \mathcal{A}^{\op} \to \mathcal{B},
  \quad (x, y) \mapsto G(y, x).
\end{equation*}
A coend of $G$ (if it exists) is also a coend of $\overline{G}$ with the same universal dinatural transformation. In this sense, we write
\begin{equation}
  \label{eq:coend-over-op-cat}
  \int^{x \in \mathcal{A}} G(x,x) = \int^{x \in \mathcal{A}^{\op}} \overline{G}(x,x).
\end{equation}

\begin{proof}[Proof of Theorem~\ref{thm:FRT-bimonad} (\ref{item:main-thm-D-left-rigid})]
  We consider the strong monoidal functor
  \begin{equation*}
    \rho: \mathcal{D}^{\op} \to \mathcal{V}^{\rev},
    \quad \rho(x) = \omega(x)^{\vee}
    \quad (x \in \mathcal{D}^{\op}).
  \end{equation*}
  Given an object $X \in \mathcal{V}$, we denote by $X^{\dagger}$ a left dual object of $X$ in $\mathcal{V}$ (if it exists), which is the same thing as a right dual object of $X \in \mathcal{V}^{\rev}$. There is a canonical isomorphism $\rho(y) \otimes^{\rev} M \otimes^{\rev} \rho(x)^{\dagger} \cong \omega(x) \otimes M \otimes \omega(y)^{\vee}$ for $x, y \in \mathcal{D}$ and $M \in \mathcal{V}$. Thus, by \eqref{eq:coend-over-op-cat}, the coend
  \begin{equation*}
    \overline{T}(M) := \int^{x \in \mathcal{D}^{\op}} \rho(x) \otimes^{\rev} M \otimes^{\rev} \rho(x)^{\dagger}
  \end{equation*}
  exists and it is canonically isomorphic to $T(M)$.

  Since $\mathcal{D}$ is assumed to be left rigid, $\mathcal{D}^{\op}$ is right rigid.
  By Theorem~\ref{thm:FRT-bimonad} (\ref{item:main-thm-D-left-rigid}), the bimonad $\overline{T}$ on $\mathcal{V}^{\rev}$ is a left Hopf monad. Since the right fusion operator of the bimonad $T$ on $\mathcal{V}$ is equal to the left fusion operator of the bimonad $\overline{T}$ on $\mathcal{V}^{\rev}$, we conclude that $T$ is a right Hopf monad on $\mathcal{V}$. The proof is done.
\end{proof}

\subsection{Proof of Theorem~\ref{thm:FRT-bimonad} (\ref{item:main-thm-D-braided})}
\label{subsec:proof-main-thm-D-braided}

In this subsection, we assume that $\mathcal{D}$ has a lax braiding $\sigma$ and aim to equip a lax braiding to the monoidal category $\lcomod{T}$. Given an object $x \in \mathcal{D}$, we set $\Gamma_x = \omega(x) \otimes T(\omega(x)^{\vee})$. By the assumption that the tensor product of $\mathcal{V}$ is cocontinuous, the object $\Gamma_x$ is in fact a coend
\begin{equation*}
  \Gamma_x = \int^{z \in \mathcal{D}} \omega(x) \otimes \omega(z) \otimes \omega(x)^{\vee} \otimes \omega(z)^{\vee}
\end{equation*}
with the universal dinatural transformation
\begin{equation*}
  \id_{\omega(x)} \otimes i_z(\omega(x)^{\vee}):
  \omega(x) \otimes \omega(z) \otimes \omega(x)^{\vee} \otimes \omega(z)^{\vee} \to \Gamma_x
  \quad (z \in \mathcal{D}).
\end{equation*}
By the universal property, we define $r_x: \Gamma_x \to \unitobj$ to be the unique morphism in $\mathcal{V}$ such that the equation
\begin{equation}
  \label{eq:T-coqt-r-def-1}
  r_x \circ (\id_{\omega(x)} \otimes i_z(\omega(x)^{\vee}))
  = \eval_{\omega(z),\omega(x)}^{(2)} \circ (\Sigma_{x, z} \otimes \id_{\omega(x)^{\vee}} \otimes \id_{\omega(z)^{\vee}})
\end{equation}
holds for all $z \in \mathcal{D}$, where $\eval^{(2)}_{\omega(z), \omega(x)}$ is given by~\eqref{eq:rigid-mon-cat-eval-2} and
\begin{equation}
  \label{eq:T-coqt-Sigma-def}
  \Sigma_{x,z} = (\omega^{(2)}_{z, x})^{-1} \circ \omega(\sigma_{x, z}) \circ \omega^{(2)}_{x, z}
  : \omega(x) \otimes \omega(z) \to \omega(z) \otimes \omega(x).
\end{equation}
In the rest of this subsection, we usually denote an object of $\lcomod{T}$ by a bold face capital letter, such as $\mathbf{M}$. The underlying object of $\mathbf{M} \in \lcomod{T}$ and the coaction of $T$ on $\mathbf{M}$ are expressed by $M$ and $\rho_M: M \otimes (-) \to T(-) \otimes M$, respectively, by using the corresponding letter. Given $\mathbf{M} \in \lcomod{T}$, we define $a_M: T(M) \to M$ to be the unique morphism such that the equation
\begin{equation}
  \label{eq:T-coqt-action-def}
  a_M \circ i_x(M) = (r_x \otimes \id_{M}) \circ (\id_{\omega(x)} \otimes \rho_M(\omega(x)^{\vee}))
\end{equation}
holds for all objects $x \in \mathcal{D}$. For $\mathbf{M}, \mathbf{N} \in \lcomod{T}$, we define
\begin{equation}
  \label{eq:T-coqt-lax-br-def}
  \widetilde{\sigma}_{\mathbf{M}, \mathbf{N}} = (a_N \otimes \id_{M}) \circ \rho_M(N).
\end{equation}
In what follows, we will show that $\widetilde{\sigma} = \{ \widetilde{\sigma}_{\mathbf{M}, \mathbf{N}}: \mathbf{M} \otimes \mathbf{N} \to \mathbf{N} \otimes \mathbf{M} \}$ is a lax braiding of $\lcomod{T}$.

\begin{remark}
  \label{rem:lax-braiding-on-omega-x}
  The pair $\boldsymbol{\omega}(x) := (\omega(x), \partial_x)$ is a $T$-comodule for all $x \in \mathcal{D}$. The lax braiding $\widetilde{\sigma}$ is an `extension' of $\Sigma$ in the sense that the equation $\widetilde{\sigma}_{\boldsymbol{\omega}(x), \boldsymbol{\omega}(y)} = \Sigma_{x,y}$ holds for all $x, y \in \mathcal{D}$.
\end{remark}

We first prove some technical identities involving the structure morphisms of $T$ and the morphism $r_x$.
For $x, y \in \mathcal{D}$, we set
\begin{gather}
  \label{eq:T-coqt-r-def-3}
  r^{(2)}_{x,y} = r_x \circ (\id_{\omega(x)} \otimes r_y \otimes \id_{T(\omega(x)^{\vee})}), \\
  \label{eq:T-coqt-r-def-4}
  r^{\sharp}_x = (\id_{\omega(x)^{\vee}} \otimes r_x) \circ (\coev_{\omega(x)} \otimes \id_{T(\omega(x)^{\vee})}).
\end{gather}

\begin{lemma}
  For all objects $x, y \in \mathcal{D}$, the following equations hold:
  \begin{gather}
    \label{eq:T-coqt-r-1}
    r_{x \otimes y} \circ (\omega^{(2)}_{x, y} \otimes T(\check{\omega}^{(2)}_{x, y}))
    = r^{(2)}_{x,y}
    \circ (\id_{\omega(x)} \otimes \id_{\omega(y)} \otimes \delta_{\omega(y)^{\vee}, \omega(x)^{\vee}}). \\
    \label{eq:T-coqt-r-2}
    r_{\unitobj} \circ (\omega_0 \otimes T(\check{\omega}_{0})) = \varepsilon. \\
    \label{eq:T-coqt-r-3}
    r_x \circ (\id_{\omega(x)} \otimes \mu_{\omega(x)^{\vee}})
    = r_x \circ (\id_{\omega(x)} \otimes T(r^{\sharp}_x)). \\
    \label{eq:T-coqt-r-4}            
    r_x \circ (\id_{\omega(x)} \otimes \eta_{\omega(x)^{\vee}})
    = \eval_{\omega(x)}.
  \end{gather}
\end{lemma}
\begin{proof}
  For $x, y \in \mathcal{D}$, there are canonical bijections
  \begin{align*}
    & \Hom_{\mathcal{V}}(\omega(x) \otimes \omega(y) \otimes T(\omega(y)^{\vee} \otimes \omega(x)^{\vee}), \unitobj) \\
    & \textstyle \cong \int_{z \in \mathcal{D}}
      \Hom_{\mathcal{V}}(\omega(x) \otimes \omega(y)
      \otimes \omega(z) \otimes \omega(y)^{\vee} \otimes \omega(x)^{\vee} \otimes \omega(z)^{\vee}, \unitobj) \\
    & \textstyle \cong \int_{z \in \mathcal{D}}
      \Hom_{\mathcal{V}}(\omega(x) \otimes \omega(y) \otimes \omega(z),
      \omega(z) \otimes \omega(x) \otimes \omega(y)) \\
    & \textstyle \cong \int_{z \in \mathcal{D}}
      \Hom_{\mathcal{V}}(\omega(x \otimes y) \otimes \omega(z),
      \omega(z) \otimes \omega(x \otimes y))
  \end{align*}
  and, under this bijection, one can check that the both sides of \eqref{eq:T-coqt-r-1} correspond to the element whose $z$-component is $\Sigma_{x \otimes y, z}$. The both sides of \eqref{eq:T-coqt-r-2} correspond to the identity natural transformation under $\Hom_{\mathcal{V}}(T(\unitobj), \unitobj) \cong \Nat(\omega, \omega)$. There are also canonical bijections
  \begin{align*}
    & \Hom_{\mathcal{V}}(\omega(x) \otimes T^2(\omega(x)^{\vee}), \unitobj) \\
    & \textstyle \cong \int_{y, z \in \mathcal{D}}
      \Hom_{\mathcal{V}}(\omega(x) \otimes \omega(y) \otimes \omega(z)
      \otimes \omega(x)^{\vee} \otimes \omega(z)^{\vee} \otimes \omega(y)^{\vee}, \unitobj) \\
    & \textstyle \cong \int_{y, z \in \mathcal{D}}
      \Hom_{\mathcal{V}}(\omega(x) \otimes \omega(y \otimes z),
      \omega(y \otimes z) \otimes \omega(x))
  \end{align*}
  and the both sides of \eqref{eq:T-coqt-r-3} correspond to the element whose $(y,z)$-component is the morphism $\Sigma_{x, y \otimes z}$. Equation \eqref{eq:T-coqt-r-4} is verified straightforwardly.
\end{proof}

The family $\{ r_x \}_{x \in \mathcal{D}}$ of morphisms plays a role like a universal R-form for a bialgebra.
Lemmas~\ref{lem:bimonad-lax-br-1} and~\ref{lem:bimonad-lax-br-2} below are an analogue of the fact that a universal R-form on a bialgebra $B$ turns a $B$-comodule into a $B$-module, and this construction gives rise to a strict monoidal functor from the category of $B$-comodules to the category of $B$-modules.

\begin{lemma}
  \label{lem:bimonad-lax-br-1}
  For all $\mathbf{M} \in \lcomod{T}$, the pair $(M, a_M)$ is a $T$-module.
\end{lemma}
\begin{proof}
  For all objects $x, y \in \mathcal{D}$, we have
  \allowdisplaybreaks
  \begin{align*}
    & a_M \circ \mu_M \circ i^{(2)}_{x,y}(M) \\
    {}^{\eqref{eq:FRT-bimonad-mu}}\!
    & = a_M \circ i_{x \otimes y}(M) \circ (\omega^{(2)}_{x,y} \otimes \id_M \otimes \check{\omega}^{(2)}_{x,y}) \\
    {}^{\eqref{eq:T-coqt-action-def}}\!
    & = r_{x \otimes y} \circ (\id_{\omega(x \otimes y)} \otimes \rho_M(\omega(x \otimes y)^{\vee}))
      \circ (\omega^{(2)}_{x,y} \otimes \id_M \otimes \check{\omega}^{(2)}_{x,y}) \\
    {}^{\eqref{eq:T-coqt-r-1}}\!
    & = \begin{aligned}[t]
      (r_{x,y}^{(2)} \otimes \id_M)
      & \circ (\id_{\omega(x)} \otimes \id_{\omega(y)} \otimes \delta_{\omega(x)^{\vee}, \omega(y)^{\vee}} \otimes \id_M) \\
      & \circ (\id_{\omega(x)} \otimes \id_{\omega(y)} \otimes \rho_M(\omega(y)^{\vee} \otimes \omega(x)^{\vee})) \\
    \end{aligned} \\
    {}^{\eqref{eq:T-comod-def-1}}\!
    & = \begin{aligned}[t]
      (r_{x,y}^{(2)} \otimes \id_M)
      & \circ (\id_{\omega(x)} \otimes \id_{\omega(y)} \otimes \id_{T(\omega(y)^{\vee}} \otimes \rho_M(\omega(x)^{\vee}) \\
      & \circ (\id_{\omega(x)} \otimes \id_{\omega(y)} \otimes \rho_M(\omega(y)^{\vee}) \otimes \id_{\omega(x)})
    \end{aligned} \\
    {}^{\eqref{eq:T-coqt-action-def}}\!
    & = (r_x \otimes \id_M) \circ (\id_{\omega(x)} \otimes \rho_M(\omega(x)^{\vee}))
      \circ (\id_{\omega(x)} \otimes (a_M \circ i_y(M)) \otimes \omega(x)^{\vee}) \\
    {}^{\eqref{eq:T-coqt-action-def}}\!
    & = a_M \circ T(a_M) \circ i^{(2)}_{x,y}(M).
  \end{align*}
  This shows $a_M \circ \mu_M = a_M \circ T(a_M)$. We also have
  \begin{gather*}
    a_M \circ \eta_M
    \mathop{=}^{\eqref{eq:T-coqt-action-def}}
    (r_{\unitobj} \otimes \id_M)
    \circ (\id_{\omega(\unitobj)} \otimes \rho_M(\omega(\unitobj)^{\vee}))
    \circ (\omega_0 \otimes \id_M \otimes \check{\omega}_0) \\
    \mathop{=}^{\eqref{eq:T-coqt-r-2}} (\varepsilon \otimes \id_M) \circ \rho_M(\unitobj)
    \mathop{=}^{\eqref{eq:T-comod-def-2}} \id_M. \qedhere
  \end{gather*}
\end{proof}

By the above lemma, we have a functor
\begin{equation*}
  \Phi : \lcomod{T} \to \lmod{T},
  \quad \mathbf{M} \mapsto (M, a_M).
\end{equation*}

\begin{lemma}
  \label{lem:bimonad-lax-br-2}
  The functor $\Phi$ is strict monoidal.
\end{lemma}
\begin{proof}
  Let $\mathbf{M}$ and $\mathbf{N}$ be two $T$-comodules.
  For all $x \in \mathcal{D}$, we have
  \allowdisplaybreaks
  \begin{align*}
    & (a_M \otimes a_N) \circ \delta_{M,N} \circ i_{x}(M \otimes N) \\
    \mathop{}^{\eqref{eq:FRT-bimonad-delta}, \eqref{eq:T-coqt-action-def}} \!
    & = \begin{aligned}[t]
      & (r_x \otimes \id_M \otimes r_x \otimes \id_N)
      \circ (\id_{\omega(x)} \otimes \rho_M(\omega(x)^{\vee}) \otimes \id_{\omega(x)} \otimes \rho_N(\omega(x)^{\vee})) \\
      & \qquad \circ (\id_{\omega(x)} \otimes \id_M \otimes \coev_{\omega(x)} \otimes \id_N \otimes \id_{\omega(x)^{\vee}})
    \end{aligned} \\
    \mathop{}^{\eqref{eq:T-coqt-r-def-4}} \!
    & = \begin{aligned}[t]
      & (r_x \otimes \id_M \otimes \id_N)
      \circ (\id_{\omega(x)} \otimes \rho_M(\omega(x)^{\vee}) \otimes \id_N) \\
      & \qquad \circ (\id_{\omega(x)} \otimes \id_M \otimes r^{\sharp}_x \otimes \id_N)
      \circ (\id_{\omega(x)} \otimes \id_M \otimes \rho_N(\omega(x)^{\vee}))
    \end{aligned} \\
    \mathop{}^{\eqref{eq:T-comod-def-0}} \!
    & = \begin{aligned}[t]
      & (r_x \otimes \id_M \otimes \id_N)
      \circ (\id_{\omega(x)} \otimes T(r^{\sharp}_x) \otimes \id_{M} \otimes \id_{N}) \\
      & \qquad \circ (\id_{\omega(x)} \otimes \rho_M(\omega(x)^{\vee}) \otimes \id_N)
      \circ (\id_{\omega(x)} \otimes \id_M \otimes \rho_N(\omega(x)^{\vee}))
    \end{aligned} \\
    \mathop{}^{\eqref{eq:T-comod-tensor-def}, \eqref{eq:T-coqt-r-3}} \!
    & = (r_x \otimes \id_M \otimes \id_N)
      \circ (\id_{\omega(x)} \otimes \rho_{M \otimes N}(\omega(x)^{\vee}) \\
    \mathop{}^{\eqref{eq:T-coqt-action-def}} \!
    & = a_{M \otimes N} \circ i_x(M \otimes N).
  \end{align*}
  Thus $\Phi(\mathbf{M} \otimes \mathbf{N}) = \Phi(\mathbf{M}) \otimes \Phi(\mathbf{N})$ as $T$-modules. We also have
  \begin{gather*}
    a_{\unitobj} \circ i_x(\unitobj)
    \mathop{=}^{\eqref{eq:T-coqt-action-def}}
    r_x \circ \rho_{\unitobj}(\omega(x)^{\vee})
    \mathop{=}^{\eqref{eq:T-comod-unit-def}}
    r_x \circ (\id_{\omega(x)} \otimes \eta_{\omega(x)^{\vee}})    
    \mathop{=}^{\eqref{eq:T-coqt-r-4}}
    \eval_{\omega(x)}
    \mathop{=}^{\eqref{eq:FRT-bimonad-eps}} \varepsilon \circ i_x(\unitobj)
  \end{gather*}
  for all $x \in \mathcal{D}$. Thus $\Phi(\unitobj_{\lcomod{T}}) = \unitobj_{\lmod{T}}$. The proof is done.
\end{proof}

Now let $\mathbf{M}$ and $\mathbf{N}$ be $T$-comodules. We prove:

\begin{lemma}
  \label{lem:T-coqt-sigma-colinearity}
  The morphism $\widetilde{\sigma}_{\mathbf{M}, \mathbf{N}}$ is a morphism in $\lcomod{T}$.
\end{lemma}
\begin{proof}
  To prove this lemma, we first show that the equation
  \begin{equation}
    \label{eq:T-coqt-lemma-4-3}
    \begin{aligned}
      & (r_x \otimes \mu_V) \circ (\id_{\omega(x)} \otimes \delta_{\omega(x)^{\vee}, T(V)}) \circ (\id_{\omega(x)} \otimes T(i^{\sharp}_x(V))) \\
      & \qquad = \mu_V \circ i_{x}(T(V)) \circ (\id_{\omega(x)} \otimes \id_{T(V)} \otimes r^{\sharp}_x)
      \circ (\id_{\omega(x)} \otimes \delta_{V, \omega(x)^{\vee}})
    \end{aligned}
  \end{equation}
  holds for all objects $x \in \mathcal{D}$, where
  \begin{equation}
    \label{eq:T-coqt-lemma-4-2}
    i^{\sharp}_x(V) = (\id_{\omega(x)^{\vee}} \otimes i_x(V)) \circ (\coev_{\omega(x)} \otimes \id_{V} \otimes \id_{\omega(x)^{\vee}}).
  \end{equation}
  Let $P$ and $Q$ be the left and the right hand side of \eqref{eq:T-coqt-lemma-4-3}, respectively.  Since the tensor product of $\mathcal{V}$ preserves colimits, we have isomorphisms
  \begin{align*}
    & \Hom_{\mathcal{V}}(\omega(x) \otimes T(V \otimes \omega(x)^{\vee}), T(V)) \\
    & \qquad \cong \textstyle \int_{y \in \mathcal{D}}
      \Hom_{\mathcal{V}}(\omega(x) \otimes \omega(y) \otimes V \otimes \omega(x)^{\vee} \otimes \omega(y)^{\vee}, T(V)) \\
    & \qquad \cong \textstyle \int_{y \in \mathcal{D}}
      \Hom_{\mathcal{V}}(\omega(x) \otimes \omega(y) \otimes V, T(V) \otimes \omega(y) \otimes \omega(x)).
  \end{align*}
  Given a morphism $f: \omega(x) \otimes T(V \otimes \omega(x)^{\vee}) \to T(V)$ in $\mathcal{V}$, we denote by
  \begin{equation*}
    \Xi(f)_y: \omega(x) \otimes \omega(y) \otimes V
    \to T(V) \otimes \omega(y) \otimes \omega(x)
  \end{equation*}
  the $y$-th component of the element corresponding to $f$.
  We prove $\Xi(P)_y = \Xi(Q)_y$ for all $y \in \mathcal{D}$ as in Figure~\ref{fig:T-coqt-lemma-3-proof}.
  This shows \eqref{eq:T-coqt-lemma-4-3}.

  \allowdisplaybreaks
  To save space, we denote the both sides of \eqref{eq:FRT-bimonad-delta} by $i_x(M, N)$. Then,
  \begin{align*}
    & (\mu_V \otimes \id_N) \circ \rho_N(T(V)) \circ (a_N \otimes \id_{T(V)}) \circ \delta_{N, V} \circ i_x(N \otimes V) \\
    {}^{\eqref{eq:FRT-bimonad-delta}}\!
    & = (\mu_V \otimes \id_N) \circ \rho_N(T(V)) \circ (a_N \otimes \id_{T(V)}) \circ i_x(N, V) \\
    {}^{\eqref{eq:T-comod-def-1}, \eqref{eq:T-coqt-lemma-4-2}}\!
    & = \begin{aligned}[t]
      & (r_x \otimes \mu_V \otimes \id_N) \circ (\id_{\omega(x)} \otimes \delta_{\omega(x)^{\vee}, T(V)} \otimes \id_N) \\
      & \qquad \circ (\id_{\omega(x)} \otimes \rho_N(\omega(x)^{\vee} \otimes T(V))) \circ (\id_{\omega(x)} \otimes \id_N \otimes i^{\sharp}_x(V))        
    \end{aligned} \\
    {}^{\eqref{eq:T-coqt-lemma-4-3}}\!
    & = \begin{aligned}[t]
      & ((\mu_V \, i_{x}(T(V)) \, (\id_{\omega(x)} \otimes \id_{T(V)} \otimes r^{\sharp}_x)
      \, (\id_{\omega(x)} \otimes \delta_{V, \omega(x)^{\vee}})) \otimes \id_N) \\
      & \qquad \circ (\id_{\omega(x)} \otimes \rho_N(V \otimes \omega(x)^{\vee}))
    \end{aligned} \\
    {}^{\eqref{eq:T-comod-def-1}}\!
    & = \begin{aligned}[t]
      & ((\mu_V \circ i_{x}(T(V)) \circ (\id_{\omega(x)} \otimes \id_{T(V)} \otimes r^{\sharp}_x)) \otimes \id_N) \\
      & \qquad \circ (\id_{\omega(x)} \otimes \id_{T(V)} \otimes \rho_N(\omega(x)^{\vee}))
      \circ (\id_{\omega(x)} \otimes \rho_N(V) \otimes \id_{\omega(x)^{\vee}}) \\
    \end{aligned} \\
    {}^{\eqref{eq:T-coqt-action-def}, \eqref{eq:T-coqt-r-def-4}}\!
    & = (\mu_V \otimes a_N) \circ i_x(T(V), N) \circ (\id_{\omega(x)} \otimes \rho_N(V) \otimes \id_{\omega(x)^{\vee}}) \\
    {}^{\eqref{eq:FRT-bimonad-delta}}\!
    & = (\mu_V \otimes a_N) \circ \delta_{T(V), N} \circ i_x(T(V) \otimes N) \circ (\id_{\omega(x)} \otimes \rho_N(V) \otimes \id_{\omega(x)^{\vee}}) \\
    & = (\mu_V \otimes a_N) \circ \delta_{T(V), N} \circ T(\rho_N(V)) \circ i_x(N \otimes V).
  \end{align*}
  Thus we have
  \begin{equation}
    \label{eq:T-coqt-lemma-4-1}
    \begin{gathered}
      (\mu_V \otimes \id_N) \circ \rho_N(T(V)) \circ (a_N \otimes \id_{T(V)}) \circ \delta_{N, V} \\
      = (\mu_V \otimes a_N) \circ \delta_{T(V), N} \circ T(\rho_N(V)).
    \end{gathered}
  \end{equation}
  For all objects $V \in \mathcal{V}$, we have
  \begin{equation*}
    \renewcommand{\PICSCALE}{.8}
    \PIC{coqt-sigma-colin-32}
    \ \mathop{=}^{\eqref{eq:T-comod-def-1}}
    \PIC{coqt-sigma-colin-34}
    \mathop{=}^{\eqref{eq:T-coqt-lemma-4-1}}
    \ \PIC{coqt-sigma-colin-35}
    \mathop{=}^{\substack{ \eqref{eq:T-comod-def-1} \\ \eqref{eq:T-comod-def-0}}}
    \ \PIC{coqt-sigma-colin-37}
  \end{equation*}
  If we denote by $\rho_{\mathbf{X} \otimes \mathbf{Y}}$ the coaction of $T$ on $\mathbf{X} \otimes \mathbf{Y}$ for $\mathbf{X}, \mathbf{Y} \in \lcomod{T}$, then the above computation shows that the equation
  \begin{equation*}
    \rho_{\mathbf{N} \otimes \mathbf{M}}(V) \circ (\widetilde{\sigma}_{\mathbf{M}, \mathbf{N}} \otimes \id_V)
    = (\id_{T(V)} \otimes \widetilde{\sigma}_{\mathbf{M}, \mathbf{N}}) \circ \rho_{\mathbf{M} \otimes \mathbf{N}}(V)
  \end{equation*}
  holds for all $V \in \mathcal{V}$. The proof is done.
\end{proof}

\begin{figure}
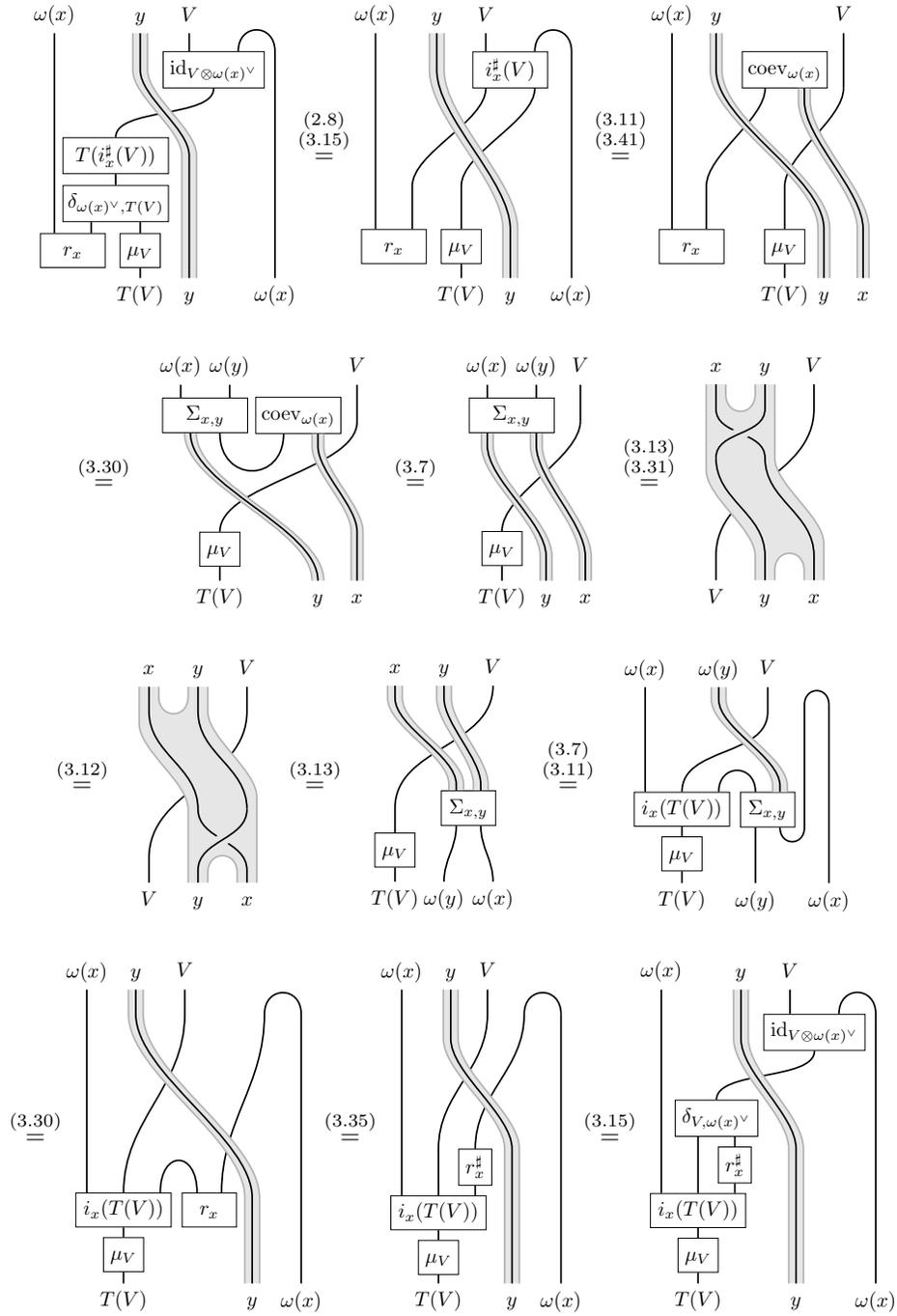

  \begin{minipage}{1.0\linewidth}
    \addtolength{\jot}{1.5em}
    \renewcommand{\PICSCALE}{.8}
    \begin{gather*}
      \PIC{coqt-sigma-colin-01}
      \mathop{=}^{\substack{\eqref{eq:T-comod-def-0} \\ \eqref{eq:FRT-bimonad-alter-delta}}}
      \PIC{coqt-sigma-colin-02}
      \mathop{=}^{\substack{\eqref{eq:partial-def} \\ \eqref{eq:T-coqt-lemma-4-2}}}
      \PIC{coqt-sigma-colin-03} \\
      \mathop{=}^{\eqref{eq:T-coqt-r-def-1}}
      \quad \PIC{coqt-sigma-colin-04}
      \quad \mathop{=}^{\eqref{eq:gra-cal-duality}}
      \quad \PIC{coqt-sigma-colin-05}
      \quad \mathop{=}^{\substack{\eqref{eq:FRT-bimonad-alter-mu} \\ \eqref{eq:T-coqt-Sigma-def}}}
      \quad \PIC{coqt-sigma-colin-06} \\
      \mathop{=}^{\eqref{eq:gra-nat-partial}}
      \quad \PIC{coqt-sigma-colin-07}
      \quad \mathop{=}^{\eqref{eq:FRT-bimonad-alter-mu}}
      \quad \PIC{coqt-sigma-colin-08}
      \quad \mathop{=}^{\substack{\eqref{eq:gra-cal-duality} \\ \eqref{eq:partial-def}}}
      \quad \PIC{coqt-sigma-colin-09} \\
      \mathop{=}^{\eqref{eq:T-coqt-r-def-1}}
      \PIC{coqt-sigma-colin-10}
      \mathop{=}^{\eqref{eq:T-coqt-r-def-4}}
      \PIC{coqt-sigma-colin-11}
      \mathop{=}^{\eqref{eq:FRT-bimonad-alter-delta}}
      \PIC{coqt-sigma-colin-12}
    \end{gather*}
    \caption{Verification of equation \eqref{eq:T-coqt-lemma-4-3}}
    \label{fig:T-coqt-lemma-3-proof}
  \end{minipage}
\end{figure}

By the definition of morphisms in $\lcomod{T}$, it is easy to see that the morphism $\widetilde{\sigma}_{\mathbf{M}, \mathbf{N}}$ is natural in $\mathbf{M}$. The fact that $\Phi$ is a functor implies that it is also natural in the variable $\mathbf{N}$. Hence we obtain a natural transformation
\begin{equation*}
  \widetilde{\sigma} = \{ \widetilde{\sigma}_{\mathbf{M}, \mathbf{N}}: \mathbf{M} \otimes \mathbf{N} \to \mathbf{N} \otimes \mathbf{M} \}_{\mathbf{M}, \mathbf{N} \in \lcomod{T}}.
\end{equation*}

\begin{proof}[Proof of Theorem~\ref{thm:FRT-bimonad} (\ref{item:main-thm-D-braided})]
  We show that $\widetilde{\sigma}$ is a lax braiding of $\lcomod{T}$. We denote by $\mathbf{1}$ the unit object of $\lcomod{T}$. Since $\Phi$ is strong monoidal, the equations
  \begin{equation*}
    \widetilde{\sigma}_{\mathbf{L}, \mathbf{M} \otimes \mathbf{N}}
    = (\id_M \otimes \widetilde{\sigma}_{\mathbf{L}, \mathbf{N}})
    \circ (\widetilde{\sigma}_{\mathbf{L}, \mathbf{M}} \otimes \id_N)
    \quad \text{and} \quad
    \widetilde{\sigma}_{\mathbf{M}, \mathbf{1}} = \id_M
  \end{equation*}
  hold for all objects $\mathbf{L}, \mathbf{M}, \mathbf{N} \in \lcomod{T}$.
  We also have
  \begin{align*}
    \widetilde{\sigma}_{\mathbf{L} \otimes \mathbf{M}, \mathbf{N}}
    \ {}^{\eqref{eq:T-coqt-lax-br-def}} \! =
    & \ (a_N \mu_N \otimes \id_{L} \otimes \id_{M}) \circ (\rho_L(T(N)) \otimes \id_{M}) \circ (\id_{L} \otimes \rho_{M}(N)) \\
    {}^{\eqref{eq:def-monad-module-1}} \! =
    & \ (a_N T(a_N) \otimes \id_{L} \otimes \id_{M}) \circ (\rho_L(T(N)) \otimes \id_{M}) \circ (\id_{L} \otimes \rho_{M}(N)) \\
    {}^{\eqref{eq:T-comod-def-0}} \! =
    & \ ((a_N \otimes \id_L) \rho_L(N) \otimes \id_M)
      \circ (\id_L \otimes (a_N \otimes \id_M) \rho_{M}(N)) \\
    {}^{\eqref{eq:T-coqt-lax-br-def}} \! =
    & \ (\widetilde{\sigma}_{\mathbf{M}, \mathbf{N}} \otimes \id_M)
      \circ (\id_L \otimes \widetilde{\sigma}_{\mathbf{M}, \mathbf{N}}).
  \end{align*}
  The equation $\widetilde{\sigma}_{\mathbf{1}, \mathbf{M}} = \id_M$ is easily verified. Thus $\widetilde{\sigma}$ is a lax braiding of $\lcomod{T}$. The proof is done.
\end{proof}

\subsection{Proof of Theorem~\ref{thm:FRT-bimonad-universality}}

In this subsection, we prove the universal property of the bimonad $T$. By the definition of the structure morphisms of $T$, we easily prove the following lemma:

\begin{lemma}
  \label{lem:T-coact-omega-x}
  The functor $\widetilde{\omega}: \mathcal{D} \to \lcomod{T}$ defined by $\widetilde{\omega}(x) =  (\omega(x), \partial_x)$ is a strong monoidal functor with the same structure morphism as $\omega$.
\end{lemma}

Given a bimonad $S$ on $\mathcal{V}$, we denote by $U^S$ the forgetful functor from $\lcomod{S}$ to $\mathcal{V}$. It is obvious that $U^T \circ \widetilde{\omega} = \omega$ as monoidal functors. We now give a proof of Theorem~\ref{thm:FRT-bimonad-universality}, which states that the pair $(T, \widetilde{\omega})$ is universal among such a pair.

\begin{proof}[Proof of Theorem~\ref{thm:FRT-bimonad-universality}]
  Let $T' = (T', \mu', \eta', \delta', \varepsilon')$ be a bimonad on $\mathcal{V}$, and let $\omega': \mathcal{D} \to \lcomod{T'}$ be a strong monoidal functor such that $U^{T'} \circ \omega' = \omega$ as monoidal functors. We can write $\omega'(x) = (\omega(x), \beta_x)$ for some natural transformation
  \begin{equation*}
    \beta_x(M): \omega(x) \otimes M \to T'(M) \otimes \omega(x)
    \quad (M \in \mathcal{V}).
  \end{equation*}
  We first prove the `uniqueness' part of this theorem. If $\phi: T \to T'$ is a morphism of bimonads such that $\phi_{\sharp} \circ \widetilde{\omega} = \omega'$, then, by the definition of $\widetilde{\omega}$ and $\phi_{\sharp}$, the morphism $\phi_M$ ($M \in \mathcal{V}$) satisfies the equation
  \begin{equation}
    \label{eq:T-universal-proof-1}
    \beta_x(M)
    = (\phi_M \otimes \id_{\omega(x)}) \circ \partial_x(M)
  \end{equation}
  for all $x \in \mathcal{D}$. By Lemma~\ref{lem:partial-universality}, this equation uniquely determines $\phi_M$.

  Now we show the `existence' part of this theorem. In view of the above discussion, we shall define $\phi: T \to T'$ by \eqref{eq:T-universal-proof-1}. The assumption that $U^{T'} \circ \omega' = \omega$ as monoidal functors implies that
  \begin{equation*}
    \omega^{(2)}_{x, y}: \omega'(x) \otimes \omega'(y) \to \omega'(x \otimes y)
    \quad \text{and} \quad
    \omega_0: \unitobj_{\lcomod{T'}} \to \omega'(\unitobj)
  \end{equation*}
  are morphisms in $\lcomod{T'}$. Thus we have
  \begin{gather}
    \label{eq:T-universal-proof-2}
    \begin{aligned}
      & (\mu'_M \otimes \id_{\omega(x)} \otimes \id_{\omega(y)})
      \circ (\beta_x(T'(M)) \otimes \id_{\omega(y)})
      \circ (\id_{\omega(x)} \otimes \beta_y(M)) \\
      & \qquad = (\id_{T'(M)} \otimes (\omega^{(2)}_{x, y})^{-1})
      \circ \beta_{x \otimes y}(M) \circ (\omega^{(2)}_{x, y} \otimes \id_M),
    \end{aligned} \\
    \label{eq:T-universal-proof-3}
    \eta'_M = (\id_{T'(M)} \otimes \omega_0^{-1})
    \circ \beta_{\unitobj}(M) \circ (\omega_0 \otimes \id_M)
  \end{gather}
  for $x, y \in \mathcal{D}$ and $M \in \mathcal{V}$. By the definition of comodules, we also have
  \begin{gather}
    \label{eq:T-universal-proof-4}
    (\delta'_{M,N} \otimes \id_{\omega(x)})
    \circ \beta_{x}(M \otimes N)
    = (\id_{T'(M)} \otimes \beta_x(N))
    \circ (\beta_x(M) \otimes \id_N), \\
    \label{eq:T-universal-proof-5}
    (\varepsilon' \otimes \id_{\omega(x)}) \circ \beta_x(\unitobj)
    = \id_{\omega(x)}
  \end{gather}
  for $x \in \mathcal{D}$ and $M, N \in \mathcal{V}$. Equations \eqref{eq:T-universal-proof-2}--\eqref{eq:T-universal-proof-5} imply that $\phi$ is a morphism of bimonads. Indeed, the left-hand side of~\eqref{eq:T-universal-proof-2} is computed as follows:
  \begin{equation*}
    \renewcommand{\PICSCALE}{.8}
    \PIC{univ-phi-01}
    \mathop{=}^{\eqref{eq:T-universal-proof-1}}
    \PIC{univ-phi-02}
    \ \mathop{=}^{\eqref{eq:gra-nat-partial}}
    \ \PIC{univ-phi-03}
  \end{equation*}
  On the other hand, the right-hand side of~\eqref{eq:T-universal-proof-2} is computed as follows:
  \begin{equation*}
    \renewcommand{\PICSCALE}{.8}
    \PIC{univ-phi-11}
    \ \mathop{=}^{\eqref{eq:T-universal-proof-1}}
    \PIC{univ-phi-12}
    \mathop{=}^{\eqref{eq:FRT-bimonad-alter-mu}}
    \PIC{univ-phi-13}
  \end{equation*}
  Thus we have $\mu'_M \circ \phi_{T(M)} \circ T(\phi_M) = \phi_M \circ \mu_M$. One can easily derive
  \begin{equation*}
    \eta'_M = \phi \circ \eta_M,
    \quad
    \delta'_{M,N} \circ (\phi_M \otimes \phi_N)
    = \phi \circ \delta_{M,N}
    \quad \text{and} \quad
    \varepsilon = \varepsilon' \circ \phi
  \end{equation*}
  from \eqref{eq:T-universal-proof-3}, \eqref{eq:T-universal-proof-4} and \eqref{eq:T-universal-proof-5}, respectively, in a similar way. Hence $\phi: T \to T'$ is a morphism of bimonads. The proof is done.
\end{proof}

\section{Bialgebroids and their universal R-forms}
\label{sec:bialgebroids}

\subsection{Bialgebroids}
\label{subsec:bialgebroids}

In this section, we first review basic results on modules and comodules over bialgebroids from the viewpoint of the theory of bimonad. We then introduce the notion of a (lax) {\em universal R-form} on a left bialgebroid. The main result states that a (lax) universal R-form on a left bialgebroid gives rise to a (lax) braiding of the category of left comodules (Theorem \ref{thm:univ-R-lax-braiding}).

Throughout this section, we fix a commutative ring $k$.
By a $k$-algebra, we mean a ring $R$ equipped with a ring homomorphism $k \to R$, called the unit map, whose image is contained in the center of $R$.
A $k$-algebra is always viewed as a $k$-module through the unit map.
We also fix a $k$-algebra $A$, which is not necessarily commutative, and denote by ${}_A \Mod_A$ the monoidal category of $A$-bimodules. This category is often identified with the category of left modules over the enveloping algebra $A^{\env} := A \otimes_k A^{\op}$.

Given a $k$-algebra $R$, an $R$-ring (see \cite{MR2553659}) is a $k$-algebra $B$ equipped with a $k$-algebra homomorphism $R \to B$ (whose image is not necessarily contained in the center of $B$). Thus an $A^{\env}$-ring is the same thing as a triple $(B, \src, \tgt)$ consisting of a $k$-algebra $B$ and two $k$-algebra homomorphisms $\src: A \to B$ and $\tgt: A^{\op} \to B$ such that the following equation holds:
\begin{equation}
  \label{eq:Re-ring}
  \src(a) \tgt(a') = \tgt(a') \src(a)
  \quad (a \in A, a' \in A^{\op}).
\end{equation}
Unless otherwise noted, we view an $A^{\env}$-ring $B$ as an $A$-bimodule by the left action $\lhd$ and the right action $\rhd$ defined by
\begin{equation}
  \label{eq:bialgebroid-Re-action-left}
  a \rhd b \lhd a' = \src(a) \tgt(a') b
  \quad (a, a' \in A, b \in B).
\end{equation}
For an $A^{\env}$-ring $B$, Takeuchi's $\times_A$-product \cite{MR0506407} is defined by
\begin{equation*}
  B \times_A B := \{ X \in B \otimes_A B \mid \text{$X^1 \tgt(a) \otimes_A X^2 = X^1 \otimes_A X^2 \src(a)$ for all $a \in A$} \},
\end{equation*}
where an element $X \in B \otimes_A B$ is written loosely as $X = X^1 \otimes_A X^2$ despite that it is not a simple tensor in general.

A left bialgebroid over $A$ is first introduced by Takeuchi \cite{MR0506407} under the name of a $\times_A$-bialgebra.
We adopt the following definition, which is different but equivalent to Takeuchi's.

\begin{definition}[see \cite{MR2553659}]
  A {\em left bialgebroid} over $A$ is a data $(B, \src, \tgt, \Delta, \pi)$ such that the following conditions (B1)--(B4) are satisfied:
  \begin{enumerate}
    \renewcommand{\labelenumi}{(B\arabic{enumi})}
  \item The triple $(B, \src, \tgt)$ is an $A^{\env}$-ring.
  \item The triple $(B, \Delta, \pi)$ is a coalgebra object in ${}_A \Mod_A$.
  \item The image of $\Delta$ is contained in $B \times_A B$, which is in fact a $k$-algebra by the component-wise multiplication, and the corestriction $\Delta: B \to B \times_A B$ is a homomorphism of $k$-algebras.
  \item The following equations hold:
    \begin{equation}
      \label{eq:bialgebroid-counit-1}
      \pi(1_B) = 1_A,
      \quad \pi(b \tgt(\pi(b')))
      = \pi(b b')
      = \pi(b \src(\pi(b')))
      \quad (b, b' \in B).
    \end{equation}
  \end{enumerate}
\end{definition}

Let $B = (B, \src, \tgt, \Delta, \pi)$ be a left bialgebroid over $A$.
The maps $\Delta$ and $\pi$ are called the {\em comultiplication} and the {\em counit} of $B$, respectively.
The maps $\src$ and $\tgt$ are called the {\em source} and the {\em target} map, respectively.

The condition (B2) means that $\Delta: B \to B \otimes_A B$ and $\pi: B \to A$ are homomorphisms of $A$-bimodules and the equations
\begin{gather}
  \label{eq:bialgebroid-coassoci}
  \Delta(b_{(1)}) \otimes_A b_{(2)} = b_{(1)} \otimes_A \Delta(b_{(2)}), \\
  \label{eq:bialgebroid-counital}
  \pi(b_{(1)}) \rhd b_{(2)} = b = b_{(1)} \lhd \pi(b_{(2)})
\end{gather}
hold for all $b \in B$, where $b_{(1)} \otimes_A b_{(2)} = \Delta(b)$ is the comultiplication of $b \in B$ in the Sweedler notation. In view of \eqref{eq:bialgebroid-coassoci}, we write
\begin{equation}
  \Delta(b_{(1)}) \otimes_A b_{(2)}
  = b_{(1)} \otimes_A b_{(2)} \otimes_A b_{(3)}
  = b_{(1)} \otimes_A \Delta(b_{(2)})
  \quad (b \in B).
\end{equation}
The condition (B3) means that the following equations hold:
\begin{gather}
  \label{eq:bialgebroid-comul-cross-A}
  b_{(1)} \tgt(a) \otimes_A b_{(2)} = b_{(1)} \otimes_A b_{(2)} \src(a), \\
  \Delta(b b') = b_{(1)} b'_{(1)} \otimes_A b_{(2)} b'_{(2)},
  \quad \Delta(1_B) = 1_B \otimes_A 1_B.
\end{gather}
Since the counit $\pi: B \to A$ is a morphism of $A$-bimodules, we have $\pi(\src(a)) = \pi(a \rhd 1_B) = a \pi(1_B) = a$ for all $a \in A$. The equation $\pi(\tgt(a)) = a$ for $a \in A$ is proved in a similar way. By these observations and~\eqref{eq:bialgebroid-counit-1}, we have
\begin{equation}
  \label{eq:bialgebroid-counit-2}
  \pi(b \src(a)) = \pi(b \tgt(a))
\end{equation}
for all $b \in B$ and $a \in A$. Since $\Delta$ is a morphism of $A$-bimodules, we also have
\begin{equation}
  \Delta(a \rhd b \lhd a') = (a \rhd b_{(1)}) \otimes_A (b_{(2)} \lhd a')
\end{equation}
for all $b \in B$ and $a, a' \in A$.

\subsection{Modules over a bialgebroid}
\label{subsec:bialgebroid-modules}

Let $B$ be a left bialgebroid over $A$. Then the $k$-linear category ${}_B \Mod$ of left $B$-modules is a monoidal category. To describe its monoidal structure, we first note that $B$ is an $A^{\env}$-bimodule by the left $A^{\env}$-action given by \eqref{eq:bialgebroid-Re-action-left} and the right $A^{\env}$-action $\blacktriangleleft$ given by
\begin{equation}
  \label{eq:bialgebroid-Re-action-right}
  b \blacktriangleleft (a \otimes a')
  = b \src(a) \tgt(a')
  \quad (b \in B, a \in A, a' \in A^{\op}).
\end{equation}
Thus we have a $k$-linear endofunctor $T = B \otimes_{A^{\env}}(-)$ on ${}_A \Mod_A$ ($= {}_{A^e} \Mod$). We now define $A$-bimodule maps $\mu_X$, $\eta_X$, $\delta_{X,Y}$ and $\varepsilon$ as follows:
\begin{gather*}
  \mu_X: TT(X) \to T(X),
  \quad \mu_X(b \otimes_{A^{\env}} b' \otimes_{A^{\env}} x) = b b' \otimes_{A^{\env}} x, \\
  \eta_X: X \to T(X),
  \quad \eta_X(x) = 1 \otimes_{A^{\env}} x, \\
  \delta_{X,Y}: T(X \otimes_A Y) \to T(X) \otimes_A T(Y), \\
  \quad \delta_{X,Y}(b \otimes_{A^{\env}} (x \otimes_A y))
  = (b_{(1)} \otimes_{A^{\env}} x) \otimes_A (b_{(2)} \otimes_{A^{\env}} y), \\
  \varepsilon: T(A) \to A,
  \quad \varepsilon(b \otimes_{A^{\env}} a) = \pi(b \src(a)), \\
  (X, Y \in {}_A \Mod_A, a \in A, b, b' \in B, x \in X, y \in Y).
\end{gather*}
Szlach\'anyi observed that $(T, \mu, \eta, \delta, \varepsilon)$ is a bimonad on ${}_A \Mod_A$ and, moreover, proved that any bimonad on ${}_A \Mod_A$ admitting a right adjoint is obtained from a left bialgebroid over $A$ in this way \cite[Section 4]{MR1984397}. A module over the monad $T$ constructed in the above is nothing but a left $B$-module in the usual sense. Thus ${}_B \Mod$ is a $k$-linear monoidal category as the category of modules over a $k$-linear bimonad.

\subsection{Comodules over a bialgebroid}

Since the monoidal category ${}_A \Mod_A$ acts on the category ${}_A \Mod$ from the left by $\otimes_A$, a coalgebra object in ${}_A \Mod_A$ gives rise to a comonad on ${}_A \Mod$. Let $B$ be a left bialgebroid over $A$. We define the category ${}^B \Mod$ of left $B$-comodules to be the category of comodules over the comonad on ${}_A \Mod$ arising from the coalgebra object $(B, \Delta, \pi)$ in ${}_A \Mod_A$. Stated differently,

\begin{definition}
  We regard $B$ as an $A$-bimodule by \eqref{eq:bialgebroid-Re-action-left}. A {\em left $B$-comodule} is a pair $(M, \delta_M)$ consisting of a left $A$-module $M$ and a morphism $\delta_M: M \to B \otimes_A M$ in ${}_A \Mod$ such that the equations
  \begin{gather}
    \label{eq:def-bialgebroid-comod}
    \Delta(m_{(-1)}) \otimes_A m_{(0)}
    = m_{(-1)} \otimes_A \delta_M(m_{(0)}),
    \quad \pi(m_{(-1)}) m_{(0)} = m
  \end{gather}
  hold for all $m \in M$, where $m_{(-1)} \otimes_A m_{(0)} = \delta_M(m)$. Given two left $A$-comodules $(M, \delta_M)$ and $(N, \delta_N)$, a morphism of left $A$-comodules from $(M, \delta_M)$ to $(N, \delta_N)$ is a morphism $f: M \to N$ in ${}_A \Mod$ such that the equation
  \begin{equation}
    \label{eq:def-bialgebroid-comod-morphism}
    (\id_B \otimes_A f) \circ \delta_M = \delta_N \circ f
  \end{equation}
  hold. We denote by ${}^B \Mod$ the category of left $B$-comodules and morphisms between them.
\end{definition}

The category ${}^B \Mod$ is in fact a $k$-linear monoidal category admitting a `forgetful' functor to ${}_A \Mod_A$. This fact is not obvious unlike the case of ordinary bialgebras since an object of ${}^B \Mod$ is only a left $A$-module in nature. For a left $B$-comodule $M$, Hai \cite{MR2448024} introduced the following right action of $A$:
\begin{equation}
  \label{eq:bialgebroid-comod-right-act}
  m a
  := \pi(m_{(-1)} \src(a)) m_{(0)}
  \mathop{=}^{\eqref{eq:bialgebroid-counit-2}} \pi(m_{(-1)} \tgt(a)) m_{(0)}
  \quad (m \in M, a \in A).
\end{equation}
This action is well-defined and makes $M$ an $A$-bimodule. Moreover, the image of the coaction $\delta_M: M \to B \otimes_A M$ is contained in the set
\begin{equation*}
  B \times_A M := \{ X \in B \otimes_A M \mid
  \text{$X^1 \tgt(a) \otimes_A X^2 = X^1 \otimes_A X^2 a$ for all $a \in A$} \},
\end{equation*}
where $X \in B \otimes_A M$ is written as $X = X^1 \otimes_A X^2$. In other words, we have
\begin{equation}
  \label{eq:bialgebroid-comod-1}
  m_{(-1)} \tgt(a) \otimes_A m_{(0)} = 
  m_{(-1)} \otimes_A m_{(0)} a
  \quad (m \in M, a \in A).
\end{equation}
If we make $B \otimes_A M$ an $A$-bimodule by
\begin{equation*}
  a \cdot (b \otimes_A m) \cdot a' = \src(a) b \src(a') \otimes_A m
  \quad (a, a' \in A, b \in B, m \in M),
\end{equation*}
then $B \times_A M$ is a subbimodule and the coaction $\delta_M: M \to B \times_A M$ is in fact a morphism of $A$-bimodules. Namely, we have
\begin{gather}
  \label{eq:bialgebroid-comod-2}
  \delta_M(a m a') = \src(a) m_{(-1)} \src(a') \otimes_A m_{(0)}
\end{gather}
for $a, a' \in A$ and $m \in M$. The above observation shows that our definition of a left $B$-comodule agrees with that used, {\it e.g.}, in \cite{MR1800718,MR3572267}. Namely, we have:

\begin{lemma}
  A left $B$-comodule is the same thing as an $A$-bimodule $M$ equipped with an $A$-bimodule map $\delta_M: M \to B \times_A M$ satisfying~\eqref{eq:def-bialgebroid-comod}.
\end{lemma}

In what follows, we always view a left $B$-comodule as an $A$-bimodule by the right action given by \eqref{eq:bialgebroid-comod-right-act}. If $M, N \in {}^B \Mod$, then the tensor product $M \otimes_A N$ is a left $A$-comodule by the coaction given by
\begin{equation*}
  m \otimes_A n \mapsto m_{(-1)} n_{(-1)} \otimes_A m_{(0)} \otimes_A n_{(0)}
  \quad (m \in M, n \in N),
\end{equation*}
which is well-defined by \eqref{eq:bialgebroid-comod-1} and~\eqref{eq:bialgebroid-comod-2}.
The left $A$-module $A$ is the `trivial' left $B$-comodule by the coaction given by $a \mapsto \src(a) \otimes_A 1$ ($a \in A$).
The category ${}^B \Mod$ is a $k$-linear monoidal category by this tensor product and the unit object $A$ \cite[Corollary 1.7.2]{MR2448024}.

The bialgebroid $B$ itself is a left $B$-comodule by the left $A$-action \eqref{eq:bialgebroid-Re-action-left} and the coaction $\Delta$. When we construct the tensor product comodule of the form $B \otimes_A M$, we must use the right $A$-action on $B$ given by \eqref{eq:bialgebroid-comod-right-act}, which is different from \eqref{eq:bialgebroid-Re-action-left} in general. As this treatment of the comodule $B$ may be confusing, we note the following lemma:

\begin{lemma}
  \label{lem:left-regular-B-comod}
  The $k$-module $C := B$ is a left $B$-comodule by
  \begin{equation*}
    a \rhd c = \src(a) c, \quad \delta_C(c) = \Delta(c) \quad (a \in A, c \in C).
  \end{equation*}
  The right $A$-action~\eqref{eq:bialgebroid-comod-right-act} on $C$, which we denote by $\prec$, is given by
  \begin{equation}
    \label{eq:bialgebroid-comod-right-act-B}
    c \prec a = c \src(a) \quad (b \in B, a \in A).
  \end{equation}
\end{lemma}
\begin{proof}
  It is trivial from the definition of bialgebroids that $C$ is a left $B$-comodule as stated. For the right action of $A$, we compute
  \begin{equation*}
    c \prec a
    \mathop{=}^{\eqref{eq:bialgebroid-comod-right-act}}
    \pi(c_{(1)} \src(a)) \rhd c_{(2)}
    \mathop{=}^{\eqref{eq:bialgebroid-counit-1}}
    \pi(c_{(1)} \tgt(a)) \rhd c_{(2)}
    \mathop{=}^{\eqref{eq:bialgebroid-comul-cross-A}}
    \pi(c_{(1)}) \rhd c_{(2)} \src(a)
    \mathop{=}^{\eqref{eq:bialgebroid-counital}} c \src(a)
  \end{equation*}
  for $a \in A$ and $c \in C$. The proof is done.
\end{proof}

We have introduced the category of $T$-comodules for a bimonad $T$. We justify our terminology by proving the following theorem:

\begin{theorem}
  \label{thm:T-comod-is-B-comod}
  Let $B$ be a left bialgebroid over $A$, and let $T = B \otimes_{A^{\env}} (-)$ be the corresponding bimonad on ${}_A \Mod_A$. Then there is an isomorphism
  \begin{equation}
    \lcomod{T} \cong {}^B \Mod
  \end{equation}
  of $k$-linear monoidal categories.
\end{theorem}
\begin{proof}
  Let ${}_B \Mod \to {}_A \Mod_A$ be the forgetful functor. Then ${}^B\Mod$ is isomorphic to $\lZlax(U)$ as a monoidal category \cite[Proposition 5.1]{MR3572267}, and hence $\lcomod{T}$ is isomorphic to ${}^B\Mod$ by Theorem~\ref{lem:Z-lax-and-comonoidal-adj}.
\end{proof}

\begin{remark}
  \label{rem:T-comod-is-B-comod}
  The object of $\lcomod{T}$ corresponding to $M \in {}^B\Mod$ is the pair $(M, \rho_M)$, where $\rho_M(X)$ for $X \in {}_A \Mod_A$ is the $A$-bimodule homomorphism
  \begin{equation}
    \label{eq:T-comod-is-B-comod-pf-1}
    \begin{aligned}
      \rho_M(X) : M \otimes_A X & \to T(X) \otimes_A M, \\
      m \otimes_A x & \mapsto (m_{(-1)} \otimes_{A^{\env}} x) \otimes_A m_{(0)}.
    \end{aligned}
  \end{equation}
\end{remark}

\subsection{Hopf algebroids}

Given a left bialgebroid $B$ over $A$, we define
\begin{equation}
  \label{eq:bialgebroid-Galois-map}
  \beta: B_{\tgt} \mathop{\otimes}_{A^{\op}} {}_{\tgt} B \to {}_{\tgt} B \mathop{\otimes}_A {}_{\src} B,
  \quad b \mathop{\otimes}_{A^{\op}} b' \mapsto b_{(1)} \mathop{\otimes}_A b_{(2)} b'.
\end{equation}
Here, ${}_{\tgt}B$ is the $k$-module $B$ regraded as a left $A^{\op}$-module through $\tgt$, which is also regarded as a right $A$-module in an obvious way. The right $A^{\op}$-module $B_{\tgt}$ and the left $A$-module ${}_{\src}B$ are defined analogously.
The map $\beta$ is called the {\em Galois map}.

\begin{definition}[Schauenburg \cite{MR1800718}]
  A {\em left Hopf algebroid} over $A$ is a left bialgebroid over $A$ whose Galois map is invertible.
\end{definition}

Let $T$ be the bimonad on ${}_A \Mod_A$ associated to a left bialgebroid $B$. As remarked in \cite{MR2793022}, the Galois map~\eqref{eq:bialgebroid-Galois-map} is identified with the left fusion operator $H^{\ell}_{M,N}$ of $T$ at $M = N = A^{\env}$. Since $A^{\env}$ is a projective generator of ${}_A \Mod_A$, we see that $B$ is a left Hopf algebroid if and only if $T$ is a left Hopf monad on ${}_A \Mod_A$ \cite[Proposition 7.2]{MR2793022}.

We warn that the above definition is not equivalent to the antipode style definition of Hopf algebroids given in \cite{MR2553659}. A concrete counterexample ({\it i.e.}, a left Hopf algebroid in the above sense which is not a Hopf algebroid in the sense of \cite{MR2553659}) is given in \cite{MR3366558}.

\subsection{Universal R-forms}

A coquasitriangular bialgebra over $k$ is a bialgebra $B$ equipped with a universal R-form $B \times B \to k$. The conditions required for a universal R-form ensure that the category of comodules over $B$ has a braiding. We extend this kind of notion to bialgebroids as follows:

\newcommand{\rform}{\mathbf{r}}

\begin{definition}
  \label{def:univ-R}
  Let $B$ be a left bialgebroid over $A$. A {\em lax universal R-form} on $B$ is a $k$-bilinear map $\rform: B \times B \to A$ satisfying the equations
  \begin{gather}
    \label{eq:def-univ-R-1}
    \rform(x, \src(a)y) = a \rform(x, y), \quad
    \rform(\tgt(a) x, y) =  \rform(x, y) a, \\
    \label{eq:def-univ-R-2}
    \rform(x, \tgt(a) y) = \rform(x \tgt(a), y), \quad
    \rform(\src(a) x, y) = \rform(x, y \src(a)), \\
    \label{eq:def-univ-R-4}
    \rform_{\src}(x_{(1)}, y_{(1)}) x_{(2)} y_{(2)}
    = \rform_{\tgt}(x_{(2)}, y_{(2)}) y_{(1)} x_{(1)}, \\
    \label{eq:def-univ-R-5}
    \rform(x, yz) = \rform(\rform_{\src}(x_{(1)}, z) x_{(2)}, y), \\
    \label{eq:def-univ-R-6}
    \rform(x y, z) = \rform(x, \rform_{\tgt}(y, z_{(2)}) z_{(1)}), \\
    \label{eq:def-univ-R-7}
    \rform(1, x) = \pi(x) = \rform(x, 1)
  \end{gather}
  for $a \in A$ and $x, y, z \in B$, where $\rform_{\src} = \src \circ \rform$ and $\rform_{\tgt} = \tgt \circ \rform$.
  A {\em universal R-form} on $B$ is a lax universal R-form $\rform$ on $B$ such that there exists a $k$-bilinear map $\overline{\rform} : B \times B \to A$ satisfying equations analogous to \eqref{eq:def-univ-R-1} and \eqref{eq:def-univ-R-2}, and
  \begin{gather}
    \label{eq:def-strong-univ-R}
    \rform(x_{(1)}, y_{(1)}) \, \overline{\rform}(y_{(2)}, x_{(2)})
    = \pi(x y)
    = \overline{\rform}(x_{(1)}, y_{(1)}) \, \rform(y_{(2)}, x_{(2)})
  \end{gather}
  for all $x, y \in B$.
  A (lax) {\em coquasitriangular left bialgebroid} is a left bialgebroid equipped with a (lax) universal R-form.
\end{definition}

We shall check the well-definedness of some expressions. Let $B$ be a left bialgebroid over $A$, and let $\rform: B \times B \to A$ be a $k$-bilinear map satisfying the equations \eqref{eq:def-univ-R-1} and \eqref{eq:def-univ-R-2}. For all $x_1, x_2, y \in B$ and $a \in A$, we have
\begin{equation*}
  \rform_{\src}(\tgt(a) x_1, y) x_2
  \mathop{=}^{\eqref{eq:def-univ-R-1}}
  \rform_{\src}(x_1, y) \src(a) x_2, \quad
  \rform_{\tgt}(y, \src(a) x_2) x_1
  \mathop{=}^{\eqref{eq:def-univ-R-1}}
  \rform_{\tgt}(y, x_2) \tgt(a) x_1.
\end{equation*}
Thus the expressions $\sum_i \rform_{\src}(x_i', y) x_i''$ and $\sum_i \rform_{\tgt}(y, x_i'') x_i'$ are well-defined for $y \in B$ and $\sum_i x_i \otimes_A x_i'' \in B \otimes_A B$. Hence the right-hand side of~\eqref{eq:def-univ-R-5} and that of \eqref{eq:def-univ-R-6} are well-defined. For all $x, y, y' \in B$ and $a \in A$, we have
\begin{equation*}
  \rform_{\src}(x_{(1)}, \tgt(a) y) x_{(2)} y'
  \mathop{=}^{\eqref{eq:def-univ-R-2}}
  \rform_{\src}(x_{(1)} \tgt(a), y) x_{(2)} y'
  \mathop{=}^{\eqref{eq:bialgebroid-comul-cross-A}} \rform_{\src}(x_{(1)}, y) x_{(2)} \src(a) y'.
\end{equation*}
This shows that the left-hand side of \eqref{eq:def-univ-R-4} is well-defined. We omit the verification of the well-definedness of other expressions in the definition.

\begin{theorem}
  \label{thm:univ-R-lax-braiding}
  Let $B$ be a left bialgebroid over $A$, and let $\rform: B \times B \to A$ be a $k$-bilinear map satisfying \eqref{eq:def-univ-R-1} and \eqref{eq:def-univ-R-2}. Then, for all left $B$-comodules $M$ and $N$, there is the following well-defined left $A$-linear map:
  \begin{equation}
    \label{eq:univ-R-lax-braiding}
    \sigma^{\rform}_{M,N}: M \otimes_A N \to N \otimes_A M,
    \quad m \otimes_A n \mapsto \rform(n_{(-1)}, m_{(-1)}) n_{(0)} \otimes_A m_{(0)}.
  \end{equation}
  The family $\{ \sigma^{\rform}_{M,N} \}$ of $A$-linear maps is a lax braiding on ${}^B\Mod$ if and only if $\rform$ is a lax universal R-form on $B$. If $\rform$ is a universal R-form, then the lax braiding $\sigma^{\rform}$ is invertible, {\it i.e.}, it is a braiding.
\end{theorem}
\begin{proof}
  Let $M$ and $N$ be left $B$-comodules. We first check that the map $\sigma_{M,N}^{\rform}$ is well-defined. For all $a \in A$, $b, b' \in B$, $m \in M$ and $n \in N$, we have
  \begin{equation*}
    \rform(\tgt(a) b, b') n \otimes_A m
    = \rform(b, b') a n \otimes_A m
  \end{equation*}
  by \eqref{eq:def-univ-R-1}. This implies that the expression $\sum_{i} \rform(b_i, b) n_i \otimes_A m$ is well-defined for all elements $b \in B$, $m \in M$ and $\sum_{i} b_i \otimes_A n_i \in B \otimes_A N$. Thus we have
  \begin{align*}
    \rform(n_{(-1)}, \tgt(a) b)n_{(0)} \otimes_A m
    & \mathop{=}^{\eqref{eq:def-univ-R-2}} \rform(n_{(-1)} \tgt(a), b)n_{(0)} \otimes_A m \\
    & \mathop{=}^{\eqref{eq:bialgebroid-comod-1}} \rform(n_{(-1)}, b)n_{(0)} a \otimes_A m
      \mathop{=} \rform(n_{(-1)}, b)n_{(0)} \otimes_A a m
  \end{align*}
  for all $a \in A$, $b \in B$, $n \in N$ and $m \in M$. Hence the map $\sigma^{\rform}_{M,N}$ is defined by the formula \eqref{eq:univ-R-lax-braiding}. The $A$-linearity of this map is shown as follows:
  \begin{align*}
    \sigma^{\rform}_{M,N}(a m \otimes_A n)
    & \mathop{=}^{\eqref{eq:bialgebroid-comod-2}} \rform(n_{(-1)}, \src(a) m_{(-1)}) n_{(0)} \otimes_A m_{(0)} \\
    & \mathop{=}^{\eqref{eq:def-univ-R-1}} a \rform(n_{(-1)}, m_{(-1)}) n_{(0)} \otimes_A m_{(0)}
    = a \sigma^{\rform}_{M,N}(m \otimes_A n)
  \end{align*}
  for $a \in A$, $m \in M$ and $n \in N$. 

  We examine when $\sigma^{\rform}_{M,N}$ is a morphism of $B$-comodules for all left $B$-comodules $M$ and $N$. Given a symbol $X$ expressing a left $B$-comodule, we denote the coaction of $B$ on $X$ by $\delta_X$. Now we set
  $F := (\id_B \otimes_A \sigma^{\rform}_{M,N})\delta_{M \otimes_A N}^{}$
  and $G := \delta_{N \otimes_A M}^{} \sigma^{\rform}_{M,N}$
  for simplicity of notation. Then we have
  \begin{align*}
    F(m \otimes_A n)
    & = m_{(-1)} n_{(-1)} \otimes_A \sigma^{\rform}_{M,N}(m_{(0)} \otimes_A n_{(0)}) \\
    {}^{\eqref{eq:univ-R-lax-braiding}}
    & = m_{(-2)} n_{(-2)} \otimes_A \rform(n_{(-1)}, m_{(-1)}) n_{(0)} \otimes_A m_{(0)} \\
    {}^{\eqref{eq:bialgebroid-Re-action-left}}
    & = \rform_{\tgt}(n_{(-1)}, m_{(-1)}) m_{(-2)} n_{(-2)} \otimes_A n_{(0)} \otimes_A m_{(0)}, \\
    G(m \otimes_A n)
    & = \delta_{N \otimes_A M}^{}(\rform(n_{(-1)}, m_{(-1)}) n_{(0)} \otimes_A m_{(0)}) \\
    {}^{\eqref{eq:bialgebroid-comod-2}}
    & = \rform_{\src}(n_{(-1)}, m_{(-1)}) (n_{(0)})_{(-1)} (m_{(-1)})_{(-1)} \otimes_A (n_{(0)})_{(0)} \otimes_A (m_{(0)})_{(0)} \\
    & = \rform_{\src}(n_{(-2)}, m_{(-2)}) n_{(-1)} m_{(-1)} \otimes_A n_{(0)} \otimes_A m_{(0)}
  \end{align*}
  in $B \otimes_A M \otimes_A N$ for $m \in M$ and $n \in N$. Thus, if $\rform$ satisfies \eqref{eq:def-univ-R-4}, then the map $\sigma_{M,N}^{\rform}$ is a morphism of $B$-comodules for all $M$ and $N$.

  In view of the above argument, we assume that $\rform$ satisfies \eqref{eq:def-univ-R-4}. Then we may ask whether the family $\sigma^{\rform} = \{ \sigma^{\rform}_{M,N} \}$ of morphisms in ${}^B\Mod$ is a lax braiding of ${}^B\Mod$ or not. It is obvious that $\sigma^{\rform}$ is a natural transformation. Now let $L$, $M$ and $N$ be left $B$-comodules. Then we have
  \begin{align*}
    \sigma^{\rform}_{L, M \otimes_A N}(\ell \otimes_A m \otimes_A n)
    & = \rform(m_{(-1)} n_{(-1)}, \ell_{(-1)}) m_{(0)} \otimes_A n_{(0)} \otimes_A \ell_{(0)}, \\
    \sigma^{\rform}_{L \otimes_A M, N}(\ell \otimes_A m \otimes_A n)
    & = \rform(n_{(-1)}, \ell_{(-1)} m_{(-1)}) n_{(0)} \otimes_A \ell_{(0)} \otimes_A m_{(0)}
  \end{align*}
  for all $\ell \in L$, $m \in M$, $n \in N$. We also have     \allowdisplaybreaks
  \begin{align*}
    (\id_M \otimes_A \sigma^{\rform}_{L,N})
    & (\sigma^{\rform}_{L,M} \otimes_A \id_N)(\ell \otimes_A m \otimes_A n) \\
    & = \rform(m_{(-1)}, \ell_{(-1)}) m_{(0)} \otimes_A
      \sigma^{\rform}_{L,N}(\ell_{(0)} \otimes_A n) \\
    & = \rform(m_{(-1)}, \ell_{(-2)}) m_{(0)} \otimes_A
      \rform(n_{(-1)}, \ell_{(-1)}) n_{(0)} \otimes_A \ell_{(0)} \\
    & = \rform(m_{(-1)}, \ell_{(-2)}) m_{(0)} \rform(n_{(-1)}, \ell_{(-1)})
      \otimes_A n_{(0)} \otimes_A \ell_{(0)} \\
    {}^{\eqref{eq:bialgebroid-comod-1}}
    & = \rform(m_{(-1)} \rform_{\tgt}(n_{(-1)}, \ell_{(-1)}), \ell_{(-2)}) m_{(0)}
      \otimes_A n_{(0)} \otimes_A \ell_{(0)}, \\
    (\sigma^{\rform}_{L,N} \otimes_A \id_M)
    & (\id_L \otimes_A \sigma^{\rform}_{M,N}) (\ell \otimes_A m \otimes_A n) \\
    & = \sigma^{\rform}_{L,N}(\ell \otimes_A \rform(n_{(-1)}, m_{(-1)}) n_{(0)}) \otimes_A m_{(0)} \\
    {}^{\eqref{eq:bialgebroid-comod-2}}
    & = \rform(\rform_{\src}(n_{(-2)}, m_{(-1)}) n_{(-1)}, \ell_{(-1)}) n_{(0)} \otimes_A \ell_{(0)} \otimes_A m_{(0)}.
  \end{align*}
  for all $\ell \in L$, $m \in M$, $n \in N$. Finally, we have
  \begin{gather*}
    \allowdisplaybreaks
    \sigma_{M,A}^{\rform}(m \otimes_A a)
    = \sigma_{M,A}^{\rform}(m a \otimes_A 1)
    = \rform(1, (ma)_{(-1)}) \otimes_A (ma)_{(0)}, \\
    \sigma_{A,M}^{\rform}(a \otimes_A m)
    = \sigma_{M,A}^{\rform}(1 \otimes_A a m)
    = \rform((am)_{(-1)}, 1) (a m)_{(0)} \otimes_A 1
  \end{gather*}
  for all $a \in A$ and $m \in M$. Thus $\sigma^{\rform}$ is a lax braiding on ${}^B\Mod$ if $\rform$ satisfies equations \eqref{eq:def-univ-R-1}--\eqref{eq:def-univ-R-7}, that is, if it is a lax universal R-form.

  To prove the converse, we assume that $\sigma^{\rform}$ is a lax braiding of ${}^B\Mod$. Let $C := B$ be the left $B$-comodule given in Lemma~\ref{lem:left-regular-B-comod}. By considering the case where $L = M = N = C$ in the above computation, we obtain the equations
  \begin{align}
    \label{eq:univ-R-pf-1}
    & \rform_{\tgt}(x_{(2)}, y_{(2)}) y_{(1)} x_{(1)} \otimes_A x_{(3)} \otimes_A y_{(3)} \\
    \notag
    & \qquad = \rform_{\src}(x_{(1)}, y_{(1)}) x_{(2)} y_{(2)} \otimes_A x_{(3)} \otimes_A y_{(3)}
    & \text{(in $B \otimes_A C \otimes_A C$)}, \\
    \label{eq:univ-R-pf-2}
    & \rform_{\src}(x_{(1)}, y_{(1)} z_{(1)}) x_{(2)} \otimes_A y_{(2)} \otimes_A z_{(2)} \\
    \notag
    & \quad = \rform_{\src}(\rform_{\src}(x_{(1)}, z_{(1)}) x_{(2)}, y_{(1)}) x_{(3)} \otimes_A y_{(1)} \otimes_A z_{(1)}
    & \text{(in $C \otimes_A C \otimes_A C$)}, \\
    \label{eq:univ-R-pf-3}
    & \rform_{\src}(x_{(1)} y_{(1)}, z_{(1)}) x_{(2)} \otimes_A y_{(2)} \otimes_A z_{(2)} \\
    \notag
    & \quad = \rform_{\src}(x_{(1)} \rform_{\tgt}(y_{(1)}, z_{(2)}), z_{(1)}) x_{(2)} \otimes_A y_{(2)} \otimes_A z_{(3)}
    & \text{(in $C \otimes_A C \otimes_A C$)}, \\
    \label{eq:univ-R-pf-4}
    & \rform_{\src}(1, x_{1})x_{(2)} = x = \rform_{\src}(x_{1}, 1) x_{(2)}
  \end{align}
  for all $x, y, z \in C$.

  Now we demonstrate that these equations imply that $\mathbf{r}$ is a lax universal R-form on $B$.
  As a first step, we deduce \eqref{eq:def-univ-R-4} from \eqref{eq:univ-R-pf-1}.
  Naively, this is done by applying the map `$\id_B \otimes_A \pi \otimes_A \pi$' to the both sides of  \eqref{eq:univ-R-pf-1}. The problem is that the map $\pi : C \to A$ is only a left $A$-module map in general (see Lemma~\ref{lem:left-regular-B-comod} for the right action of $A$ on $C$) and therefore `$\id_B \otimes_A \pi \otimes_A \pi$' does not make sense. Thus, instead, we apply the map
  \begin{equation*}
    \setlength{\arraycolsep}{2pt}
    \begin{array}{ccccc}
      B \otimes_A C \otimes_A C
      & \xrightarrow{\makebox[8em]{\scriptsize $\id_B \otimes_A \id_C \otimes_A \pi$}}
      & B \otimes_A C \otimes_A A
      & \xrightarrow{\quad \cong \quad}
      & B \otimes_A C \\[5pt]
      & \xrightarrow{\makebox[8em]{\scriptsize $\id_B \otimes_A \pi$}}
      & B \otimes_A A
      & \xrightarrow{\quad \cong \quad} & B      
    \end{array}
  \end{equation*}
  to the both sides of \eqref{eq:univ-R-pf-1}.

  For simplicity, we set $\pi_{\src} = \src \circ \pi$ and $\pi_{\tgt} = \tgt \circ \pi$.
  Since $\pi : C \to A$ is a morphism of left $A$-modules, the map $\id_{B} \otimes_A \id_C \otimes_A \pi$ is defined. By applying this map to the left-hand side of \eqref{eq:univ-R-pf-1}, we have
  \begin{align*}
    \rform_{\tgt}(x_{(2)},
    & \, y_{(2)}) y_{(1)} x_{(1)} \otimes_A x_{(3)} \otimes_A \pi(y_{(3)}) \\
    {}^{\eqref{eq:bialgebroid-comod-right-act-B}}
    & = \rform_{\tgt}(x_{(2)}, y_{(2)}) y_{(1)} x_{(1)} \otimes_A x_{(3)} \pi_{\src}(y_{(3)}) \otimes_A 1 \\
    {}^{\eqref{eq:bialgebroid-comul-cross-A}}
    & = \rform_{\tgt}(x_{(2)} \pi_{\tgt}(y_{(3)}), y_{(2)}) y_{(1)} x_{(1)} \otimes_A x_{(3)} \otimes_A 1 \\
    {}^{\eqref{eq:def-univ-R-2}}
    & = \rform_{\tgt}(x_{(2)}, \pi_{\tgt}(y_{(3)}) y_{(2)}) y_{(1)} x_{(1)} \otimes_A x_{(3)} \otimes_A 1 \\
    {}^{\eqref{eq:bialgebroid-counital}}
    & = \rform_{\tgt}(x_{(2)}, y_{(2)}) y_{(1)} x_{(1)} \otimes_A x_{(3)} \otimes_A 1.
  \end{align*}
  By applying $\id_{B} \otimes_A \id_{C} \otimes_A \pi$ to the right-hand side of \eqref{eq:univ-R-pf-1}, we also have
  \begin{equation*}
    \rform_{\src}(x_{(1)}, y_{(1)}) x_{(2)} y_{(2)} \otimes_A x_{(3)} \otimes_A \pi(y_{(3)})
    = \rform_{\src}(x_{(1)}, y_{(1)}) x_{(2)} y_{(2)} \otimes_A x_{(3)} \otimes_A 1.
  \end{equation*}
  Hence, through the isomorphism $B \otimes_A C \otimes_A A \cong B \otimes_A C$, we have
  \begin{equation*}
    \rform_{\tgt}(x_{(2)}, y_{(2)}) y_{(1)} x_{(1)} \otimes_A x_{(3)}
    = \rform_{\src}(x_{(1)}, y_{(1)}) x_{(2)} y_{(2)} \otimes_A x_{(3)}
  \end{equation*}
  in $B \otimes_A C$. By applying $\id_{B} \otimes_A \pi$ to the both sides, we have
  \begin{align*}
    \rform_{\tgt}(x_{(2)}, y_{(2)}) y_{(1)} x_{(1)} \otimes_A \pi(x_{(3)})
    & = \pi_{\tgt}(x_{(3)}) \rform_{\tgt}(x_{(2)}, y_{(2)}) y_{(1)} x_{(1)} \otimes_A 1 \\
    {}^{\eqref{eq:def-univ-R-1}}
    & = \rform_{\tgt}(\pi_{\tgt}(x_{(3)}) x_{(2)}, y_{(2)}) y_{(1)} x_{(1)} \otimes_A 1 \\
    {}^{\eqref{eq:bialgebroid-counital}}
    & = \rform_{\tgt}(x_{(2)}, y_{(2)}) y_{(1)} x_{(1)} \otimes_A 1, \\
    \rform_{\src}(x_{(1)}, y_{(1)}) x_{(2)} y_{(2)} \otimes_A \pi(x_{(3)})
    & = \rform_{\src}(x_{(1)}, y_{(1)}) x_{(2)} y_{(2)} \otimes_A 1
  \end{align*}
  in $B \otimes_A A$. Hence we obtain \eqref{eq:def-univ-R-4}.

  Next we deduce \eqref{eq:def-univ-R-5} from \eqref{eq:univ-R-pf-2}. By applying $\id_{C} \otimes_A \id_C \otimes_A \pi$ to the both sides of \eqref{eq:univ-R-pf-2} and regarding the result as an equation in $C \otimes_A C$, we obtain
  \begin{equation*}
    \rform_{\src}(x_{(1)}, y_{(1)} z) x_{(2)} \otimes_A y_{(2)}
    = \rform_{\src}(\rform_{\src}(x_{(1)}, z) x_{(2)}, y_{(1)}) x_{(3)} \otimes_A y_{(1)}.
  \end{equation*}
  By applying $\id_C \otimes_A \pi$ to the both sides, we obtain the equation
  \begin{equation}
    \label{eq:univ-R-pf-5}
    \rform_{\src}(x_{(1)}, y z) x_{(2)}
    = \rform_{\src}(\rform_{\src}(x_{(1)}, z) x_{(2)}, y) x_{(3)}
  \end{equation}
  in $C$. Since the counit $\pi: C \to A$ is a morphism of left $A$-modules, we have
  \begin{equation*}
    \pi(\rform_{\src}(x_{(1)}, y z) x_{(2)})
    = \rform(x_{(1)}, y z) \pi(x_{(2)})
    \mathop{=}^{\eqref{eq:def-univ-R-1}}
    \rform(\pi_{\tgt}(x_{(2)}) x_{(1)}, y z)
    \mathop{=}^{\eqref{eq:bialgebroid-counital}}
    \rform(x, y z)
  \end{equation*}
  in $A$. In a similar way, we have
  \begin{align*}
    \pi(\rform_{\src}(\rform_{\src}(x_{(1)}, z) x_{(2)}, y) x_{(3)})
    & = \rform(\rform_{\src}(x_{(1)}, z) x_{(2)}, y) \pi( x_{(3)}) \\
    {}^{\eqref{eq:def-univ-R-1}}
    & = \rform(\pi_{\tgt}(x_{(3)}) \rform_{\src}(x_{(1)}, z) x_{(2)}, y) \\
    {}^{\eqref{eq:Re-ring},\eqref{eq:bialgebroid-counital}}
    & = \rform(\rform_{\src}(x_{(1)}, z) x_{(2)}, y)
  \end{align*}
  in $A$. Thus we obtain~\eqref{eq:def-univ-R-5} by applying $\pi$ to \eqref{eq:univ-R-pf-5}.
  Equation \eqref{eq:def-univ-R-6} is deduced from \eqref{eq:univ-R-pf-3} in a similar way as above.
  Equation \eqref{eq:def-univ-R-7} is obtained by applying $\pi$ to \eqref{eq:univ-R-pf-4}.
  Hence we have proved that $\sigma^{\rform}$ is a lax braiding if and only if $\rform$ is a lax universal R-form.

  Finally, we assume that $\rform$ is a universal R-form. Then, by definition, there is a $k$-bilinear map $\overline{\rform} : B \to B \to A$ satisfying \eqref{eq:def-univ-R-1}, \eqref{eq:def-univ-R-2} and \eqref{eq:def-strong-univ-R}. By the first part of this theorem, there is a well-defined left $A$-linear map
  \begin{equation*}
    \sigma^{\overline{\rform}}_{M,N}: M \otimes_A N \to N \otimes_A M,
    \quad m \otimes_A n \mapsto \overline{\rform}(n_{(-1)}, m_{(-1)}) n_{(0)} \otimes_A m_{(0)},
  \end{equation*}
  for $M, N \in {}^B\Mod$. For $m \in M$ and $n \in N$, we have
  \begin{align*}
    \sigma^{\overline{\rform}}_{N,M}\sigma^{\rform}_{M,N}(m \otimes_A n)
    & = \sigma^{\overline{\rform}}_{N,M}(\rform(m_{(-1)},n_{(-1)}) n_{(0)} \otimes_A m_{(0)}) \\
    & = \rform(m_{(-1)},n_{(-1)}) \sigma^{\overline{\rform}}_{N,M}(n_{(0)} \otimes_A m_{(0)}) \\
    & = \rform(m_{(-2)},n_{(-2)}) \overline{\rform}(n_{(-1)}, m_{(-1)}) (m_{(0)} \otimes_A n_{(0)}) \\
    {}^{\eqref{eq:def-strong-univ-R}} \!
    & = \pi(m_{(-1)},n_{(-1)}) m_{(0)} \otimes_A n_{(0)} = m \otimes_A n.
  \end{align*}
  Thus $\sigma^{\overline{\rform}}_{N,M}\sigma^{\rform}_{M,N} = \id_{M \otimes_A N}$.
  One can verify that $\sigma^{\rform}_{M,N}\sigma^{\overline{\rform}}_{N,M}$ is the identity map in a similar way as above. The proof is done.
\end{proof}

\section{Tannaka construction for bialgebroids}
\label{sec:construction-bialgebroids}

\subsection{Duality in ${}_A\Mod_A$}

Throughout this section, we fix a commutative ring $k$ and a $k$-algebra $A$. As we have mentioned in Section~\ref{sec:monadic-FRT}, a construction data over the monoidal category ${}_A \Mod_A$ of $A$-bimodules defines a bimonad on ${}_A \Mod_A$ admitting a right adjoint and hence gives rise to a bialgebroid over $A$. The aim of this section is to give an explicit description of the resulting bialgebroid. Some properties of such a bialgebroid will also be discussed.

We first recall some basic facts on dual objects in ${}_A \Mod_A$. Given $M \in {}_A \Mod_A$, we denote by $M^{\vee}$ the $k$-module of left $A$-module homomorphisms from $M$ to $A$. The $k$-module $M^{\vee}$ is an $A$-bimodule by the actions defined by
\begin{equation*}
  \langle a \cdot f \cdot a', m \rangle = \langle f, m a \rangle a'
  \quad (a, a' \in A, f \in M^{\vee}, m \in M),
\end{equation*}
where $\xi(m) \in A$ for $\xi \in M^{\vee}$ and $m \in M$ is written as $\langle \xi, m \rangle$ for better readability.
There is an $A$-bimodule map
\begin{equation*}
  \eval_M: M \otimes_A M^{\vee} \to A,
  \quad m \otimes_A \xi \mapsto \langle \xi, m \rangle
  \quad (m \in M, \xi \in M^{\vee}),
\end{equation*}
which we call the evaluation map, no matter whether $M$ is a right rigid object of ${}_A \Mod_A$ or not.

It is known that an $A$-bimodule is a right rigid object of ${}_A \Mod_A$ if and only if it is finitely generated and projective as a left $A$-module. For such an $A$-bimodule $M$, there is a unique $A$-bimodule map $\coev_M: A \to M^{\vee} \otimes_A M$ such that the triple $(M^{\vee}, \eval_M, \coev_M)$ is a right dual object of $M$.
Let $m_1, \cdots, m_r$ be a finite number of elements of $M$, and let $m^1, \cdots, m^r$ be the same number of elements of $M^{\vee}$. We call $\{ m_i, m^i \}$ a {\em pair of dual bases} for $M$ if the equation $\coev_M(1_A) = m_i \otimes_A m^i$ holds, where the Einstein convention is used to suppress the sum over $i$. By the definition of a right dual object, we have
\begin{equation}
  \label{eq:dual-basis}
  \langle m^i, m \rangle \cdot m_i = m
  \quad \text{and} \quad
  m^i \cdot \langle \xi, m_i \rangle = \xi
  \quad (m \in M, \xi \in M^{\vee}).
\end{equation}

\subsection{Realization of coends}
\label{subsec:realize-coends}

Our construction of bimonads heavily relies on the notion of coends.
Let $\mathcal{D}$ be a small category, and let $\omega: \mathcal{D} \to {}_A \Mod_A$ be a functor such that $\omega(x)$ is right rigid for all objects $x \in \mathcal{D}$ (no monoidal structure is required at this stage). Since ${}_A \Mod_A$ admits small direct sums, the coend \eqref{eq:FRT-bimonad-def} can be constructed as follows: We fix an object $V \in {}_A \Mod_A$ and define the functor $G$ by
\begin{equation*}
  G: \mathcal{D} \times \mathcal{D}^{\op} \to {}_A \Mod_A,
  \quad G(x, y) = \omega(x) \otimes_A V \otimes_A \omega(y)^{\vee}
  \quad (x, y \in \mathcal{D}).
\end{equation*}
We set $G_f = G(x, y)$ for a morphism $f : x \to y$ in $\mathcal{D}$.
We consider the following two $A$-bimodules:
\begin{equation*}
  \mathcal{G}_0 = \bigoplus_{x \in \Obj(\mathcal{D})} G(x, x)
  \quad \text{and} \quad
  \mathcal{G}_1 = \bigoplus_{f \in \Mor(\mathcal{D})} G_f.
\end{equation*}
For $x \in \Obj(\mathcal{D})$, we denote by $j_x: G(x,x) \to \mathcal{G}_0$ the inclusion map. There are two $A$-bimodule maps $R, R': \mathcal{G}_1 \to \mathcal{G}_0$ such that, for all morphisms $f: x \to y$ in $\mathcal{D}$, the restrictions of $R$ and $R'$ to $G_f \subset \mathcal{G}_1$ are identical to
\begin{equation*}
  G_f
  \xrightarrow{\ G(f, \id_y) \ } G(y, y)
  \xrightarrow{\ j_y \ } \mathcal{G}_0
  \quad \text{and} \quad
  G_f
  \xrightarrow{\ G(\id_x, f) \ } G(x, x)
  \xrightarrow{\ j_x \ } \mathcal{G}_0,
\end{equation*}
respectively. The quotient $T(V) := \mathcal{G}_0 / \Img(R - R')$ is in fact a coend of the functor $G$ with universal dinatural transformation $G(x, x) \to T(V)$ given by the composition of $j_x$ and the quotient map.

For our purpose, it is important to find an $A^{\env}$-bimodule $B$ such that $B \otimes_{A^{\env}} V$ is isomorphic to the coend $T(V)$ constructed in the above. Such an $A^{\env}$-bimodule is given as follows: For $x \in \Obj(\mathcal{D})$, we set
\begin{equation}
  \label{eq:B-omega-E-tilde}
  \widetilde{E}_x = \omega(x) \otimes_k \omega(x)^{\vee}
\end{equation}
and make it an $A^{\env}$-bimodule by the action given by
\begin{equation}
  \label{eq:B-omega-Ae-actions}
  (a_1 \otimes_k a_2^{\op}) \rhd (m \otimes_k \xi) \lhd (a_3 \otimes_k a_4^{\op})
  = (a_1 m a_3) \otimes_k (a_4 \xi a_2)
\end{equation}
for $a_1, a_3 \in A$, $a_2, a_4 \in A^{\op}$, $m \in \omega(x)$ and $\xi \in \omega(x)^{\vee}$. There is an isomorphism
\begin{equation}
  \label{eq:construction-coend-iso-1}
  \omega(x) \otimes_A V \otimes_A \omega(x)^{\vee} \to \widetilde{E}_x \otimes_{A^{\env}} V,
  \quad m \otimes_A v \otimes_A \xi \mapsto (m \otimes_k \xi) \otimes_{A^{\env}} v
\end{equation}
for $V \in {}_A \Mod_A$. We now consider the direct sum
\begin{equation}
  \widetilde{B}_{\omega} = \bigoplus_{x \in \Obj(\mathcal{D})} \widetilde{E}_x
\end{equation}
of $A^{\env}$-bimodules.
For a morphism $f: x \to y$ in $\mathcal{D}$, we define the $k$-linear map
\begin{equation}
  \label{eq:map-Rel-f}
  \Rel_f: \omega(x) \otimes_k \omega(y)^{\vee} \to \widetilde{B}_{\omega}
\end{equation}
by the following formula:
\begin{equation*}
  \Rel_f(m \otimes \xi) = \omega(f)(m) \otimes_k \xi - m \otimes_k \omega(f)^{\vee}(\xi)
  \quad (m \in \omega(x), \xi \in \omega(y)^{\vee}).
\end{equation*}
If we view $\omega(x) \otimes_k \omega(y)^{\vee}$ as an $A^{\env}$-bimodule by the action defined by the same way as~\eqref{eq:B-omega-Ae-actions}, then $\Rel_f$ is a homomorphism of $A^{\env}$-bimodules. Thus
\begin{equation}
  \label{eq:ideal-J-omega}
  J_{\omega} := \sum_{f \in \Mor(\mathcal{D})} \Img(\Rel_f)
\end{equation}
is an $A^{\env}$-subbimodule of $\widetilde{B}_{\omega}$. There is an isomorphism $T(V) \cong (\widetilde{B}_{\omega}/J_{\omega}) \otimes_{A^{\env}} V$ of left $A^{\env}$-modules induced by \eqref{eq:construction-coend-iso-1}.
Now we set $B_{\omega} = \widetilde{B}_{\omega} / J_{\omega}$ and let $q: \widetilde{B}_{\omega} \to B_{\omega}$ be the quotient map. Summarizing the above discussion so far, we have

\begin{lemma}
  \label{lem:construction-coend}
  The left $A^{\env}$-bimodule $B_{\omega} \otimes_{A^{\env}} V$ is a coend
  \begin{equation*}
    B_{\omega} \otimes_{A^{\env}} V
    = \int^{x \in \mathcal{D}} \omega(x) \otimes_A V \otimes_A \omega(x)^{\vee}
  \end{equation*}
  with the universal dinatural transformation given by
  \begin{equation*}
    \omega(x) \otimes_A V \otimes_A \omega(x)^{\vee} \to B_{\omega} \otimes_{A^{\env}} V,
    \quad m \otimes_A v \otimes_A \xi \mapsto q(m \otimes_{k} \xi) \otimes_{A^{\env}} v
  \end{equation*}
  for $x \in \mathcal{D}$, $m \in \omega(x)$, $v \in V$ and $\xi \in \omega(x)^{\vee}$.
\end{lemma}

\subsection{The bialgebroid $B_{\omega}$}
\label{subsec:bialgebroid-B-omega}

Let $(\mathcal{D}, \omega)$ be a construction data over ${}_A \Mod_A$ with $\mathcal{D}$ a small category. We define the $A^{\env}$-bimodule $B_{\omega}$ as above. By Theorem~\ref{thm:construction-bimonad}, the endofunctor $T := B_{\omega} \otimes_{A^{\env}} (-)$ on ${}_A \Mod_A$ has a structure of a bimonad on ${}_A \Mod_A$ and therefore $B_{\omega}$ has a structure of a left bialgebroid over $A$.

Although our category $\mathcal{D}$ is not assumed to be additive, one can apply Tannaka construction theorems of \cite{MR2448024,2009arXiv0907.1578S} to $\omega$. The bialgebroid $B_{\omega}$ is constructed based on this idea. Thus we obtain explicit descriptions of the structure maps of $B_{\omega}$ along the same argument as \cite{2009arXiv0907.1578S}.

The detail is as follows:
As in Section~\ref{sec:monadic-FRT}, we denote by $\omega^{(i)}$ ($i = 0, 2$) the monoidal structure of $\omega$ and define $\check{\omega}^{(i)}$ ($i = 0, 2$) by \eqref{eq:omega-check-def}.
For an object $x \in \mathcal{D}$, we define the $A^{\env}$-bimodule $\widetilde{E}_x$ as in the previous subsection and write
\begin{equation*}
  \wE{m}{\xi} = m \otimes_{k} \xi \in \widetilde{E}_x,
  \quad E_x = q(\widetilde{E}_x)
  \quad \text{and} \quad
  \eE{m}{\xi} = q\left( \wE{m}{\xi} \right)
\end{equation*}
for $m \in \omega(x)$ and $\xi \in \omega(x)^{\vee}$, where $q : \widetilde{B}_{\omega} \to B_{\omega}$ is the quotient map. The universal dinatural transformation for the coend $T(V)$ given in Lemma \ref{lem:construction-coend} is expressed with this notation as follows:
\begin{equation}
  \label{eq:B-omega-universal-dinat}
  i_x(V): \omega(x) \otimes_A V \otimes_A \omega(x)^{\vee} \to T(V),
  \quad m \otimes_A v \otimes_A \xi \mapsto \eE{m}{\xi} \otimes_{A^{\env}} v.
\end{equation}
For $x_i \in \mathcal{D}$, $m_i \in \omega(x_i)$ and $\xi_i \in \omega(x_i)^{\vee}$ ($i = 1, 2$), we set
\begin{equation}
  \label{eq:B-omega-oast}
  m_1 \otimes_{\omega} m_2 = \omega^{(2)}_{x_1,x_2}(m_1 \otimes_A m_2), \quad
  \xi_1 \mathbin{\check{\otimes}}_{\omega} \xi_2 = \check{\omega}^{(2)}_{x_2, x_1}(\xi_2 \otimes_A \xi_1).
\end{equation}
Finally, for $a \in A$, we set $a_{\omega} = \omega_0(a)$ and $\check{a}_{\omega} = \check{\omega}_{0}(a)$.

\begin{theorem}
  The multiplication, the unit, the source map, the target map, the comultiplication and the counit of $B_{\omega}$ are given respectively by
  \begin{gather}
    \label{eq:B-omega-mult}
    \eE{m}{\xi} \cdot \eE{n}{\zeta}
    = \eE{m \otimes_{\omega} n}{\xi \mathbin{\check{\otimes}}_{\omega} \zeta}
    \in E_{x \otimes y}, \\
    \label{eq:B-omega-unit}
    1_{B_{\omega}} = \eE{1_{\omega}}{\check{1}_{\omega}} \in E_{\unitobj}, \\
    \label{eq:B-omega-src-tgt}
    \src(a) = \eE{a_{\omega}}{\check{1}_{\omega}} \in E_{\unitobj},
    \quad
    \tgt(a) = \eE{1_{\omega}}{\check{a}_{\omega}} \in E_{\unitobj}, \\
    \label{eq:B-omega-comul}
    \Delta\left( \eE{m}{\xi} \right)
    = \eE{m}{m^j} \otimes_A \eE{m_j}{\xi}, \\
    \label{eq:B-omega-counit}
    \pi\left( \eE{m}{\xi} \right) = \langle \xi, m \rangle
  \end{gather}
  for $x, y \in \mathcal{D}$, $m \in \omega(x)$, $n \in \omega(y)$, $\xi \in \omega(x)^{\vee}$ and $\zeta \in \omega(y)^{\vee}$.
  In the expression of the comultiplication, $\{ m_j, m^j \}$ is a pair of dual bases for $\omega(x) \in {}_A \Mod_A$ and the Einstein convention is used to suppress the sum over $j$.
\end{theorem}
\begin{proof}
  We reconstruct the structure maps of $B_{\omega}$ from the bimonad $T$ by the way of \cite[Section 4]{MR1984397}.
  Let $\mu$, $\eta$, $\delta$ and $\varepsilon$ be the multiplication, the unit, the comultiplication and the counit of $T$, respectively. There is a canonical isomorphism
  \begin{equation*}
    \Psi: B_{\omega} \to T(A^{\env}),
    \quad b \mapsto b \otimes_{A^{\env}} (1_A \otimes_k 1_{A^{\op}}) \quad (b \in B_{\omega})
  \end{equation*}
  of left $A^{\env}$-modules. The multiplication of $B_{\omega}$ is given by $b \cdot b' = \nabla(b \otimes_{A^{\env}} b')$ for $b, b' \in B_{\omega}$, where $\nabla$ is the unique map making the following diagram commutes:
  \begin{equation*}
    \xymatrix@C=48pt@R=16pt{
      B_{\omega} \otimes_{A^{\env}} B_{\omega}
      \ar[d]_{\nabla}
      \ar[r]^(.425){\id \otimes_{A^{\env}} \Psi}
      & B_{\omega} \otimes_{A^{\env}} B_{\omega} \otimes_{A^{\env}} A^{\env}
      \ar@{=}[r] & T T(A^{\env})
      \ar[d]^{\mu_{A^{\env}}^{}} \\
      B_{\omega}
      \ar[r]^(.425){\Psi}
      & B_{\omega} \otimes_{A^{\env}} A^{\env}
      \ar@{=}[r] & T (A^{\env})
    }
  \end{equation*}
    We define $A$-bimodule maps $\underline{\Delta}: A^{\env} \to A^{\env} \otimes_A A^{\env}$ and $\underline{\varepsilon}: A^{\env} \to A$ by
  \begin{equation*}
    \underline{\Delta}(a \otimes_k a') = (a \otimes_k 1) \otimes_A (1 \otimes_k a')
    \quad \text{and} \quad
    \underline{\varepsilon}(a \otimes a') = a a'
    \quad (a, a' \in A),
  \end{equation*}
  respectively. The comultiplication $\Delta$ and the counit $\pi$ of $B_{\omega}$ are unique maps such that the following diagrams commute:
  \begin{equation*}
    \xymatrix@C=30pt@R=20pt{
      B_{\omega}
      \ar[d]_{\Delta} \ar[r]^(.425){\Psi}
      & T(A^{\env}) \ar[r]^(.425){T(\underline{\Delta})}
      & T(A^{\env} \otimes_A A^{\env}) \ar[d]^{\delta_{A^{\env},A^{\env}}} \\
      B_{\omega} \otimes_A B_{\omega}
      \ar[rr]^(.425){\Psi \otimes_A \Psi}
      & & T(A^{\env}) \otimes_A T(A^{\env})
    }
    \quad
    \xymatrix@C=30pt@R=20pt{
      B_{\omega} \ar[d]_{\pi}
      \ar[r]^{\Psi}
      & T(A^{\env}) \ar[d]^{T(\underline{\varepsilon})} \\
      A & T(A) \ar[l]_{\varepsilon}
    }
  \end{equation*}
  By the expression \eqref{eq:B-omega-universal-dinat} of the universal dinatural transformation, one can verify that the multiplication, the comultiplication, and the counit of $B_{\omega}$ are given by \eqref{eq:B-omega-mult}, \eqref{eq:B-omega-comul} and \eqref{eq:B-omega-counit}, respectively. It is easy to check that the element given by \eqref{eq:B-omega-unit} is the unit of $B_{\omega}$. The expressions for the source and the target maps, \eqref{eq:B-omega-src-tgt}, are verified as follows: For $a \in A$,
  \begin{equation*}
    \src(a) = (a \otimes_k 1) \rhd 1_{B_{\omega}}
    \mathop{=}^{\eqref{eq:B-omega-Ae-actions}} \eE{a}{\check{1}}, \quad
    \tgt(a) = (1 \otimes_k a) \rhd 1_{B_{\omega}}
    \mathop{=}^{\eqref{eq:B-omega-Ae-actions}} \eE{1}{\check{a}}. \qedhere
  \end{equation*}
\end{proof}

\subsection{Universal coaction of $B_{\omega}$}
\label{subsec:universal-coactions}

Till the end of this section, we fix a construction data $(\mathcal{D}, \omega)$ with $\mathcal{D}$ a small category as in the last subsection. Let $B_{\omega}$ be the bialgebroid constructed in the above. By Theorem~\ref{thm:FRT-bimonad-universality}, an object of the form $\omega(x)$, $x \in \mathcal{D}$, has a structure of a left $B_{\omega}$-comodule.

\begin{proposition}
  \label{prop:universal-coactions}
  The coaction of $B_{\omega}$ on $\omega(x)$ is given by
  \begin{equation*}
    \omega(x) \to B_{\omega} \otimes_A \omega(x),
    \quad m \mapsto \eE{m}{m^i} \otimes_A m_i
    \quad (m \in \omega(x)),
  \end{equation*}
  where $\{ m_i, m^i \}$ is a pair of dual bases for $M$.
\end{proposition}
\begin{proof}
  We consider the bimonad $T = B_{\omega} \otimes_{A^{\env}}(-)$ on ${}_A \Mod_A$. The left $T$-comodule structure of $\omega(x)$ is given by
  \begin{equation*}
    \partial_x(N) = (i_{x}(N) \otimes_A \omega(x)) \circ (\id_{\omega(x)} \otimes \coev_{\omega(x)})
    \quad (N \in {}_A \Mod_A).
  \end{equation*}
  By~\eqref{eq:B-omega-universal-dinat}, we have
  \begin{equation*}
    \partial_x(N) (m \otimes_A n)
    = \left( \eE{m}{m^i} \otimes_{A^{\env}} n \right) \otimes_A m_i
    \quad (m \in \omega(x), n \in N).
  \end{equation*}
  The formula of the left $B_{\omega}$-coaction is obtained by letting $N = A^{\env}$.
\end{proof}

\subsection{Universal R-form of $B_{\omega}$}
\label{subsec:FRT-lax-universal-r}

In this subsection, we assume that the monoidal category $\mathcal{D}$ has a lax braiding $\sigma$.
For $x, y \in \mathcal{D}$, we write
\begin{equation*}
  \Sigma_{x,y} = (\omega^{(2)}_{y,x})^{-1} \circ \omega(\sigma_{x,y}) \circ \omega^{(2)}_{x,y}:
  \omega(x) \otimes_A \omega(y) \to \omega(y) \otimes_A \omega(x).
\end{equation*}
By Theorem~\ref{thm:FRT-bimonad}, the category of left $B_{\omega}$-comodules has a lax braiding.
As the following theorem shows, it is induced by a lax universal R-form on $B_{\omega}$.

\begin{theorem}
  \label{thm:B-omega-lax-universal-R}
  The bialgebroid $B_{\omega}$ has a lax universal R-form $\mathbf{r}_{\omega}$ given by
  \begin{equation}
    \label{eq:B-omega-lax-universal-R}
    \rform_{\omega}\Big( \eE{m}{\xi}, \eE{n}{\zeta} \Big)
    = \eval^{(2)}_{\omega(x), \omega(y)} \Big( \Sigma_{y,x}(n \otimes_A m) \otimes_A \zeta \otimes_A \xi \Big)
  \end{equation}
  for $x, y \in \mathcal{D}$, $m \in \omega(x)$, $\xi \in \omega(x)^{\vee}$, $n \in \omega(y)$ and $\zeta \in \omega(y)^{\vee}$.
  The lax braiding of the category of left $B_{\omega}$-comodules associated to this lax universal R-form coincides with the lax braiding given by Theorem~\ref{thm:FRT-bimonad}. If, moreover, the lax braiding $\sigma$ of $\mathcal{D}$ is invertible, then $\rform_{\omega}$ is in fact a universal R-form on $B_{\omega}$.
\end{theorem}
\begin{proof}
  The well-definedness of the bilinear map $\rform_{\omega}: B_{\omega} \times B_{\omega} \to A$ follows from the naturality of $\Sigma$ and the dinaturality of $\eval^{(2)}$.
  It is easy to see that $\rform_{\omega}$ satisfies~\eqref{eq:def-univ-R-1}--\eqref{eq:def-univ-R-2}.
  Although it may be possible to verify \eqref{eq:def-univ-R-4}--\eqref{eq:def-univ-R-7} in a direct way, we prefer to take a Tannaka theoretic approach: Let $\widetilde{\sigma}$ be the lax braiding of the category of left $B_{\omega}$-comodules given by Theorem~\ref{thm:FRT-bimonad}. We will show that the lax braiding $\widetilde{\sigma}$ is given by
  \begin{equation}
    \label{eq:B-omega-univ-R-pf-1}
    \widetilde{\sigma}_{M,N}(m \otimes n) = \rform_{\omega}(n_{(-1)}, m_{(-1)}) n_{(0)} \otimes_A m_{(0)}
    \quad (m \in M, n \in N)
  \end{equation}
  for all left $B_{\omega}$-comodules $M$ and $N$. Once this equation is proved, we will see that $\rform_{\omega}$ is a lax universal R-form on $B_{\omega}$ by Theorem~\ref{thm:univ-R-lax-braiding}.

  We recall the construction of the lax braiding $\widetilde{\sigma}$. Let $T = B_{\omega} \otimes_{A^{\env}} (-)$ be the bimonad on ${}_A \Mod_A$ corresponding to $B_{\omega}$. In Subsection~\ref{subsec:proof-main-thm-D-braided}, we first defined
  \begin{equation*}
    r_x: \omega(x) \otimes_A T(\omega(x)^{\vee}) \to A \quad (x \in \mathcal{D})
  \end{equation*}
  to be the unique morphism of $A$-bimodules such that the equation
  \begin{equation}
    \label{eq:B-omega-univ-R-pf-2}
    r_x \circ (\id_{\omega(x)} \otimes_A i_{y}(\omega(x)^{\vee})
    = \eval_{\omega(x), \omega(y)}^{(2)} \circ (\Sigma_{x,y} \otimes_A \id_{\omega(x)^{\vee}} \otimes_A \id_{\omega(y)^{\vee}})
  \end{equation}
  holds for all $y \in \mathcal{D}$. Then, for a left $B_{\omega}$-comodule $N$, we defined $a_N: T(N) \to N$ to be the unique morphism of $A$-bimodules such that
  \begin{equation}
    \label{eq:B-omega-univ-R-pf-3}
    a_N \circ i_x(N) = (r_x \otimes_A \id_N) \circ (\id_{\omega(x)} \otimes_A \rho_N(\omega(x)^{\vee}))
  \end{equation}
  for all $x \in \mathcal{D}$, where $\rho_N: N \otimes_A (-) \to T(-) \otimes_A N$ is the natural transformation associated to the coaction of $B_{\omega}$ on $N$ (see Remark~\ref{rem:T-comod-is-B-comod}). Finally, we defined
  \begin{equation*}
    \widetilde{\sigma}_{M,N}(m \otimes_A n) = (m_{(-1)} \rightharpoonup n) \otimes_A m_{(0)}
    \quad (m \in M, n \in N),
  \end{equation*}
  where $b \rightharpoonup n = a_N(b \otimes_{A^{\env}} n)$ for $b \in B_{\omega}$ and $n \in N$.

  We note that $B_{\omega}$ is generated by elements of the form $\eE{m}{\xi}$.
  Thus, to prove equation \eqref{eq:B-omega-univ-R-pf-1}, it suffices to show that the equation
  \begin{equation}
    \label{eq:B-omega-univ-R-pf-4}
    \eE{m}{\xi} \rightharpoonup n = \rform_{\omega}\left( n_{(-1)}, \eE{m}{\xi} \right) n_{(0)}
    \quad (m \in \omega(x), \xi \in \omega(x)^{\vee}, n \in N).
  \end{equation}
  holds. We fix $n \in N$ and write $n_{(-1)} \otimes_A n_{(0)}$ as
  \begin{equation}
    \label{eq:B-omega-univ-R-pf-5}
    n_{(-1)} \otimes_A n_{(0)} = \sum_{j} \eE{n_j}{\zeta_j} \otimes_A n'_j
  \end{equation}
  for some $n_j \in \omega(y_j)$, $\zeta \in \omega(y_j)^{\vee}$ ($y_j \in \mathcal{D}$) and $n'_j \in N$. Then,
  \begin{align*}
    \eE{m}{\xi} \rightharpoonup n
    & = a_N ((m \otimes_k \xi) \otimes_{A^{\env}} n) \\
    {}^{\eqref{eq:B-omega-universal-dinat}} \!
    & = a_N i_x(N) (m \otimes_A n \otimes_A \xi) \\
    {}^{\eqref{eq:T-comod-is-B-comod-pf-1}, \eqref{eq:B-omega-univ-R-pf-3}} \!
    & = r_x(m \otimes_A (n_{(-1)} \otimes_{A^{\env}} \xi)) n_{(0)} \\
    {}^{\eqref{eq:B-omega-univ-R-pf-5}} \!
    & = \sum_{j} r_x \Big(m \otimes_A i_y(\omega(x)^{\vee})(n_j \otimes_A \xi \otimes_A \zeta_j) \Big) n_j' \\
    {}^{\eqref{eq:B-omega-universal-dinat}, \eqref{eq:B-omega-univ-R-pf-2}} \!
    & = \sum_{j} \eval_{\omega(x),\omega(y)}\Big( \Sigma_{x,y}(m \otimes_A n_j) \otimes_A \xi \otimes_A \zeta_j \Big) n_j' \\
    {}^{\eqref{eq:B-omega-lax-universal-R}} \!
    & = \sum_{j} \rform_{\omega}\left( \eE{n_j}{\zeta_j}, \eE{m}{\xi} \right) n_j'
      = \rform_{\omega}\left( n_{(-1)}, \eE{m}{\xi} \right) n_{(0)}.
  \end{align*}

  Now we suppose that the lax braiding $\sigma$ of $\mathcal{D}$ is invertible. To complete the proof, we shall show that $\rform_{\omega}$ is a universal R-form on $B_{\omega}$. We note that $\Sigma_{x,y}$ is invertible for all $x, y \in \mathcal{D}$. Thus we define $\overline{\rform}: B_{\omega} \times B_{\omega} \to k$ by
  \begin{equation}
    \overline{\rform}\Big( \eE{m}{\xi}, \eE{n}{\zeta} \Big)
    = \eval^{(2)}_{\omega(x), \omega(y)} \Big( \Sigma_{x, y}^{-1}(n \otimes_A m) \otimes_A \zeta \otimes_A \xi \Big)
  \end{equation}
  for $x, y \in \mathcal{D}$, $m \in \omega(x)$, $\xi \in \omega(x)^{\vee}$, $n \in \omega(y)$ and $\zeta \in \omega(y)^{\vee}$. For simplicity of notation, we set $X := \eE{m}{\xi}$ and $Y = \eE{n}{\zeta}$. Let $\{ m^i, m_i \}$ and $\{ n^j, n_j \}$ be pairs of dual bases for $\omega(x)$ and $\omega(y)$, respectively. Then we have
  \begin{gather*}
    \rform_{\omega}(X_{(1)}, Y_{(1)}) \, \overline{\rform}(Y_{(2)}, X_{(2)})
    = \rform_{\omega}\left( \eE{m}{m^i}, \eE{n}{n^j} \right)
    \overline{\rform}\left( \eE{n_j}{\zeta}, \eE{m_i}{\xi} \right) \\
    = \eval^{(2)}( \Sigma_{y,x}(n \otimes_A m) \otimes_A n^j \otimes_A m^i)
    \, \eval^{(2)}( \Sigma_{y,x}^{-1}(m_i \otimes_A n_j) \otimes_A \xi \otimes_A \zeta).
  \end{gather*}
  This is the element of $A$ obtained by applying the map
  \begin{gather*}
    (\eval^{(2)}_{\omega(x),\omega(y)} \otimes_A \eval^{(2)}_{\omega(y),\omega(x)})
    \circ (\Sigma_{y,x} \otimes_A \id \otimes_A \id \otimes \Sigma_{y,x}^{-1} \otimes_A \id \otimes_A \id) \\
    \circ (\id_{\omega(x)} \otimes_A \id_{\omega(y)} \otimes_A \coev^{(2)}_{\omega(x),\omega(y)} \otimes_A \id_{\omega(y)^{\vee}} \otimes_A \id_{\omega(x)^{\vee}})
  \end{gather*}
  to $n \otimes_A m \otimes_A \xi \otimes_A \zeta$.
  One easily finds that this map is equal to $\eval^{(2)}_{\omega(x),\omega(y)}$ by representing the above expression graphically.
  Thus we continue the computation as follows:
  \begin{equation*}
    \rform_{\omega}(X_{(1)}, Y_{(1)}) \, \overline{\rform}(Y_{(2)}, X_{(2)})
    = \langle \zeta, n \langle \xi, m \rangle \rangle
    = \langle \xi \otimes_{\omega} \zeta, m \otimes_{\omega} n \rangle
    = \pi(X Y).
  \end{equation*}
  One can verify $\overline{\rform}(X_{(1)}, Y_{(1)}) \, \rform_{\omega}(Y_{(2)}, X_{(2)}) = \pi(X Y)$ in a similar way. Thus $\rform_{\omega}$ is a universal R-form. The proof is done.
\end{proof}

\subsection{The inverse of the Galois map of $B_{\omega}$}

In this subsection, we assume that the monoidal category $\mathcal{D}$ is right rigid. Then, by Theorem~\ref{thm:FRT-bimonad}~(\ref{item:main-thm-D-right-rigid}), the bimonad $T$ is a left Hopf monad and therefore the bialgebroid $B_{\omega}$ is a left Hopf algebroid. This means that the Galois map~\eqref{eq:bialgebroid-Galois-map} for $B = B_{\omega}$ is invertible. We denote by $\zeta_x: \omega(x)^{\vee} \to \omega(x^{\vee})$ the duality transformation for $\omega$ used in the proof of Theorem~\ref{thm:FRT-bimonad}~(\ref{item:main-thm-D-right-rigid}). Then we have:

\begin{theorem}
  \label{thm:B-omega-inv-Galois}
  The inverse of the Galois map $\beta$ is given by
  \begin{equation}
    \label{eq:B-omega-inv-Galois}
    \beta^{-1} \left( \eE{m}{\xi} \otimes_A \eE{n}{\chi} \right)
    = \eE{m}{\phi_i} \otimes_{A^{\op}} \eE{\xi'}{\phi^i} \eE{n}{\chi}
  \end{equation}
  for $m \in \omega(x)$, $\xi \in \omega(x)^{\vee}$, $n \in \omega(y)$ and $\chi \in \omega(y)^{\vee}$, where
  \begin{equation*}
    \xi' = \zeta_x(\xi) \in \omega(x^{\vee})
    \quad \text{and} \quad
    \phi^i \otimes_A \phi_i = (\id \otimes_A \zeta_x^{-1}) \coev_{\omega(x^{\vee})}(1).
  \end{equation*}
\end{theorem}
\begin{proof}
  Since we have already known that the Galois map $\beta$ is invertible, and since $\beta$ is right $B_{\omega}$-linear, it is sufficient to verify the equation
  \begin{equation*}
    \beta \left( \eE{m}{\phi_i} \otimes_{A^{\op}} \eE{\xi'}{\phi^i} \right)
    = \eE{m}{\xi} \otimes 1_{B_{\omega}}
  \end{equation*}
  for $m$ and $\xi$ as above.
  By the definition of $\check{\otimes}_{\omega}$ and $\zeta_x$, we have
  \begin{equation}
    \label{eq:B-omega-inv-Galois-2}
    \phi_i \check{\otimes}_{\omega} \phi^i
    = \check{\omega}^{(2)}_{x^{\vee}, x}(\phi^i \otimes_A \phi_i)
    \mathop{=}^{\eqref{eq:def-duality-trans-1}}
    (\omega(\eval_x)^{\vee} \circ \check{\omega}_{0})(1_A)
    = \omega(\eval_x)^{\vee}(\check{1}_{\omega}).
  \end{equation}
  Let $\{ m_j, m^j \}$ be a pair of dual bases for $\omega(x)$. The proof is completed as follows:
  \begin{align*}
    & \beta \left( \eE{m}{\phi_i} \otimes_{A^{\op}} \eE{\xi'}{\phi^i} \right) \\
    {}^{\eqref{eq:bialgebroid-Galois-map}, \eqref{eq:B-omega-comul}} \!
    & = \eE{m}{m^j} \otimes_{A} \eE{m_j}{\phi_i} \eE{\xi'}{\phi^i} \\
    {}^{\eqref{eq:B-omega-mult}} \!
    & =\eE{m}{m^j} \otimes_{A} \eE{m_j \otimes_{\omega} \xi'}{\phi_i \check{\otimes}_{\omega} \phi^i} \\
    {}^{\eqref{eq:B-omega-inv-Galois-2}} \!
    & =\eE{m}{m^j} \otimes_{A} \eE{\omega(\eval_x)(m_j \otimes_{\omega} \xi')}{\check{1}_{\omega}} \\
    {}^{\eqref{eq:def-duality-trans-1}} \!
    & =\eE{m}{m^j} \otimes_{A} \eE{\langle \xi, m_j \rangle 1_{\omega}}{\check{1}_{\omega}}
      \quad \text{(in ${}_{\tgt}B \otimes_A {}_{\src}B$)} \\
    {}^{\eqref{eq:B-omega-src-tgt}} \!
    & =\eE{m}{m^j \langle \xi, m_j \rangle} \otimes_{A} \eE{1_{\omega}}{\check{1}_{\omega}} \\
    {}^{\eqref{eq:dual-basis},\eqref{eq:B-omega-unit}} \!
    & =\eE{m}{\xi} \otimes_A 1_{B_{\omega}}. \qedhere
  \end{align*}
\end{proof}

\section{The FRT construction over non-commutative algebras}
\label{sec:FRT-over-non-comm-rings}

\subsection{Generators and relations for monoidal categories}
\label{subsec:gen-rel-mon}

Throughout this section, we fix a commutative ring $k$ and a $k$-algebra $A$.
In the last section, we have constructed a left bialgebroid $B_{\omega}$ over $A$ from a construction data $(\mathcal{D}, \omega)$ over ${}_A \Mod_A$.
This section aims to present $B_{\omega}$ by generators and relations in the case where the monoidal category $\mathcal{D}$ is presented by generators and relations.

We first recall what a presentation of a monoidal category is. A {\em tensor scheme} \cite[Definition 1.4]{MR1113284} (also called a {\em monoidal signature}) is a data $\Sigma = (\Sigma_0, \Sigma_1, \src, \tgt)$ consisting of two sets $\Sigma_0$ and $\Sigma_1$ and two maps $\src, \tgt: \Sigma_1 \to \langle \Sigma_0 \rangle_{\otimes}$, where $\langle X \rangle_{\otimes}$ for a set $X$ is the free monoid on $X$ with binary operation $\otimes$ and unit $\unitobj$. We usually write $f: x \to y$ if $f$ is an element of $\Sigma_1$ such that $\src(f) = x$ and $\tgt(f) = y$.

Given a tensor scheme $\Sigma = (\Sigma_0, \Sigma_1, \src, \tgt)$, a {\em free monoidal category} on $\Sigma$ is defined as a monoidal category with a certain universal property. A free monoidal category on $\Sigma$ exists and is unique up to equivalence. In this paper, we denote by $\mathcal{F}_{\Sigma}$ the free monoidal category given explicitly in \cite[Theorem 1.2]{MR1113284}. The reader needs not to review the full description of the construction of $\mathcal{F}_{\Sigma}$. Besides the freeness of $\mathcal{F}_{\Sigma}$, we only use the following properties:
\begin{enumerate}
\item $\mathcal{F}_{\Sigma}$ is a strict monoidal category.
\item As a monoid, $\Obj(\mathcal{F}_{\Sigma})$ is identical to $\langle \Sigma_0 \rangle_{\otimes}$.
\item Every morphism of $\mathcal{F}_{\Sigma}$ is built from elements of $\Sigma_1$ and $\id_x$ ($x \in \Sigma_0 \sqcup \{ \unitobj \}$)
  by taking the tensor product and the composition.
\end{enumerate}

Given a family $\mathcal{E} := (f_i, g_i)_{i \in I}$ of parallel morphisms in $\mathcal{F}_{\Sigma}$ indexed by a set $I$, we define $\sim_{\mathcal{E}}$ to be the minimal equivalence relation on the set $\Mor(\mathcal{F}_{\Sigma})$ such that the following statements hold:
\begin{enumerate}
\item We have $f_i \sim_{\mathcal{E}} g_i$ for all  $i \in I$.
\item If $f \sim_{\mathcal{E}} g$, then we have $f \circ a \sim_{\mathcal{E}} g \circ a$ and $b \circ f \sim_{\mathcal{E}} b \circ g$ for all morphisms  $a$ and $b$ in $\mathcal{F}_{\Sigma}$ such that $f \circ a$ and $b \circ f$ make sense.
\item If $f \sim_{\mathcal{E}} g$ and $f' \sim_{\mathcal{E}} g'$, then we have $f \otimes f' \sim_{\mathcal{E}} g \otimes g'$.
\end{enumerate}
Now we define the category $\mathcal{F}_{\Sigma / \mathcal{E}}$ by
\begin{equation*}
  \Obj(\mathcal{F}_{\Sigma / \mathcal{E}}) = \Obj(\mathcal{F}_{\Sigma})
  \quad \text{and} \quad
  \Mor(\mathcal{F}_{\Sigma / \mathcal{E}}) = \Mor(\mathcal{F}_{\Sigma}) / \mathord{\sim}_{\mathcal{E}}
\end{equation*}
with the composition of morphisms given by $[f] \circ [g] = [f \circ g]$, where $[h]$ means the equivalence class of $h \in \Mor(\mathcal{F}_{\Sigma})$. The category $\mathcal{F}_{\Sigma / \mathcal{E}}$ is a strict monoidal category with the tensor product $[f] \otimes [g] = [f \otimes g]$.

\begin{definition}
  We call $\mathcal{F}_{\Sigma / \mathcal{E}}$ the monoidal category generated by the monoidal signature $\Sigma$ with relations $f_i = g_i$ ($i \in I$). By abuse of notation, we usually denote a morphism $[f]$ in this category by $f$.
\end{definition}

\begin{example}
  \label{ex:braid-cat}
  The category $\mathfrak{B}$ of braids is defined as follows: An object of this category is a non-negative integer. Given $m, n \in \Obj(\mathfrak{B})$, the set $\mathfrak{B}(m, n)$ is the braid group of $m$ strands if $m = n$ and the empty set otherwise. The composition of morphisms is given by the multiplication in the braid group.
  This category is in fact a monoidal category with tensor product given by $n \otimes m = n + m$ for objects and by juxtaposition for morphisms. It is well-known that the monoidal category $\mathfrak{B}$ is generated by the tensor scheme
  \begin{equation*}
    \Sigma_0 = \{ x \},
    \quad \Sigma_1 = \{ \sigma: x \otimes x \to x \otimes x, \ \overline{\sigma}: x \otimes x \to x \otimes x \}
  \end{equation*}
  with defining relations
  \begin{equation*}
    (\sigma \otimes \id_x) (\id_x \otimes \sigma) (\sigma \otimes \id_x)
    = (\id_x \otimes \sigma)(\sigma \otimes \id_x) (\id_x \otimes \sigma),
    \quad \sigma \overline{\sigma} = \id_{x \otimes x} = \overline{\sigma} \sigma.
  \end{equation*}
  An object $n$ of $\mathfrak{B}$ corresponds to the $n$-th tensor power of $x$ in this presentation.
\end{example}

\subsection{Tensor $R$-ring}

Let $R$ be a $k$-algebra. Given an $R$-bimodule $M$, we define
\begin{equation*}
  T_R(M) = \bigoplus_{i = 0}^{\infty} T^i_R(M),
  \quad T^0_R(M) = R,
  \quad T^i_R(M) = \underbrace{M \otimes_R \dotsb \otimes_R M}_i
  \quad (i > 0).
\end{equation*}
The $R$-bimodule $T_{R}(M)$ is an $R$-ring by the tensor product over $R$.

\begin{definition}
  We call $T_R(M)$ {\em the tensor $R$-ring} over $M$.
\end{definition}

By the construction, the tensor $R$-ring $T_{R}(M)$ has the following universal property: For any $R$-ring $B$ and any morphism $\phi: M \to B$ of $R$-bimodules, there is a unique morphism $\widetilde{\phi}: T_R(M) \to B$ of $R$-rings such that $\widetilde{\phi}|_M = \phi$.

The universal property allows us to extend the assignment $M \mapsto T_{R}(M)$ to a functor $T_R$ from ${}_R \Mod_R$ to the category of $R$-rings. We remark:

\begin{lemma}
  \label{lem:Ker-TRf}
  Let $f: M \to N$ be a morphism in ${}_R \Mod_R$. Then the kernel of $T_R(f)$ is the ideal generated by $\Ker(f) \subset M$ regarded as a subset of $T_R(M)$.
\end{lemma}
\begin{proof}
  The map $T_{R}(f)$ is actually the direct sum of
  \begin{equation*}
    \id_{R}: T^0_{R}(M) \to T^0_{R}(N), \quad
    \underbrace{f \otimes_A \dotsb \otimes_A f}_i
    : T_{R}^{i}(M) \to T_{R}^{i}(N)
    \quad (i = 1, 2, \dotsc).
  \end{equation*}
  We have $\Ker(f_1 \otimes_R f_2) = \Ker(f_1) \otimes_R M_2 + M_1 \otimes_R \Ker(f_2)$ in $M_1 \otimes_A M_2$ if $f_1$ and $f_2$ are homomorphisms of $R$-bimodules out of $M_1$ and $M_2$, respectively. The claim of this lemma is proved by using this formula iteratively.
\end{proof}

The free $R$-bimodule $R X R$ over a set $X$ is an $R$-bimodule that is free with basis $X$ as a left $R^{\env}$-module. The {\em free $R$-ring} $R\langle X \rangle$ over $X$ is the tensor $R$-ring over the free $R$-bimodule $R X R$. The above lemma says that one can present $T_{R}(M)$ as a quotient of $R \langle X \rangle$ if there is an epimorphism $R X R \twoheadrightarrow M$ of $R$-bimodules whose kernel is known.

\subsection{Generators and relations for the bialgebroid $B_{\omega}$}

Let $(\mathcal{D}, \omega)$ be a construction data over ${}_A \Mod_A$ such that the monoidal category $\mathcal{D}$ is generated by a monoidal signature $(\Sigma_0, \Sigma_1, \src, \tgt)$ with some relations. The bialgebroid $B_{\omega}$ we have constructed in the last section is a quotient of the $A^{\env}$-bimodule $\widetilde{B}_{\omega} = \bigoplus_{x \in \Sigma} \widetilde{E}_x$, where $\widetilde{E}_x$ is the $A^{\env}$-bimodule introduced in Subsection \ref{subsec:realize-coends}. In this subsection, we present the bialgebroid $B_{\omega}$ as a quotient of the tensor $A^{\env}$-ring
\begin{equation*}
  \widetilde{B}^{\Sigma}_{\omega} := T_{A^{\env}}(\widetilde{G}^{\Sigma}_{\omega}),
  \quad \text{where} \quad \widetilde{G}^{\Sigma}_{\omega} := \bigoplus_{x \in \Sigma_0} \widetilde{E}_x.
\end{equation*}
For this purpose, we first remark:

\begin{lemma}
  The $A^{\env}$-bimodule $\widetilde{B}_{\omega}$ is an $A^{\env}$-ring by the structure maps given by the same formula as $B_{\omega}$ but with the symbol $\mathbf{e}$ replaced with $\widetilde{\mathbf{e}}$.
\end{lemma}
\begin{proof}
  Straightforward.
\end{proof}

By the universal property, there is a unique morphism $\widetilde{\phi}: \widetilde{B}^{\Sigma}_{\omega} \to \widetilde{B}_{\omega}$ of $A^{\env}$-rings extending the inclusion map $\widetilde{G}^{\Sigma}_{\omega} \hookrightarrow \widetilde{B}_{\omega}$. The map $\widetilde{\phi}$ is in fact an isomorphism of $A^{\env}$-rings. To see this, we introduce some notation:
Let $\ell$ be a positive integer.
For objects $z_1, \dotsc, z_{\ell}$ of $\mathcal{D}$, there are canonical isomorphisms
\begin{gather*}
  \omega^{(\ell)}_{z_1, \dotsc, z_{\ell}}:
  \omega(z_1) \otimes_A \dotsb \otimes_A \omega(z_{\ell}) \to \omega(z_1 \otimes \dotsb \otimes z_{\ell}), \\
  \check{\omega}^{(\ell)}_{z_{\ell}, \dotsc, z_{1}} :
  \omega(z_{\ell})^{\vee} \otimes_A \dotsb \otimes_A \omega(z_{1})^{\vee}
  \to \omega(z_1 \otimes \dotsb \otimes z_{\ell})^{\vee}
\end{gather*}
obtained by using the monoidal structure of $\omega$ iteratively.
In the case where $\ell = 0$, we make the convention that $z_1 \otimes \dotsb \otimes z_{\ell} = \unitobj$,
\begin{equation}
  \label{eq:omega-0-convention}
  \begin{array}{cc}
    \omega(z_1) \otimes_A \dotsb \otimes_A \omega(z_{\ell}) = A,
    & \omega^{(0)}_{z_1, \cdots, z_{\ell}} = \omega_0, \\[3pt]
    \omega(z_{\ell})^{\vee} \otimes_A \dotsb \otimes_A \omega(z_{1})^{\vee} = A^{\op},
    & \check{\omega}^{(0)}_{z_{\ell}, \cdots, z_{1}} = \check{\omega}_0.
  \end{array}
\end{equation}
Every non-unit object $x$ of $\mathcal{D}$ can be written as $x = x_1 \otimes \dotsb \otimes x_m$ for some $x_i \in \Sigma_0$ in a unique way. By using this expression, we define
\begin{equation*}
  \widetilde{E}^{\Sigma}_{x} := \widetilde{E}_{x_1} \otimes_{A^{\env}} \dotsb \otimes_{A^{\env}} \widetilde{E}_{x_m}
\end{equation*}
and often identify it with the $A^{\env}$-bimodule
\begin{equation*}
  (\omega(x_1) \otimes_A \dotsb \otimes_A \omega(x_{\ell}))
  \otimes_k (\omega(x_{\ell})^{\vee} \otimes_A \dotsb \otimes_A \omega(x_1)^{\vee})
\end{equation*}
through the isomorphism defined by
\begin{equation}
  \label{eq:widetilde-phi-iso-1}
  \begin{aligned}
    \wE{m_1}{\xi_1} \otimes_{A^{\env}} \dotsb \otimes_{A^{\env}} \wE{m_{\ell}}{\xi_{\ell}}
    & \mapsto (m_1 \otimes_A m_2 \otimes_A \dotsb \otimes_A m_{\ell}) \\[-3pt]
    & \qquad \qquad \mbox{} \otimes_k (\xi_{\ell} \otimes_A \dotsb \otimes_A \xi_2 \otimes_A \xi_1).
  \end{aligned}
\end{equation}
For the unit object of $\mathcal{D}$, we set $\widetilde{E}^{\Sigma}_{\unitobj} = A^{\env}$.
Then we have $\widetilde{B}^{\Sigma}_{\omega} = \bigoplus_{x \in \Obj(\mathcal{D})} \widetilde{E}^{\Sigma}_x$ by the definition of tensor $A^{\env}$-rings. Now we prove:

\begin{lemma}
  \label{lem:widetilde-phi-iso}
  The map $\widetilde{\phi}$ is an isomorphism of $A^{\env}$-rings.
\end{lemma}
\begin{proof}
  Both $\widetilde{B}_{\omega}$ and $\widetilde{B}^{\Sigma}_{\omega}$ are graded by the monoid $\Obj(\mathcal{D})$, and the map $\widetilde{\phi}$ preserves the grading.
  For $x \in \Obj(\mathcal{D})$, we denote by $\widetilde{\phi}_x : \widetilde{E}^{\Sigma}_x \to \widetilde{E}_x$ the map induced by $\widetilde{\phi}$.
  Let $x$ be a non-unit object of $\mathcal{D}$ and write it as $x = x_1 \otimes \dotsb \otimes x_{\ell}$ for some $x_i \in \Sigma_0$. By the definition of the multiplication of $\widetilde{B}_{\omega}$, we have
  \begin{gather*}
    \widetilde{\phi}\left( \wE{m_1}{\xi_1} \cdots \wE{m_{\ell}}{\xi_{\ell}} \right)
    = \widetilde{\phi}\left( \wE{m_1}{\xi_1} \right) \cdots \widetilde{\phi}\left( \wE{m_{\ell}}{\xi_{\ell}} \right) \\
    = \wE{m_1}{\xi_1} \cdots \wE{m_{\ell}}{\xi_{\ell}}
    = \wE{m_1 \otimes_{\omega} \dotsb \otimes_{\omega} m_{\ell}}
    {\xi_{\ell} \otimes_{\omega} \dotsb \otimes_{\omega} \xi_1}
  \end{gather*}
  for $m_i \in \omega(x_i)$ and $\xi_i \in \omega(x_i)^{\vee}$ ($i = 1, \dotsc, \ell$). This means that $\widetilde{\phi}_x$ is equal to the composition of \eqref{eq:widetilde-phi-iso-1} and $\omega^{(\ell)}_{z_1,\cdots,z_{\ell}} \otimes_k \omega^{(\ell)}_{z_{\ell}, \cdots, z_1}$.
  Hence $\widetilde{\phi}_x$ is bijective.
  It can be directly verified that $\widetilde{\phi}_{\unitobj}$ is also bijective.
  Thus $\widetilde{\phi}$ is an isomorphism.
\end{proof}

Given a morphism $f: x_1 \otimes \dotsb \otimes x_m \to y_1 \otimes \dotsb \otimes y_n$ in $\mathcal{D}$ with $x_i, y_j \in \Sigma_0$, we define the $k$-linear map $\Rel^{\Sigma}_f$ by
\begin{equation}
  \label{eq:B-omega'-Rel}
  \begin{gathered}
    \Rel^{\Sigma}_f: (\omega(x_1) \otimes_A \dotsb \otimes_A \omega(x_m)) \otimes_k (\omega(y_n)^{\vee} \otimes_A \dotsb \otimes_A \omega(y_1)^{\vee}) \to \widetilde{B}^{\Sigma}_{\omega}, \\
    \Rel^{\Sigma}_f = \widetilde{\phi}{}^{-1} \circ \Rel_f \circ (\omega^{(m)}_{x_1, \cdots, x_m} \otimes_{k} \check{\omega}^{(n)}_{y_n, \cdots, y_1})
  \end{gathered}
\end{equation}
where the convention \eqref{eq:omega-0-convention} is used for the case where either $m$ or $n$ is zero.

\begin{definition}
  We define $B^{\Sigma}_{\omega}$ to be the quotient of the $A^{\env}$-ring $\widetilde{B}^{\Sigma}_{\omega}$ by the ideal generated by the set $\bigcup_{f \in \Sigma_1} \Img(\Rel^{\Sigma}_f)$. 
\end{definition}

For $m \in \omega(x)$ and $\xi \in \omega(x)^{\vee}$ with $x \in \Sigma_0$, we denote by $\eE{m}{\xi}$ the image of the element $\wE{m}{\xi} \in \widetilde{B}^{\Sigma}_{\omega}$ under the quotient map $q^{\Sigma} : \widetilde{B}^{\Sigma}_{\omega} \twoheadrightarrow B^{\Sigma}_{\omega}$.
As in Subsection~\ref{subsec:bialgebroid-B-omega}, we use same symbols to express elements of the bialgebroid $B_{\omega}$.
We also denote by $G_{\omega}^{\Sigma}$ the image of $\widetilde{G}_{\omega}^{\Sigma}$ under $q^{\Sigma}$.

\begin{theorem}
  \label{thm:B-omega-gen-rel}
  The $A^{\env}$-ring $B^{\Sigma}_{\omega}$ has a unique structure of a left bialgebroid over $A$ determined by the same formula as \eqref{eq:B-omega-mult}--\eqref{eq:B-omega-counit} on the generating set $G^{\Sigma}_{\omega}$. There is a unique isomorphism $\phi: B^{\Sigma}_{\omega} \to B_{\omega}$ of left bialgebroids given by
  \begin{equation}
    \label{eq:B-dash-to-B-omega}
    \phi\left( \eE{m}{\xi} \right) = \eE{m}{\xi}
  \end{equation}
  for $m \in \omega(x)$ and $\xi \in \omega(x)^{\vee}$ with $x \in \Sigma_0$.
\end{theorem}

We recall that the bialgebroid $B_{\omega}$ is defined to be the quotient of $\widetilde{B}_{\omega}$ by the sum of $\Img(\Rel_f)$'s, where $f$ runs over all morphisms in $\mathcal{D}$.
This theorem expresses the bialgebroid $B_{\omega}$ as a quotient of the tensor ring over $\widetilde{G}^{\Sigma}_{\omega}$. Furthermore, relations we need to impose on generators can be reduced to those coming from the generating set $\Sigma_1$ of morphisms of $\mathcal{D}$.

The following lemma is a mechanism of reducing relations:

\begin{lemma}
  \label{lem:B-omega-gen-rel-2}
  For two morphisms $f: y \to z$ and $g: x \to y$ in $\mathcal{D}$, we have
  \begin{equation}
    \label{eq:B-omega-gen-rel-2}
    \Img(\Rel^{\Sigma}_{f g}) \subset \Img(\Rel^{\Sigma}_f) + \Img(\Rel^{\Sigma}_g).
  \end{equation}
  For two morphisms $f_1: x_1 \to y_1$ and $f_2: x_2 \to y_2$ in $\mathcal{D}$, we have
  \begin{equation}
    \label{eq:B-omega-gen-rel-3}
    \Img(\Rel^{\Sigma}_{f_1 \otimes f_2}) \subset
    \Img(\Rel^{\Sigma}_{f_1}) \cdot \widetilde{B}^{\Sigma}_{\omega}
    + \widetilde{B}^{\Sigma}_{\omega} \cdot \Img(\Rel^{\Sigma}_{f_2}).
  \end{equation}
  If $f$ is an isomorphism in $\mathcal{D}$, then we have
  \begin{equation}
    \label{eq:B-omega-gen-rel-4}
    \Img(\Rel^{\Sigma}_{f}) = \Img(\Rel^{\Sigma}_{f^{-1}}).
  \end{equation}
\end{lemma}
\begin{proof}
  Since $\widetilde{\phi}$ is an isomorphism of $A^{\env}$-rings and $\widetilde{\phi}(\Img(\Rel^{\Sigma}_f)) = \Img(\Rel_f)$ for all morphisms $f$ in $\mathcal{D}$, it suffices to show the corresponding results for $\Rel_f$. We first prove \eqref{eq:B-omega-gen-rel-2}. For $m \in \omega(x)$ and  $\xi \in \omega(z)^{\vee}$, we have
  \begin{align*}
    \Rel_{f g}(m \otimes_k \xi)
    & = \Rel_f(\omega(g)(m) \otimes_k \xi) + \Rel_g(m \otimes_k \omega(f)^{\vee}(\xi)),
  \end{align*}
  which implies $\Img(\Rel_{f g}) \subset \Img(\Rel_f) + \Img(\Rel_g)$. Thus \eqref{eq:B-omega-gen-rel-2} follows. To prove the next one, we first note that \eqref{eq:B-omega-gen-rel-2} yields
  \begin{equation*}
    \Img(\Rel_{f_1 \otimes f_2}) \subset
    \Img(\Rel_{f_1 \otimes \id_{y_2}})
    + \Img(\Rel_{\id_{x_1} \otimes f_2}).
  \end{equation*}
  Thus, to prove \eqref{eq:B-omega-gen-rel-3}, it is enough to show
  \begin{equation}
    \label{eq:B-omega-gen-rel-pf-1}
    \Img(\Rel_{f_1 \otimes \id_{y_2}}) \subset \Img(\Rel_{f_1}) \cdot \widetilde{B}_{\omega}, \quad
    \Img(\Rel_{\id_{x_1} \otimes f_2}) \subset \widetilde{B}_{\omega} \cdot \Img(\Rel_{f_2}).
  \end{equation}
  Now we fix elements $m \in \omega(x_1 \otimes y_2)$ and $\xi \in \omega(y_1 \otimes y_2)^{\vee}$ and write
  \begin{equation*}
    (\omega^{(2)}_{x_1, y_2})^{-1}(m) = \sum_{i} m_i' \otimes_A m_i''
    \quad \text{and} \quad
    (\check{\omega}^{(2)}_{y_2, y_1})^{-1}(\xi) = \sum_{j} \xi_j' \otimes_A \xi_j''
  \end{equation*}
  for some $m_i' \in \omega(x_1)$, $m_i'' \in \omega(y_2)$, $\xi_j' \in \omega(y_2)^{\vee}$ and $\xi_j'' \in \omega(y_1)^{\vee}$. Then we have
  \begin{align*}
    \Rel_{f_1 \otimes \id_{y_2}}(m \otimes_k \xi)
    = \sum_{i, j} \Rel_{f_1}(m_i' \otimes_k \xi_j'') \cdot \wE{m_i''}{\xi_{j}'}
    \in \Img(\Rel_{f_1}) \cdot \widetilde{B}_{\omega},
  \end{align*}
  which implies the first inclusion relation in \eqref{eq:B-omega-gen-rel-pf-1}. The second one is verified in a similar way. Thus we obtain \eqref{eq:B-omega-gen-rel-3}.

  Finally, we prove \eqref{eq:B-omega-gen-rel-4}. If $f: y \to z$ is an isomorphism in $\mathcal{D}$, then
  \begin{equation*}
    \Rel_f(m \otimes \xi)
    = -\Rel_{f^{-1}}(\omega(f)(m) \otimes_k \omega(f)^{\vee}(\xi))
    \in \Img(\Rel_{f^{-1}})
  \end{equation*}
  for $m \in \omega(y)$ and $\xi \in \omega(z)^{\vee}$. Hence $\Img(\Rel_f) \subset \Img(\Rel_{f^{-1}})$. By symmetry, we also have the converse inclusion. The proof is done.
\end{proof}

\begin{proof}[Proof of Theorem~\ref{thm:B-omega-gen-rel}]
  Let $J^{\Sigma}_{\omega}$ be the ideal of $\widetilde{B}^{\Sigma}_{\omega}$ generated by $\bigcup_{f \in \Sigma_1} \Img(\Rel^{\Sigma}_f)$, and let $J_{\omega}$ be the $k$-submodule of $\widetilde{B}_{\omega}$ defined by~\eqref{eq:ideal-J-omega}. Since the quotient map $\widetilde{B}_{\omega} \twoheadrightarrow B_{\omega}$ ($:= \widetilde{B}_{\omega}/J_{\omega}$) is a morphism of $A^{\env}$-rings, $J_{\omega}$ is in fact an ideal of $\widetilde{B}_{\omega}$. Since $\Img(\Rel_f) \subset J_{\omega}$ for all morphisms $f$ in $\mathcal{D}$, we have $\widetilde{\phi}(J^{\Sigma}_{\omega}) \subset J_{\omega}$. Since $\mathcal{D}$ is generated by $\Sigma$, we also have
  \begin{equation*}
    J_{\omega} = \sum_{f \in \Mor(\mathcal{D})} \Img(\Rel_f)
    \subset \sum_{f \in \Sigma_1} \widetilde{B}_{\omega} \cdot \Img(\Rel_f) \cdot \widetilde{B}_{\omega}
    \subset \widetilde{\phi}(J_{\omega}^{\Sigma})
  \end{equation*}
  by Lemma~\ref{lem:B-omega-gen-rel-2}. Thus, in conclusion, we have $\widetilde{\phi}(J^{\Sigma}_{\omega}) = J_{\omega}$.
  Hence $\widetilde{\phi}$ induces an isomorphism between the quotient $A^{\env}$-rings $B^{\Sigma}_{\omega}$ and $B_{\omega}$, which we denote by $\phi: B^{\Sigma}_{\omega} \to B_{\omega}$. We equip $B^{\Sigma}_{\omega}$ with a structure of a left bialgebroid through $\phi$. The resulting bialgebroid structure of $B^{\Sigma}_{\omega}$ is as stated. The isomorphism $\phi$ obviously fulfills the required conditions. The proof is done.
\end{proof}

\subsection{The FRT construction over non-commutative algebras}
\label{subsec:FRT-construction}

A {\em braided object} in a monoidal category $\mathcal{C}$ is a pair $(V, c)$ consisting of an object $V \in \mathcal{C}$ and an isomorphism $c: V \otimes V \to V \otimes V$ in $\mathcal{C}$ satisfying the braid equation
\begin{equation*}
  (c \otimes \id_V) (\id_V \otimes c) (c \otimes \id_V)
  = (\id_V \otimes c)(c \otimes \id_V) (\id_V \otimes c).
\end{equation*}
Let $(M, c)$ be a braided object in ${}_A \Mod_A$ such that $M$ is right rigid or, equivalently, finitely generated and projective as a left $A$-module.
We present the category $\mathfrak{B}$ of braids as in Example~\ref{ex:braid-cat}.
There is a unique strict monoidal functor $\omega: \mathfrak{B} \to {}_A \Mod_A$ such that $\omega(x) = M$ and $\omega(\sigma) = c$. Since $\mathfrak{B}$ is braided, the left bialgebroid $B_{\omega}$ associated to $\omega$ has a universal R-form given by Theorem~\ref{thm:B-omega-lax-universal-R}.

\begin{definition}
  \label{def:FRT-bialgebroid}
  We call the coquasitriangular left bialgebroid $B_{\omega}$ with $\omega$ as above {\em the FRT bialgebroid} associated to the braided object $(M, c)$.
\end{definition}

As an application of our results, we give a presentation of the FRT bialgebroid associated to $(M, c)$. We first fix a pair $\{ m_i, m^i \}_{i = 1}^N$ of dual bases for $M$. It is easy to see that $\{ m_i \}_{i = 1}^N$ generates $M$ as a left $A$-module and $\{ m^i \}_{i = 1}^N$ generates $M^{\vee}$ as a right $A$-module.
Moreover, the set $\{ m_i \otimes_A m_j \}_{i, j = 1}^N$ generates $M \otimes_A M$ as a left $A$-module.
Thus we can write
\begin{equation}
  \label{eq:B-omega-FRT-R-matrix}
  c(m_i \otimes_A m_j) = \sum_{r, s = 1}^N w_{i j}^{r s} \, m_r \otimes_A m_s
  \quad (i, j = 1, \dotsc, N)
\end{equation}
for some $w_{i j}^{r s} \in A$ (although this expression is not unique in general). We make the $k$-module $G := M \otimes_k M^{\vee}$ an $A^{\env}$-bimodule by
\begin{equation}
  \label{eq:B-omega-FRT-generators-action}
  (a_1 \otimes_k a_2^{\op}) \cdot (m \otimes_k \xi) \cdot (a_3 \otimes_k a_4^{\op})
  = (a_1 m a_3) \otimes_k (a_4 \xi a_2)
\end{equation}
for $a_1, a_2, a_3, a_4 \in A$, $m \in M$ and $\xi \in M^{\vee}$.

\begin{theorem}
  \label{thm:FRT-coqt-bialgebroid}
  We define the $A^{\env}$-ring $B(M, c)$ to be the quotient of $T_{A^{\env}}(G)$ by the ideal generated by
  \begin{equation}
    \label{eq:B-omega-FRT-relation}
    \sum_{r, s = 1}^N w_{i j}^{r s}\rhd T^p_r T^q_s
    - \sum_{r, s = 1}^N T_i^r T_j^s \lhd \overline{w}^{p q}_{r s}
    \quad (i, j, p, q = 1, \dotsc, N),
  \end{equation}
  where $T_i^j := m_i \otimes_k m^j \in G$,
  \begin{equation}
    \label{eq:B-omega-FRT-relation-c-mathring}
    \overline{w}_{r s}^{p q} := \sum_{i, j = 1}^N w_{r s}^{i j} \langle m^p, m_i \langle m^q, m_j \rangle \rangle,
  \end{equation}
  and the $A$-actions $\rhd$ and $\lhd$ are given by our rule \eqref{eq:bialgebroid-Re-action-left} when we view an $A^{\env}$-ring as an $A$-bimodule. Then $B(M, c)$ has a unique structure of a coquasitriangular left bialgebroid over $A$ determined by
  \begin{equation*}
    \Delta(T_{i}^{j}) = \sum_{r = 1}^N T_i^r \otimes_A T_r^j,
    \quad \pi(T_{i}^{j}) = \langle m^j, m_i \rangle
    \quad \text{and}
    \quad \mathbf{r}(T_{i}^{r}, T_{j}^{s}) = \overline{w}_{j i}^{r s}
  \end{equation*}
  for $i, j, r, s = 1, \cdots, N$. The FRT bialgebroid of Definition \ref{def:FRT-bialgebroid} associated to $(M, c)$ is isomorphic to the bialgebroid $B(M, c)$.
\end{theorem}

The tensor $A^{\env}$-ring $T_{A^{\env}}(G)$ is not a free $A^{\env}$-ring in general. If one desires to present $B(M, c)$ as a quotient of a free $A^{\env}$-ring, consider the free $A^{\env}$-bimodule $F$ generated by the set $X = \{ X_i^j \}_{i, j = 1, \dotsc, N}$. There is a unique morphism $p: F \to G$ of $A^{\env}$-bimodules such that $p(X_i^j) = T_i^j$ for all $i, j$. By Lemma~\ref{lem:Ker-TRf} and the succeeding remark, the $A^{\env}$-ring $B(M, c)$ can be defined as the quotient of the free $A^{\env}$-ring $A^{\env} \langle X \rangle$ by the ideal generated by the set
\begin{equation*}
  \text{(generators of $\Ker(p)$)} \cup \{ E_{ij}^{pq} \mid i, j, p, q = 1, \dotsc, N \},
\end{equation*}
where $E_{ij}^{pq}$ is the element of $A^{\env}\langle X \rangle$ given by the same formula as \eqref{eq:B-omega-FRT-relation} but with $T_{a}^{b}$ replaced with $X_{a}^{b}$.

\begin{proof}[Proof of Theorem \ref{thm:FRT-coqt-bialgebroid}]
  Let $\Sigma$ be the tensor scheme used to present the category $\mathfrak{B}$ of braids in Example~\ref{ex:braid-cat}. Since the FRT bialgebroid is isomorphic to $B^{\Sigma}_{\omega}$, we will actually show that $B(M, c)$ is identical to $B^{\Sigma}_{\omega}$.

  Given a morphism $f$ in $\mathfrak{B}$, we denote by $I(f)$ the ideal of $T_{A^{\env}}(G)$ generated by $\Img(\Rel^{\Sigma}_{f})$.
  By the construction, $B^{\Sigma}_{\omega}$ is a quotient of $T_{A^{\env}}(G)$ by $I := I(\sigma) + I(\overline{\sigma})$.
  Since $I(\sigma) = I(\overline{\sigma})$ by Lemma~\ref{lem:B-omega-gen-rel-2}, we have $I = I(\sigma)$. For simplicity, we set
  \begin{equation*}
    c^{\dagger} := (\check{\omega}^{(2)}_{x,x})^{-1} \circ \omega(\sigma)^{\vee} \circ \check{\omega}^{(2)}_{x,x}
    : M^{\vee} \otimes_A M^{\vee} \to M^{\vee} \otimes_A M^{\vee}.
  \end{equation*}
  This map is equal to the composition
  \begin{align}
    M^{\vee} \otimes M^{\vee}
    \label{eq:B-omega'-FRT-relation-computation-1}
    & \xrightarrow{\makebox[8em]{$\scriptstyle \coev^{(2)}_{M,M} \otimes \id \otimes \id$}}
      M^{\vee} \otimes M^{\vee} \otimes M \otimes M \otimes M^{\vee} \otimes M^{\vee} \\
    \label{eq:B-omega'-FRT-relation-computation-2}
    & \xrightarrow{\makebox[8em]{$\scriptstyle \id \otimes \id \otimes c \otimes \id \otimes \id$}}
      M^{\vee} \otimes M^{\vee} \otimes M \otimes M \otimes M^{\vee} \otimes M^{\vee} \\
    \label{eq:B-omega'-FRT-relation-computation-3}
    & \xrightarrow{\makebox[8em]{$\scriptstyle \id \otimes \id \otimes \eval^{(2)}_{M,M}$}}
      M^{\vee} \otimes M^{\vee},
  \end{align}
  where $\otimes = \otimes_A$. Thus, for $p, q = 1, \dotsc, N$, we compute
  \begin{align*}
    m^q \otimes_A m^p
    & \xmapsto{\makebox[3em]{$\scriptstyle \eqref{eq:B-omega'-FRT-relation-computation-1}$}}
      m^s \otimes_A m^r \otimes_A m_r \otimes_A m_s \otimes_A m^q \otimes_A m^p \\
    & \xmapsto{\makebox[3em]{$\scriptstyle \eqref{eq:B-omega'-FRT-relation-computation-2}$}}
      m^s \otimes_A m^r \otimes_A w_{r s}^{i j} m_i \otimes_A m_j \otimes_A m^q \otimes_A m^p \\
    & \xmapsto{\makebox[3em]{$\scriptstyle \eqref{eq:B-omega'-FRT-relation-computation-3}$}}
      m^s \otimes_A m^r w_{r s}^{i j} \langle m^p, m_i \langle m^q, m_j \rangle \rangle
      \mathop{=}^{\eqref{eq:B-omega-FRT-relation-c-mathring}}
      m^s \otimes_A m^r \overline{w}_{r s}^{p q},
  \end{align*}
  where the Einstein convention is used to suppress sums over $i$, $j$, $r$ and $s$. Hence,
  \begin{align*}
    & \Rel^{\Sigma}_{\sigma}((m_i \otimes_A m_j) \otimes_k (m^q \otimes_A m^p)) \\
    & = c(m_i \otimes_A m_j) \otimes_k (m^q \otimes_A m^p)
      - (m_i \otimes_A m_j) \otimes_k c^{\dagger}(m^q \otimes_A m^p) \\
    & = (w_{i j}^{r s} m_r \otimes_A m_s) \otimes_k (m^q \otimes_A m_p)
      - (m_i \otimes_A m_j) \otimes_k (m^s \otimes_A m^r \overline{w}_{r s}^{p q}) \\
    & = w_{i j}^{r s}\rhd T^p_r T^q_s
      - T_i^r T_j^s \lhd \overline{w}^{p q}_{r s}
  \end{align*}
  for all $i, j, p, q = 1, \cdots, N$, where the Einstein convention is used again.
  Since the sets $\{ m_i \otimes_A m_j \}_{i, j}$ and $\{ m^q \otimes_A m^p \}_{p, q}$ generate $M \otimes_A M$ and $M^{\vee} \otimes_A M^{\vee}$ as a left and a right $A$-module, respectively, the image of the $A^{\env}$-bimodule map $\Rel^{\Sigma}_{\sigma}$ is generated by the elements \eqref{eq:B-omega-FRT-relation}.

  By the above argument, the ideal $I$ is generated by the elements~\eqref{eq:B-omega-FRT-relation}. Hence the $A^{\env}$-ring $B(M, c)$ is identical to $B^{\Sigma}_{\omega}$. It is easy to see that the comultiplication and the counit of $B(M, c)$ are given as stated on the generators $T_{i}^{j}$.
  By Theorem~\ref{thm:B-omega-lax-universal-R}, the universal R-form form $\mathbf{r}$ is given by
  \begin{align*}
    \mathbf{r}\left( T_i^r, T_j^s \right)
    & = \eval_{M,M}^{(2)}(c(m_j \otimes_A m_i) \otimes_A m^s \otimes_A m^r) \\
    & = \eval_{M,M}^{(2)}(w_{j i}^{p q} m_p \otimes_A m_q \otimes_A m^s \otimes_A m^r) \\
    & = \langle m^r, w_{j i}^{p q} m_p \langle m^s, m_q \rangle \rangle
      = \overline{w}^{r s}_{j i}
  \end{align*}
  for $i, j, r, s = 1, \cdots, N$. The uniqueness part of this theorem follows from that $B(M, c)$ is generated by $T_{i}^j$'s as an $A^{\env}$-ring.
\end{proof}

\subsection{Rigid extension of the category of braids}

In the previous subsection, we have constructed a left bialgebroid $B(M, c)$ from a braided object $(M, c)$ in ${}_A \Mod_A$ such that $M$ is right rigid. A `Hopf version' of the bialgebroid $B(M, c)$ is obtained in a similar way. Namely, we assume that the braided object $(M, c)$ is `dualizable' in the sense of \cite[Definition 7.3]{MR1250465}. Then the monoidal functor $\omega: \mathfrak{B} \to {}_A \Mod_A$ in the above lifts to the `rigid extension' of $\mathfrak{B}$. By Theorem~\ref{thm:FRT-bimonad}, the bialgebroid arising from the lift of $\omega$ is in fact a coquasitriangular Hopf algebroid.

We first recall the definition of dualizability from \cite{MR1250465}. Let $\mathcal{C}$ be a monoidal category, and let $V$ and $X$ be objects of $\mathcal{C}$. Suppose that $V$ is right rigid. Given a morphism $f: V \otimes X \to X \otimes V$ in $\mathcal{C}$, we define $f^{\flat} : X \otimes V^{\vee} \to V^{\vee} \otimes X$ by
\begin{equation*}
  f^{\flat} = (\id_{V^{\vee}} \otimes \id_X \otimes \eval_V)
  \circ (\id_{V^{\vee}} \otimes f \otimes \id_{V^{\vee}})
  \circ (\coev_V \otimes \id_X \otimes \id_{V^{\vee}}).
\end{equation*}

\begin{definition}[\cite{MR1250465}]
  A braided object $(V, c)$ in $\mathcal{C}$ is said to be {\em dualizable} if $V$ is right rigid and the morphisms $c^{\flat}$ and $(c^{-1})^{\flat}$ are invertible.
\end{definition}

Since $\mathfrak{B}$ is a free braided monoidal category on a single object, a `rigid extension' of $\mathfrak{B}$ shall mean a free braided rigid monoidal category on a single object. Such a category has been constructed in \cite{MR1154897} in a geometric way. For our purpose, the following algebraic construction is rather useful:

\begin{definition}
  The monoidal category $\widetilde{\mathfrak{B}}$ is generated by the tensor scheme
  \begin{equation*}
    \Sigma_0 = \{ x_{+}, x_{-} \},
    \quad \Sigma_1 = \{ \varepsilon, \eta, \sigma_{a,b}, \overline{\sigma}_{a,b} \, (a, b \in \{ +, - \}) \},
  \end{equation*}
  where $\varepsilon: x_{+} \otimes x_{-} \to \unitobj$, $\eta: \unitobj \to x_{-} \otimes x_{+}$,
  $\sigma_{a,b}: x_a \otimes x_b \to x_b \otimes x_a$ and
  $\overline{\sigma}_{a,b}:  x_b \otimes x_a \to x_a \otimes x_b$, subject to the relations
  \begin{gather}
    \allowdisplaybreaks
    \label{eq:B-tilde-rel-YBE}
    \begin{aligned}
      (\sigma_{+, +} \otimes \id_{+}) & (\id_{+} \otimes \sigma_{+, +}) (\sigma_{+, +} \otimes \id_{+}) \\
      = \, & (\id_{+} \otimes \sigma_{+, +}) (\sigma_{+, +} \otimes \id_{+}) (\id_{+} \otimes \sigma_{+, +}),
    \end{aligned} \\
    \label{eq:B-tilde-rel-invertible}
    \overline{\sigma}_{a b} \sigma_{a b}
    = \id_{a} \otimes \id_{b}
    =  \sigma_{b a} \overline{\sigma}_{b a}, \\
    \label{eq:B-tilde-rel-rigid}
    (\varepsilon \otimes \id_{+}) (\id_{+} \otimes \eta) = \id_{+},
    \quad (\id_{-} \otimes \varepsilon) (\eta \otimes \id_{-}) = \id_{-}, \\
    \label{eq:B-tilde-rel-4}
    (\id_a \otimes \varepsilon) (\sigma_{+,a} \otimes \id_{-})
    = (\varepsilon \otimes \id_a) (\id_{+} \otimes \overline{\sigma}_{-,a}), \\
    \label{eq:B-tilde-rel-5}
    (\id_a \otimes \varepsilon) (\overline{\sigma}_{a,+} \otimes \id_{-})
    = (\varepsilon \otimes \id_a) (\id_{+} \otimes \sigma_{a,-}) \phantom{,}
  \end{gather}
  for $a, b \in \{ +, - \}$, where $\id_a = \id_{x_a}$.
\end{definition}

Equation~\eqref{eq:B-tilde-rel-rigid} says that $(x_{-}, \varepsilon, \eta)$ is a right dual object of $x_{+}$.
By expressing $\sigma_{a,b}$ and $\overline{\sigma}_{a,b}$ by a positive and a negative crossing in string diagrams, respectively, equations \eqref{eq:B-tilde-rel-4} and~\eqref{eq:B-tilde-rel-5} are graphically interpreted as follows:
\begin{equation*}
  \PIC{rigid-braid-01}
  \mathop{=}^{\eqref{eq:B-tilde-rel-4}}
  \PIC{rigid-braid-02} \qquad \qquad
  \PIC{rigid-braid-03}
  \mathop{=}^{\eqref{eq:B-tilde-rel-5}}
  \PIC{rigid-braid-04}
\end{equation*}
In view of (1) and (3) of the following lemma, we may say that $\widetilde{\mathfrak{B}}$ is a free braided rigid monoidal category on a single object, and thus call it a `rigid extension' of the monoidal category $\mathfrak{B}$ of braids.

\begin{lemma}
  \label{lem:free-rigid-br}
  The monoidal category $\widetilde{\mathfrak{B}}$ enjoys the following property:
  \begin{enumerate}
  \item $\widetilde{\mathfrak{B}}$ is a braided rigid monoidal category.
  \item Given a dualizable braided object $(V, c)$ in a monoidal category $\mathcal{C}$, there is a strong monoidal functor $F: \widetilde{\mathfrak{B}} \to \mathcal{C}$ such that
    \begin{gather*}
      F(x_{+}) = V,
      \quad F(x_{-}) = V^{\vee},
      \quad F(\varepsilon) = \eval_V,
      \quad F(\eta) = \coev_{V}, \\
      F(\sigma_{+,+}) = c, \quad
      F(\sigma_{-,+}) = (c^{\flat})^{-1}, \quad
      F(\sigma_{+,-}) = (c^{-1})^{\flat}, \quad
      F(\sigma_{-,-}) = c^{\flat\flat}, \\
      F(\overline{\sigma}_{a,b}) = F(\sigma_{a,b})^{-1}
      \quad (a, b \in \{ +, - \}),
    \end{gather*}
    and such a strong monoidal functor is unique up to isomorphisms.
  \item For an object $V$ of a braided rigid monoidal category $\mathcal{C}$, there exists a braided strong monoidal functor $F_V: \widetilde{\mathfrak{B}} \to \mathcal{C}$ such that $F_{V}(x_{+}) = V$. Furthermore, such a braided strong monoidal functor $F_V$ is unique up to isomorphism.
  \end{enumerate}
\end{lemma}
\begin{proof}
  (1) By applying the map $f \mapsto (\id_{-} \otimes f) \circ (\eta \otimes \id_{-})$ to \eqref{eq:B-tilde-rel-4} and \eqref{eq:B-tilde-rel-5}, we obtain the following equations:
  \begin{equation}
    \label{eq:B-tilde-rel-6}
    (\sigma_{+, a})^{\flat} = \overline{\sigma}_{-, a}, \quad
    (\overline{\sigma}_{a, +})^{\flat} = \sigma_{a, -}  \quad (a \in \{ +, - \}).
  \end{equation}
  Equations \eqref{eq:B-tilde-rel-YBE}, \eqref{eq:B-tilde-rel-invertible} and \eqref{eq:B-tilde-rel-6} imply that $(x_{+}, \sigma_{+, +})$ is a dualizable braided object. Thus, by \cite[Proposition 7.5]{MR1250465}, we have
  \begin{equation}
    \label{eq:free-rigid-monoidal-pf-1}
    (\id_{r} \otimes \sigma_{p,q}) (\sigma_{p,r} \otimes \id_{q}) (\id_{p} \otimes \sigma_{q,r})
    = (\sigma_{q,r} \otimes \id_{p}) (\id_{q} \otimes \sigma_{p,r}) (\sigma_{p,q} \otimes \id_{r})
  \end{equation}
  for all $p, q, r \in \{ +, - \}$.

  By~\eqref{eq:free-rigid-monoidal-pf-1} and \cite[Proposition 2.2]{MR1250465} applied to the discrete category over the set $\{ +, - \}$, we obtain a family of isomorphisms $c_{x,y}: x \otimes y \to y \otimes x$ in $\widetilde{\mathfrak{B}}$ such that the equations \eqref{eq:def-lax-br-1}, \eqref{eq:def-lax-br-2}, \eqref{eq:def-lax-br-3} and $c_{a,b} = \sigma_{a,b}$ hold for all $a, b \in \{ +, - \}$. We show that $c$ is a braiding of $\widetilde{\mathfrak{B}}$. For this, it is enough to check that the equations
  \begin{equation}
    \label{eq:free-rigid-monoidal-pf-2}
    (\id_a \otimes f) c_{\src(f), a}
    = c_{\tgt(f), a} (f \otimes \id_a), \quad
     c_{a, \tgt(f)} (\id_a \otimes f)
     = (f \otimes \id_a) c_{a, \src(f)}
  \end{equation}
  hold for all $f \in \Sigma_1$ and $a \in \{ +, - \}$. If $f = \sigma_{p, q}$ or $f = \overline{\sigma}_{p, q}$ for some $p, q \in \{ +, - \}$, then \eqref{eq:free-rigid-monoidal-pf-2} follows from \eqref{eq:free-rigid-monoidal-pf-1}. For $f = \varepsilon$, we have
  \begin{align*}
    (\id_a \otimes \varepsilon) c_{x_{+} \otimes x_{-}, a}
    = (\id_a \otimes \varepsilon) (\sigma_{+,a} \otimes \id_{-}) (\id_{+} \otimes \sigma_{-,a})
    \mathop{=}^{\eqref{eq:B-tilde-rel-4}}
    c_{a, \unitobj} (\varepsilon \otimes \id_a)
  \end{align*}
  and, in a similar way, prove $c_{a, \unitobj} (\id_a \otimes \varepsilon) = (\varepsilon \otimes \id_a) c_{a,x_{-} \otimes x_{+}}$. For $f = \eta$, we remark that the equations
  \begin{gather*}
    (\eta \otimes \id_a) (\id_{-} \otimes \sigma_{+,a})
    = (\id_a \otimes \eta) (\widetilde{\sigma}_{-,a} \otimes \id_{+}), \\
    (\eta \otimes \id_a) (\id_{-} \otimes \widetilde{\sigma}_{a,+})
    = (\id_a \otimes \eta) (\sigma_{a,-} \otimes \id_{+}) \phantom{,}
  \end{gather*}
  are obtained by applying the map $f \mapsto (\id_{-} \otimes f \otimes \id_{+}) (\eta \otimes \id_a \otimes \eta)$ to the both sides of \eqref{eq:B-tilde-rel-4} and \eqref{eq:B-tilde-rel-5}, respectively. By using these equations, we prove \eqref{eq:free-rigid-monoidal-pf-2} for $f = \eta$ in a similar way as the case where $f = \varepsilon$.

  Therefore $c$ is a braiding of $\widetilde{\mathfrak{B}}$. By the general fact that an object of a braided monoidal category is left rigid if and only if it is right rigid, both $x_{+}$ and $x_{-}$ are rigid. Since the tensor product of rigid objects is again rigid, we conclude that every object of $\widetilde{\mathfrak{B}}$ is rigid. The proof of (1) is done.

  (2) This follows from the universal property of $\widetilde{\mathfrak{B}}$.

  (3) Let $\mathcal{C}$ be a braided rigid monoidal category with braiding $c'$, and let $V$ be an object of $\mathcal{C}$. Then the pair $(V, t)$, where $t = c'_{V,V}$, is a dualizable braided object in $\mathcal{C}$ \cite[Proposition 7.4]{MR1250465}. The rest follows from Part (2).
\end{proof}

\subsection{The FRT construction: a Hopf version}

Let $(M, c)$ be a dualizable braided object in ${}_A \Mod_A$.
Then a strong monoidal functor $\omega : \mathfrak{B} \to {}_A \Mod_A$ is constructed from $(M, c)$ by the way of Lemma~\ref{lem:free-rigid-br} (2).

\begin{definition}
  \label{def:FRT-Hopf}
  We call the coquasitriangular Hopf algebroid $B_{\omega}$ with $\omega$ as above {\em the FRT Hopf algebroid} associated to $(M, c)$.
\end{definition}

We give a presentation of the FRT Hopf algebroid associated to $(M, c)$.
For this purpose, we first prove the following technical lemma:

\begin{lemma}
  There are a natural number $N$ and elements
  \begin{equation*}
    m_i \in M, \quad m^i \in M^{\vee}
    \quad \text{and} \quad \hat{m}_i \in M^{\vee\vee}
    \quad (i = 1, \dotsc, N)
  \end{equation*}
  such that the following equations hold:
  \begin{equation}
    \label{eq:FRT-special-dual-basis}
    \coev_{M}(1) = \sum_{i = 1}^N m^i \otimes_A m_i
    \quad \text{and} \quad
    \coev_{M^{\vee}}(1) = \sum_{i = 1}^N \hat{m}_i \otimes_A m^i.
  \end{equation}
\end{lemma}
\begin{proof}
  There are elements $n_i \in M$, $n^i \in M^{\vee}$, $\phi_j \in M^{\vee}$ and $\phi^j \in M^{\vee\vee}$ ($i = 1, \dotsc, r$, $j = 1, \dotsc, s$) such that $\coev_M(1) = \sum_{i = 1}^r n^i \otimes_A n_i$ and $\coev_{M^{\vee}}(1) = \sum_{i = 1}^r \phi^i \otimes_A \phi_i$. We set $N = r + s$ and define $m_i$, $m^i$ and $\hat{m}_i$ ($i = 1, \dotsc, N$) to be the first, the second and the third entry of the $i$-th row of the matrix
  \begin{equation*}
    \begin{pmatrix}
      n_1 & n_2 & \cdots & n_s & 0 & 0 & \dotsb & 0 \\
      n^1 & n^2 & \cdots & n^s & \phi_1 & \phi_2 & \cdots & \phi_r \\
      0 & 0 & \dotsc & 0 & \phi^1 & \phi^2 & \cdots & \phi^r
    \end{pmatrix},
  \end{equation*}
  respectively. Then these elements satisfy \eqref{eq:FRT-special-dual-basis}.
\end{proof}

We fix elements $\{ m_i \}_{i = 1}^N$, $\{ m^i \}_{i = 1}^N$ and $\{ \hat{m}_i \}_{i = 1}^N$ satisfying~\eqref{eq:FRT-special-dual-basis} (they are only required to satisfy \eqref{eq:FRT-special-dual-basis} and need not to be taken in the way of the proof of the above lemma).
For $i, j = 1, \cdots, N$, we set
\begin{equation}
  \label{eq:B-omega-FRT-Hopf-generators}
  T_{i}^{j} = m_i \otimes_k m^j \in M \otimes_k M^{\vee}
  \quad \text{and} \quad
  \overline{T}{}_{i}^j = m^j \otimes_k \hat{m}_i \in M^{\vee} \otimes_k M^{\vee\vee}.
\end{equation}
We also make the $k$-module $G := (M \otimes_k M^{\vee}) \oplus (M^{\vee} \otimes_k M^{\vee\vee})$ an $A^{\env}$-bimodule in a similar way as \eqref{eq:B-omega-FRT-generators-action}. Now we prove:

\begin{theorem}
  \label{thm:FRT-coqt-Hopf}
  The FRT Hopf algebroid of Definition~\ref{def:FRT-Hopf} is isomorphic to the Hopf algebroid $H(M, c)$ defined as follows: As an $A^{\env}$-ring, $H(M, c)$ is the quotient of the tensor $A^{\env}$-ring $T_{A^{\env}}(G)$ by the ideal generated by the elements \eqref{eq:B-omega-FRT-relation} and
  \begin{align}
    \label{eq:B-omega-FRT-Hopf-relation}
    \beta_{i j} - \sum_{r = 1}^N T_{i}^{r} \overline{T}{}_{r}^{j},
    \quad
    \overline{\beta}_{i j} - \sum_{r = 1}^N \overline{T}{}_{i}^{r} T_{r}^{j},
    \quad
    (i, j = 1, \dotsc, N),
  \end{align}
  where $\beta_{i j}$ and $\overline{\beta}_{i j}$ are elements of $A^{\env} =T_{A^{\env}}^{0}(G)$ defined by
  \begin{equation*}
    \beta_{i j} = \langle m^j, m_i \rangle \otimes_k 1_{A^{\op}}
    \quad \text{and} \quad
    \overline{\beta}_{i j} = 1_{A} \otimes_k \langle \hat{m}_i, m_j \rangle.
  \end{equation*}
  The comultiplication and the counit are given by
  \begin{gather*}
    \Delta(T_{i}^{j}) = \sum_{r = 1}^N T_{i}^{r} \otimes_A T_{r}^{j},
    \quad \pi(T_{i}^{j}) = \langle m^j, m_i \rangle, \\
    \Delta(\overline{T}_{i}^{j}) = \sum_{r = 1}^N \overline{T}_{i}^{r} \otimes_A \overline{T}_{r}^{j},
    \quad \pi(\overline{T}_{i}^{j}) = \langle \hat{m}_i, m^j \rangle.
  \end{gather*}
  As $H(M, c)$ is isomorphic to the FRT Hopf algebroid associated to $(M, c)$, it has a universal R-form $\mathbf{r}$. On the generators, the form $\mathbf{r}$ is given by
  \begin{align*}
    \mathbf{r}(T_{r}^{i}, T_{s}^{j})
    & = \eval^{(2)}_{M,M}(c(m_j \otimes_A m_i) \otimes_A m^s \otimes_A m^r), \\
    \mathbf{r}(T_{r}^{i}, \overline{T}_{s}^{j})
    & = \eval^{(2)}_{M,M^{\vee}}((c^{\flat})^{-1}(m^j \otimes_A m_i) \otimes_A \hat{m}_s \otimes_A m^r), \\
    \mathbf{r}(\overline{T}_{r}^{i}, T_{s}^{j})
    & = \eval^{(2)}_{M^{\vee},M}((c^{-1})^{\flat}(m_j \otimes_A m^i) \otimes_A m^s \otimes_A \hat{m}_r), \\
    \mathbf{r}(\overline{T}_{r}^{i}, \overline{T}_{s}^{j})
    & = \eval^{(2)}_{M^{\vee},M^{\vee}}(c^{\flat\flat}(m^j \otimes_A m^i) \otimes_A \hat{m}_s \otimes_A \hat{m}_r)
  \end{align*}
  for $i, j, r, s = 1, \cdots, N$.
\end{theorem}
\begin{proof}
  We show that $H(M, c)$ is identical to $B_{\omega}^{\Sigma}$.
  Given a morphism $f$ in $\mathfrak{B}$, we denote by $I(f)$ the ideal of $T_{A^{\env}}(G)$ generated by $\Img(\Rel^{\Sigma}_{f})$.
  By the construction, $B^{\Sigma}_{\omega}$ is a quotient of $T_{A^{\env}}(G)$ by the ideal
  \begin{equation*}
    I := I(\sigma_{+,+})
    + I(\sigma_{+,-})
    + I(\sigma_{-,+})
    + I(\sigma_{-,-})
    + I(\varepsilon) + I(\eta).
  \end{equation*}
  Now we shall find generators of this ideal. The process is very similar to the case of Hopf algebras \cite[Section 2]{MR1800719}.
  First, by Lemma~\ref{lem:B-omega-gen-rel-2}, we see that the ideals $I(\sigma_{+,-})$, $I(\sigma_{-,+})$ and $I(\sigma_{-,-})$ are contained in the sum of $I(\sigma_{+,+})$, $I(\varepsilon)$ and $I(\eta)$.
  By the same computation as in the proof of Theorem~\ref{thm:FRT-coqt-bialgebroid}, the ideal $I(\sigma_{+,+})$ is generated by \eqref{eq:B-omega-FRT-relation}. By the construction of $\omega$, we have $\omega(\varepsilon) = \eval_M$ and
  \begin{equation*}
    (\check{\omega}^{(2)}_{x_{-}, x_{+}})^{-1} \circ \omega(\varepsilon)^{\vee} \circ \check{\omega}_0 = \coev_{M^{\vee}}.
  \end{equation*}
  Hence, identifying $\widetilde{E}_{x_{+}} \otimes_{A^{\env}} \widetilde{E}_{x_{-}}$ with $(\omega(x_{+}) \otimes_A \omega(x_{-})) \otimes_k (\omega(x_{-})^{\vee} \otimes_A \omega(x_{+})^{\vee})$ by the isomorphism \eqref{eq:widetilde-phi-iso-1}, we have
  \begin{gather*}
    \Rel^{\Sigma}_{\varepsilon}((m_i \otimes_A m^j) \otimes_k 1_A)
    = \eval_M(m_i \otimes_A m^j) \otimes_k 1 - (m_i \otimes_A m^j) \otimes_k \coev_{M^{\vee}}(1) \\
    \mathop{=}^{\eqref{eq:FRT-special-dual-basis}}
    \beta_{i j} - \sum_{r = 1}^N (m_i \otimes_A m^j) \otimes_k (\hat{m}_r \otimes_A m^r)
    \mathop{=}^{\eqref{eq:B-omega-FRT-Hopf-generators}}
    \beta_{i j} - \sum_{r = 1}^N T_{i}^{r} \overline{T}{}_{r}^{j}
  \end{gather*}
  for $i, j = 1, \cdots, N$. Similarly, since $\omega(\eta) = \coev_M$ and
  \begin{equation*}
    \check{\omega}_0^{-1}
    \circ \omega(\eta)^{\vee}
    \circ (\check{\omega}^{(2)}_{x_{+}, x_{-}})^{-1} = \eval_{M^{\vee}},
  \end{equation*}
  we have
  \begin{equation*}
    \Rel^{\Sigma}_{\eta}(1_A \otimes_k m^j \otimes_A \hat{m}_i)
    = \sum_{r = 1}^N \overline{T}{}_{i}^{r} T_{r}^{j} - \overline{\beta}_{i j}
  \end{equation*}
  for $i, j = 1, \dotsc, N$. Thus the elements \eqref{eq:B-omega-FRT-Hopf-relation} generate $I(\varepsilon) + I(\eta)$. In conclusion, the ideal $I$ in concern is generated by the elements \eqref{eq:B-omega-FRT-relation} and \eqref{eq:B-omega-FRT-Hopf-relation} and therefore the $A^{\env}$-ring $H(M, c)$ agrees with the $A^{\env}$-ring $B^{\Sigma}_{\omega}$.

  The $A^{\env}$-ring $H(M, c)$ has a structure of a coquasitriangular left Hopf algebroid over $A$ since so does $B^{\Sigma}_{\omega}$. The proof is completed by writing down the structure maps of $H(M, c)$ for the generators by the same way as Theorem~\ref{thm:FRT-coqt-bialgebroid}. 
\end{proof}

\begin{remark}
  We close this paper by discussing the relation between our presentation of the FRT Hopf algebroid and Hayashi's construction \cite{MR1623965} of face algebras (which can be regarded as a class of bialgebroids by \cite{MR1615924}).
  Let $\Lambda$ be a non-empty finite set.
  A {\em quiver} over the set $\Lambda$ is a triple $Q = (Q, \src, \tgt)$ consisting of a set $Q$ and two maps $\src$ and $\tgt$ from $Q$ to $\Lambda$. Given two quivers $Q$ and $Q'$, we define
  \begin{equation*}
    Q \times_{\Lambda} Q' = \{ (q, q') \in Q \times Q' \mid \tgt(q) = \src(q') \}.
  \end{equation*}
  The set $\Lambda$ is a quiver over $\Lambda$ by $\src = \tgt = \id_{\Lambda}$. We identify
  \begin{equation*}
    Q \times_{\Lambda} \Lambda = Q = \Lambda \times_{\Lambda} Q,
    \quad (q, \tgt(q)) \leftrightarrow q \leftrightarrow (\src(q), q).
  \end{equation*}
  The set $Q \times_{\Lambda} Q'$ is a quiver over $\Lambda$ by $\src((q,q')) = \src(q)$ and $\tgt((q,q')) = \tgt(q')$. For a quiver $Q$ over $\Lambda$ and a non-negative integer $n$, we define
  \begin{equation*}
    Q^{(0)} = \Lambda
    \quad \text{and} \quad
    Q^{(n)} = Q \times_{\Lambda} Q^{(n-1)} \quad (n \ge 1).
  \end{equation*}
  For $p = (p_1, \cdots, p_m) \in Q^{(m)}$ and $q = (q_1, \cdots, q_n) \in Q^{(n)}$, we set
  \begin{equation*}
    p \cdot q = (p_1, \cdots, p_m, q_1, \cdots, q_n) \in Q^{(m + n)}
  \end{equation*}
  when $\tgt(p) = \src(q)$. 

  We consider the algebra $A = \mathrm{Map}(\Lambda, k)$ of $k$-valued functions on $\Lambda$. If $Q$ is a quiver $Q$ over $\Lambda$, then the free $k$-module $M_Q$ over $Q$ is an $A$-bimodule by
  \begin{equation*}
    f \cdot m_q \cdot f' = f(\src(q)) f'(\tgt(q)) m_q
    \quad (f, f' \in A, q \in Q),
  \end{equation*}
  where $m_q$ is the element $q \in Q$ regarded as an element of $M_Q$.
  For $\lambda \in \Lambda$, we denote by $e_{\lambda} \in A$ the characteristic function on the set $\{ \lambda \}$. Then we have
  \begin{equation*}
    m_p \otimes_A m_q
    = m_p e_{\tgt(p)} \otimes_A m_q
    = m_p \otimes_A e_{\tgt(p)} m_q
    = \delta_{\tgt(p),\src(q)} m_p \otimes_A m_q
  \end{equation*}
  for all $p, q \in Q$, where $\delta$ denotes the Kronecker symbol. By this observation, we see that the $n$-fold tensor product of $M_Q$ over $A$ is a free $k$-module with basis
  \begin{equation*}
    \{ m_{q_1} \otimes_A \dotsb \otimes_A m_{q_n} \mid (q_1, \cdots, q_n) \in Q^{(n)} \}.
  \end{equation*}

  Now let $(M, w)$ be a braided object of ${}_A \Mod_A$ such that $M = M_Q$ for some finite non-empty quiver $Q$ over $\Lambda$. We say that a quadruple $(a,b,c,d) \in Q^4$ is a {\em face} if the equations $\src(a) = \src(c)$, $\tgt(a) = \src(b)$, $\tgt(c) = \src(d)$ and $\tgt(b) = \tgt(d)$ hold. Since $w$ is a morphism of $A$-bimodules, we have
  \begin{equation*}
    w(m_p \otimes_A m_q)
    = \sum_{r, s \in Q} w_{p q}^{r s} \, m_r \otimes_A m_s
    \quad (p, q \in Q)
  \end{equation*}
  for some $w_{p q}^{r s} \in k$ such that $w_{p q}^{r s} = 0$ unless $(p,q,r,s)$ is a face. Thus $(M, w)$ can be regarded as what is called a star-triangular face model in \cite[p.234]{MR1623965}.
  The closability \cite[p.235]{MR1623965} of $(M, w)$ is equivalent to the dualizability of $(M, w)$.
  If $(M, w)$ is dualizable, then its Lyubashenko double \cite[p.235]{MR1623965} is the braided object $(M \oplus M^{\vee}, \widetilde{w})$, where $\widetilde{w}$ is the extension of $w$ such that
  \begin{equation*}
    \widetilde{w}|_{M^{\vee} \otimes_A M} = (w^{\flat})^{-1},
    \quad \widetilde{w}|_{M \otimes_A M^{\vee}} = (w^{-1})^{\flat}
    \quad \text{and}
    \quad \widetilde{w}|_{M^{\vee} \otimes_A M^{\vee}} = w^{\flat\flat}.
  \end{equation*}

  We suppose that the braided object $(M, w)$ is dualizable.
  For $q \in Q$, we define $m^q \in M^{\vee}$ and $\hat{m}_q \in M^{\vee\vee}$ by $m^q(m_p) = \delta_{p,q} e_{\src(q)}$ and $\hat{m}_q(m^p) = \delta_{p,q} e_{\tgt(q)}$ ($p \in Q$), respectively. It is easy to see that the following equations hold:
  \begin{equation*}
    \coev_M(1_A) = \sum_{q \in Q} m^q \otimes_A m_q
    \quad \text{and} \quad
    \coev_{M^{\vee}}(1_A) = \sum_{q \in Q} \hat{m}_q \otimes_A m^q.
  \end{equation*}
  Now we set $G := (M \otimes_k M^{\vee}) \oplus (M^{\vee} \otimes_k M^{\vee\vee})$,
  \begin{equation*}
    T_{p}^{q} = m_p \otimes_k m^q \in M \otimes_k M^{\vee}
    \quad \text{and} \quad
    \overline{T}{}_{p}^{q} = m^p \otimes_k \hat{m}_q
    \in M^{\vee} \otimes_k M^{\vee\vee}
  \end{equation*}
  for $p, q \in Q$. Then, as an $A^{\env}$-ring, the Hopf algebroid $H(M, c)$ is the quotient of $T_{A^{\env}}(G)$ by the ideal generated by the elements
  \begin{gather}
    \label{eq:face-alg-rel-1}
    \sum_{r, s \in Q} w_{i j}^{r s} \, T^p_r T^q_s
    - \sum_{r, s \in Q} w^{p q}_{r s} \, T_i^r T_j^s, \\
    \label{eq:face-alg-rel-2}
    \delta_{i j} (e_{\src(i)} \otimes_k 1_{A^{\op}}) - \sum_{r \in Q} T_i^r \overline{T}{}_r^j,
    \quad \delta_{i j} (1_{A} \otimes_k e_{\tgt(i)}) - \sum_{r \in Q} \overline{T}{}_{i}^{r} T_{r}^{j}
  \end{gather}
  for $i, j, p, q \in Q$. We omit to write down the coring structure and the universal R-form of $H(M, w)$.

  To explain the relation between $H(M, w)$ and Hayashi's face algebra, we introduce the $k$-algebra $\widetilde{H}(Q)$ generated by the symbols $\eE{p}{q}$, $\bE{p}{q}$ ($p, q \in Q^{(m)}$, $m = 0, 1, \cdots$) subject to the relations
  \begin{gather*}
    \bE{\lambda}{\mu} = \eE{\lambda}{\mu}
    \quad (\lambda, \mu \in Q^{(0)} = \Lambda),
    \quad \sum_{\lambda, \mu \in \Lambda} \eE{\lambda}{\mu}
    = 1, \\
    \eE{p}{q} \cdot \eE{r}{s}
    = \delta_{\tgt(p),\src(r)} \delta_{\tgt(q),\src(s)} \eE{p \cdot r}{q \cdot s}, \\
    \bE{p}{q} \cdot \bE{r}{s}
    = \delta_{\src(p),\tgt(r)} \delta_{\src(q),\tgt(s)} \bE{r \cdot p}{s \cdot q}\phantom{,}
  \end{gather*}
  for $p, q \in Q^{(m)}$, $r, s \in Q^{(n)}$ ($m, n = 0, 1, \cdots$)\footnote{The elements $\eE{p}{q}$ and $\bE{p}{q}$ correspond to $e \begin{pmatrix} p \\ q \end{pmatrix}$ and $e \begin{pmatrix} \widetilde{p} \\ \widetilde{q} \end{pmatrix}$ of \cite[Proposition 5.5]{MR1623965}, respectively. In \cite{MR1623965}, for a path $p$ in $Q$, the symbol $\widetilde{p}$ is used to mean the path in the opposite quiver obtained from $p$ by reversing its direction. Thus we have $\src(\widetilde{p}) = \tgt(p)$ and $\tgt(\widetilde{p}) = \src(p)$. The relation $\bE{\lambda}{\mu} = \eE{\lambda}{\mu}$ is imposed because, for $\lambda \in \Lambda$, the `path $\widetilde{\lambda}$ of length 0' in the opposite quiver is identified with $\lambda$ in \cite{MR1623965}}.
  The algebra $\widetilde{H}(Q)$ can be identified with the $k$-algebra $T_{A^{\env}}(G)$ via the algebra map defined by
  \begin{equation*}
    \eE{p}{q} \mapsto T_{p_1}^{q_1} \cdots T_{p_n}^{q_n}
    \quad \text{and} \quad
    \bE{p}{q} \mapsto \overline{T}{}_{p_n}^{q_n} \cdots \overline{T}{}_{p_1}^{q_1}
  \end{equation*}
  for $p = (p_1, \cdots, p_n), q = (q_1, \cdots, q_n) \in Q^{(n)}$. Under this identification, the face algebra $\mathrm{Hc}(\mathfrak{A}(w))$ of \cite[Proposition 5.5]{MR1623965} is the quotient of $T_{A^{\env}}(G)$ by the ideal generated by \eqref{eq:face-alg-rel-1}, \eqref{eq:face-alg-rel-2} and some more elements. However, in actuality, the `some more elements' are redundant generators of the ideal since they belong to the ideals $I(\sigma_{+,-})$, $I(\sigma_{-,+})$ or $I(\sigma_{-,-})$ with notation of the proof of Theorem~\ref{thm:FRT-coqt-Hopf}.
  Thus $\mathrm{Hc}(\mathfrak{A}(w))$ is identified with $H(M, c)$ as an algebra.
  By looking the coring structure on the generators, we conclude that, as a left bialgebroid, $H(M, c)$ is identified with $\mathrm{Hc}(\mathfrak{A}(w))^{\mathrm{cop}}$.
\end{remark}

\def\cprime{$'$}

\end{document}